\documentclass[11pt,a4paper,reqno]{amsart}
\usepackage{amsmath}
\usepackage{amsfonts}
\usepackage{amssymb}
\usepackage{amsthm}
\usepackage{latexsym}

\usepackage{xy}
\xyoption{all}

\reversemarginpar
\newtheorem{thm}{Theorem}[section]
\newtheorem{prop}[thm]{Proposition}
\newtheorem{lemma}[thm]{Lemma}
\newtheorem{cor}[thm]{Corollary}

\newtheorem{definition}[thm]{Definition}
\newtheorem{rem}[thm]{Remark}

\newtheorem{ex}[thm]{Example}
\newtheorem{property}[thm]{Property}
\def\v{\mathfrak{v}}

\def\g{\mathfrak{g}}
\def\a{\mathfrak{a}}

\def\k{\mathfrak{k}}

\def\p{\mathfrak{p}}

\def\N{\mathbb{N}}
\def\Z{\mathbb{Z}}
\def\R{\mathbb{R}}
\def\C{\mathbb{C}}

\def\E{\mathcal{E}}

\def\E{\mathcal{E}}
\def\Fou{\mathcal{F}}

\def\H{\mathbb{H}}
\def\SS{\mathcal{S}}
\def\O{\mathcal{O}}

\def\P{\mathcal{P}}

\def\cM{\mathcal{M}}
\def\cI{\mathcal{I}}

\def\ann{\mathrm{ann}}
\def\ev{\mathrm{ev}}
\def\id{\mathrm{id}}
\newcommand{\End}[1]{\mathrm{End}({#1})}
\newcommand{\tr}[1]{\mathrm{tr}({#1})}

\def\Hom{\mathrm{Hom}}

\def\fv{{\mathfrak v}}
\def\E{\mathcal{E}}

\renewcommand{\theequation}{\thesection.\arabic{equation}}

\def\medno{\medbreak\noindent}
\def\text#1{\;\;\;\;{\rm \hbox{#1}}\;\;\;\;}
\def\qquad{\quad\quad}

\def\itema{\item[{\rm (a)}]}
\def\itemb{\item[{\rm (b)}]}
\def\itemc{\item[{\rm (c)}]}
\def\itemd{\item[{\rm (d)}]}

\def\msy#1{{\mathbb #1}}
\def\C{{\msy C}}
\def\N{{\msy N}}
\def\Z{{\msy Z}}
\def\R{{\msy R}}

\def\ga{\alpha}
\def\gb{\beta}
\def\gd{\delta}

\def\gf{\varphi}
\def\gg{\gamma}

\def\gl{\lambda}

\def\gs{\sigma}

\def\got#1{\mathfrak #1}
\def\fa{{\got a}}

\def\fg{{\got g}}

\def\fk{{\got k}}

\def\fm{{\got m}}
\def\fn{{\got n}}

\def\fp{{\got p}}

\def\fv{{\got v}}

\def\to{\rightarrow}
\def\Re{{\rm Re}\,}

\def\inp#1#2{\langle#1\,,\,#2\rangle}

\def\Ad{{\rm Ad}}
\def\Endo{{\rm End}}
\def\Hom{{\rm Hom}}

\def\ad{{\rm ad}}
\def\after{\,{\scriptstyle\circ}\,}

\def\pr{{\rm pr}}

\def\tr{{\rm tr}\,}
\def\ideal{\,\lhd\,}

\def\iC{{\scriptscriptstyle \C}}

\def\cD{{\mathcal D}}
\def\cE{{\mathcal E}}
\def\cF{{\mathcal F}}
\def\cG{{\mathcal G}}
\def\cH{{\mathcal H}}
\def\cI{{\mathcal I}}

\def\cL{{\mathcal L}}
\def\cM{{\mathcal M}}

\def\cO{{\mathcal O}}
\def\cP{{\mathcal P}}

\def\cS{{\mathcal S}}
\def\cT{{\mathcal T}}

\mathcode`:="603A
\def\col{\,:\,}

\def\too{\longrightarrow}

\def\Cartan{\theta}

\def\Cci{C_c^\infty}

\def\ev{{\rm ev}}

\def\dotvar{\,\cdot\,}

\def\embeds{\hookrightarrow}

\def\PW{{\rm PW}}

\def\Fou{\cF}

\def\Cartan{\theta}

\def\inj{{\mathrm i}}

\def\gL{{\Lambda}}
\def\faPdc{\fa_{P\iC}^*}

\def\PWpreP{{\rm PW}_P^{\rm pre}}

\def\fadc{{\fa_\iC^*}}
\def\PWPzeropre{{\rm PW}_{P_0}^{\rm pre}}
\def\bfi{{\mathbf i}}
\def\diffD{J}
\def\Cartan{\theta}

\def\fac{\fa_\iC}
\def\data{\Sigma}
\def\faPd{\fa_P^*}
\def\half{{\textstyle \frac12}}

\def\Vect{{\rm Vect}}
\def\FDVect{{\rm FDVect}}

\def\coleq{\;:\!=}
\def\dM{M^\wedge}
\def\dMP{M_P^\wedge}
\def\dMPds{M_{P, ds}^\wedge}

\def\hatMPds{\widehat M_{P,ds}}
\def\hatMPzero{\widehat M_{P_0}}
\def\dK{K^\wedge}

\def\hatotimes{\widehat\otimes}
\def\holspace{\cF}

\def\HFG{{\rm H\rm F}_G}
\def\HFGinfty{\HFG^\infty}
\def\HFGinftyadm{\HFG^{\infty, {\rm adm}}}
\def\HFGO{{}^\circ\HFG}
\def\HFGOinfty{{}^\circ\HFGinfty}
\def\HFGOinftyadm{{}^\circ\HFGinftyadm}
\def\HFL{{\rm H\rm F}_L}
\def\HFLinfty{\HFL^\infty}

\def\obj{{\rm obj}}
\def\rmi{{\rm i}}
\def\fg{{\mathfrak g}}
\def\HFgKadm{{\rm H\rm F}_{(\fg, K)}^{\,{\rm adm}}}

\def\gs{\sigma}
\def\Fou{{\mathfrak F}}

\def\FMod{{\rm F}{\rm M}}
\def\pre{{\rm pre}}
\def\bs{\backslash}

\numberwithin{equation}{section}

\title{A comparison of Paley--Wiener theorems for real reductive Lie groups}
\author{E.~P.~van den Ban,  S.~Souaifi}

\begin{document}
\maketitle
\tableofcontents
\section*{Introduction}
In this paper we make a detailed comparison between the Paley--Wiener
theorem of J. Arthur \cite{arthurpw},
and the one recently established by P.\ Delorme \cite{delormepw}.
         \par
Let $G$ be a real reductive Lie group of the Harish-Chandra class and let
$K$ be a maximal compact subgroup. Let $C_c^\infty(G, K)$ denote the
space of smooth
compactly supported functions on $G$ which behave finitely under
both left and right
translation by $K.$ The Paley--Wiener theorem of each of the above
mentioned authors describes the image of $C_c^\infty(G,K)$ under
Fourier transformation,
in terms of a so called Paley--Wiener space. In this paper we will show that the two
Paley--Wiener spaces are equal, without using the proof or the validity of any of the
associated Paley--Wiener theorems.
It thus follows that the two theorems are equivalent from an a priori point of view.
         \par
In order to be able to be more specific about the contents of this paper,
we shall first give a more detailed description of the two
Paley--Wiener theorems mentioned above.
         \par
Let $\fg = \fk \oplus \fp$ be a Cartan decomposition associated with
the maximal compact subgroup $K.$
Here and in the following we use the convention to denote Lie groups
by Roman capitals, and their
Lie algebras by the corresponding lower case German letters. Let $\fa$
be a maximal abelian subspace
of $\fp,$
and let $A = \exp \fa$ be the associated vectorial closed subgroup of $G.$
We denote by $\cP(A)$ the (finite) set of cuspidal parabolic subgroups of
$G$ containing $A.$
Each parabolic subgroup $P \in \cP(A)$ has a Langlands decomposition of the form
$P = M_P A_P N_P,$ with $A_P \subset A.$ Let
$\dMPds$ denote the set of (equivalence classes of)
discrete series representations of $M_P.$
For $(\xi,\cH) \in \dMPds$ and $\gl \in \faPdc,$ we denote by
$\pi_{\xi,\gl}= \pi_{P, \xi,\gl}$
the representation of $G$  induced from the representation
$\xi \otimes (\gl + \rho_P) \otimes 1$ of $P.$
The associated module of smooth vectors for $\pi_{\xi, \gl}$ has a realization on the space
$C^\infty(K\col \xi)$ consisting of smooth functions $\psi: \;K \to \cH_\xi,$
transforming according to the rule:
$$
\psi(mk) = \xi(m) f(k),\qquad\mbox{\rm for all}  \quad k \in K, m \in K\cap M_P.
$$
Accordingly, each function $f \in C^\infty_c(G)$ has an operator valued Fourier transform
$$
\hat f (P, \xi, \gl) \coleq \pi_{\xi,\gl}(f)
               = \int_G f(x)\, \pi_{\xi, \gl} (x) \; dx \in \Endo(C^\infty(K: \xi)).
$$
Moreover, the endomorphism $\hat f(P, \xi, \gl)$ depends holomorphically
on the variable $\gl \in \faPdc.$
If $f$ is bi-K-finite, then $\hat f(P, \xi, \gl)$ belongs to the space
$\cS(P: \xi)$ of bi-K-finite
elements of $\Endo(C^\infty(K: \xi)).$ Moreover, the holomorphic maps
$\gl \mapsto \hat f (P, \xi, \gl)$
 are non-zero for only finitely many of the pairs $(P, \xi).$ It follows
 that $\hat f(P)$ may be viewed as
 an element of the algebraic direct sum
\begin{equation}
\label{e: holomorphic functions for PW}
\oplus_{\xi}\;\;\cO(\faPdc) \otimes \cS(P:\xi).
\end{equation}
We define the pre-Paley--Wiener space $\PWpreP(G,K)$ as the space of elements
$(\gf(P,\xi)\mid \xi \in \dMPds)$ in (\ref{e: holomorphic functions for PW})
for which there exists a constant $R> 0$ and for every $n > 0$ a constant $C_n > 0$ such that
\begin{equation}
\label{e: PW estimate}
\| \gf (P, \xi, \gl) \| \leq C_n ( 1 + |\gl|)^{-n} e^{R |\Re \gl|}
\end{equation}
for all $\xi, \gl.$
\par
We can now describe the Paley--Wiener space involved in Arthur's theorem in \cite{arthurpw}.
Let $P_0$ be a fixed minimal parabolic subgroup in $\cP(A).$ Its Langlands decomposition
is of the form $P_0 = M A N_0,$ where $M$ is the centralizer of $A$ in $K.$
The Arthur Paley--Wiener space
\begin{equation}
\label{e: PW A}
\PW^A(G,K)
\end{equation}
is defined as the space of $\gf \in \PWPzeropre(G,K)$ satisfying all finite linear
relations of the form
\begin{equation}
\label{e: arthurs condition}
\sum_i \inp{ \gf(P_0,\xi_i, \gl_i; u_i)}{ \psi_i} = 0,
\end{equation}
with $\xi_i \in \dM,$ $\psi_i \in \cS(P_0: \xi_i)^*_{K \times K},$ $\gl_i \in \fadc$
and with
$u_i \in S(\fa^*)$ acting as differential operators in the $\gl$-variable
(see Section \ref{ss: derivation process} for notation), as soon as these
 relations are satisfied by all families of functions
$(\gl \mapsto \pi_{P_0, \xi, \gl}(x)\mid \xi\in \dM),$ for $x \in G.$
These are the so-called Arthur-Campoli relations.
In \cite{arthurpw}, Arthur defines a similar Paley--Wiener space involving all
minimal parabolic subgroups from $\cP(A).$ In \cite{BSpwspace} this space is
shown to be isomorphic to the one defined in (\ref{e: PW A}).
\par
Next, let us describe the Paley--Wiener space introduced by Delorme \cite{delormepw}.
The definition of this Paley--Wiener space involves the operation of taking successive
derivatives of a family $\pi_{\gl}$ of representations, depending holomorphically on
a parameter $\gl \in \faPdc,$ for some $P \in\cP(A).$ Such an operation is encoded
by a sequence $\eta =(\eta_1, \ldots, \eta_N)$ in $\faPdc,$ listing the directions
in which the derivatives should be taken successively.
The associated family $\pi^{(\eta)}_\gl$ is again a holomorphic family
of representations.
The operation of derivation also applies to a holomorphic family $\gl \mapsto \gf_\gl$ of
continuous endomorphisms of $\pi_\gl$ and then gives a holomorphic family $\gf^{(\eta)}$ of
endomorphisms of $\pi^{(\eta)}_\gl.$
Let $\cD$ be the set of all $4$-tuples $(P, \xi, \gl, \eta)$ with $P \in \cP(A),$
$\xi \in \dMPds,$ $\gl \in \faPdc$ and $\eta$ a finite sequence of linear functionals
from $\faPdc$ as above.
Given a datum $\delta = (P, \xi,\gl ,\eta)$ we define $\pi_\gd \coleq \pi_{P, \xi, \gl}^{(\eta)}.$
Moreover, given  $\gf \in \PWPzeropre(G,K)$ we define $\gf_\gd$ in a similar fashion.
Finally, given a sequence $\gd = (\gd_1, \ldots, \gd_N)$ of data from $\cD,$ we write
$\pi_{\gd}\coleq  \pi_{\gd_1} \oplus \cdots \oplus \pi_{\gd_N},$ and
$\gf_\gd\coleq  \gf_{\gd_1} \oplus \cdots \oplus \gf_{\gd_N}.$
\par
Delorme's Paley--Wiener space is defined as the space $\PW^D(G,K)$ of
functions $\gf \in \oplus_{P\in\cP(A)}\PWpreP(G,K)$ such that
\begin{enumerate}
\itema
for each finite sequence $\gd \in \cD^N$  the function $\gf_\gd$
preserves all invariant subspaces of $\pi_{\gd};$
\itemb
for any two finite sequences $\gd_1 \in \cD^{N_1}$ and
 $\gd_2 \in\cD^{N_2},$
and any two sequences of closed invariant subspaces
$U_j \subset V_j$
for $\pi_{\gd_j},$ $(j=1,2),$ the induced maps
$\bar \gf_{\gd_j} \in \Endo(V_j/U_j)$
are intertwined by all $G$-equivariant operators
$T: V_1/U_1 \to V_2/U_2.$
\end{enumerate}
\par
There is a natural map $\PW^D(G,K) \to \PWPzeropre(G,K),$ given by
\begin{equation}\label{e: restriction map PW}
(\gf(P, \xi)\mid P \in \cP(A),\xi \in \hatMPds) \mapsto
(\gf(P_0, \xi)\mid \xi \in \hatMPzero).
\end{equation}
In this paper we show that the map (\ref{e: restriction map PW})
 is a linear isomorphism
from $\PW^D(G,K)$
onto $\PW^A(G,K),$ see Theorem \ref{t: main thm}.
\par
To understand better the conditions involving the derivatives in the
 definition of the
Paley--Wiener spaces, we start, in
Section \ref{s: holomorphic families}, with the
study of holomorphic families and their derivatives. Instead of
focusing on
first order derivatives, we replace a holomorphic family by the
associated holomorphic section
in a suitable jet bundle. This idea also occurs in W.\ Casselman's paper
\cite[\S 9]{casselman}.
We reformulate it slightly, by using a suitable trivialization of the jet bundle.
Our construction starts with fixing a finite dimensional module $E$ for the ring $\cO_0$
of germs of holomorphic functions (of $\gl$) at zero. It then gives,
for $\pi_\gl$ a holomorphic family of
representations in a fixed complete space $V$ a new holomorphic family
$\pi_\gl^{(E)}$ in the space $E \otimes V.$
The differentiation procedure of Delorme turns out to be a
special case of this procedure, with $E$ a suitable module of dimension $2.$
In the same Section \ref{s: holomorphic families} we study how the functor
$\pi_\gl \mapsto \pi_\gl^{(E)}$ behaves with respect to analytic families of
intertwining operators and with respect to induction.
\par
In the next section, \ref{s: Arthur's PW}, we give an equivalent definition of
Arthur's Paley--Wiener space
by invoking the functors $\pi \mapsto \pi^{(E)},$ instead of the derivations given by
elements of $S(\fa^*)$ in the Arthur-Campoli relations.
These relations may then be reformulated as linear relations on differentiated
families of representations.
        \par
In Section \ref{s: Delorme's PW} we simplify the definition of ${\rm PW}^D(G,K).$
First of all, due to the intertwining relations in the definition, the space can be defined in terms
of just the minimal parabolic subgroup $P_0.$
Next, the intertwining conditions (b) turn out to be a consequence of the invariant subspace conditions (a).
        \par
In the final Section \ref{s: Hecke algebra},  we bring into play the Hecke
algebra $\H(G,K)$
consisting of all bi-K-finite distributions on $G$ supported by $K.$ The importance of this
algebra for representation theory is based on the fact that the category of
 Harish-Chandra
modules is isomorphic to the category of finitely generated admissible modules for this algebra.

A key lemma in this section is the following. For $(\pi,V)$ a
 Harish-Chandra module, let $\Endo(\pi)^\#$ denote
the space of $K\times K$-finite endomorphisms $\gf$  of $V$ with the property that for every positive
integer $n$ the Cartesian power $\gf^{\times n}$ preserves all invariant subspaces of $V^{\times n}.$
The mentioned key lemma asserts that
$$
\Endo(\pi)^{\#} = {\rm image}(\H(G,K))\subset \Endo(V).
$$
It follows from this lemma that the Paley--Wiener space $\PW^{D}(G,K)$ allows the following
description in terms of the Hecke algebra. For every finite dimensional $\cO_0$-module $E$,
 and all finite sets $\Xi \subset \dM$ and $\gL \subset \fadc,$ we define the representation
  $\pi_{E, \Xi, \gL}$ to be the direct sum of the representations $\pi_{P_0, \xi, \gl}^{(E)},$
  for $(\xi, \gl) \in \Xi\times \gL.$
Moreover, for $\gf \in \PWPzeropre(G,K),$ we define the endomorphism $\gf_{E, \Xi, \gL}$ of $\pi_{E, \Xi, \gL}$ by
taking a similar direct sum.
Then $\PW^{D}(G,K)$ maps isomorphically onto the space of $\gf \in \PWPzeropre(G,K)$ such
that for all $E, \Xi, \gL$ as above,
\begin{equation}
\label{e: gf in Hecke}
\gf_{E, \Xi, \gL} \in \pi_{E, \Xi, \gL}(\H(G,K)).
\end{equation}

On the other hand, it follows from its definition that Arthur's Paley--Wiener space $\PW^A(G,K)$
is equal to  the space of $\gf \in \PWPzeropre(G,K)$ such that for all $E, \Xi, \gL$ as above,
$
\gf_{E, \Xi, \gl}
$
is annihilated by the annihilator of $\pi_{E, \Xi, \gL}(\H(G,K))$ in the contragredient module.
Since this condition is equivalent to (\ref{e: gf in Hecke}), it thus follows that the map
(\ref{e: restriction map PW}) is a linear isomorphism onto $\PW^A(G,K)$ (Theorem \ref{t: main thm}).

Returning to the original formulation of Arthur's Paley--Wiener theorem, we finally wish to mention that the condition (\ref{e: arthurs condition}) may
be replaced by the condition that for all $\xi_i, u_i, \gl_i$ as in (\ref{e: arthurs condition}), there exists a $h \in \H(G,K)$ such that for all $i,$
$$
\gf(\xi_i, \gl_i;u_i) =
\pi_{P_0, \xi_i, \gl_i;u_i}(h).
$$
This characterization is given in Subsection
 \ref{s: another useful char of PW}  where it is used to derive, from Arthur's theorem,  the Paley-Wiener theorem for bi-K-invariant
functions, due to
S.\ Helgason \cite{helgasonpw} and R.\ Gangolli \cite{gangolli}.
\medno
{\bf Acknowledgments:\ } We thank Pierre Baumann for a helpful discussion,
which led to a simpler proof of Lemma \ref{l: dcomm}.
The second named author was partially supported by a grant of The Netherlands Organization for Scientific Research, NWO, under project number 613.000.213.

\section{Notation and preliminaries}
Throughout this paper, $G$ will be a real reductive Lie group in the
Harish-Chandra class and $K$ a maximal compact subgroup. Let $U(\g)$ be the universal
enveloping algebra of the complexification $\g_\iC$ of $\g.$ We denote by $X\mapsto X^\vee$ the anti-automorphism of $U(\fg)$ which on $\fg$ is given
by $X \mapsto - X.$
           \par
 In this paper, locally convex spaces will always
be assumed to be Hausdorff and defined over $\C.$
\par
For any continuous representation $(\tau,V)$ of $K$
(in a quasi-complete locally convex space) and any class $\gamma$ in the unitary dual
$\dK$ of $K$, the $K$-isotypic component of $(\tau,V)$ of type $\gamma$ is denoted
by $V_\gg.$ The associated $K$-equivariant projection onto $V_\gg$ is denoted by $P_\gg.$

For every finite subset $\theta$ of $\dK$, we put
$$
V_\theta\coleq \oplus_{\gamma\in\theta}\;V_\gamma \quad \text{and}\quad
P_\theta\coleq \oplus_{\gamma\in\theta}\;P_\gamma.
$$
          \par
For any continuous representation $(\pi,V)$ of $G$, with $V$ a quasi-complete locally convex
space, let $V^\infty$ and $V_K$ denote the vector subspaces of smooth
and $K$-finite elements of $V,$ respectively. The first one gives rise to a subrepresentation of $\pi$
and the second one to its underlying $(\g,K)$-module $(\pi,V_K).$          \par
We say that a (continuous) $G$-representation or a $(\g,K)$-module is admissible
if all its $K$-isotypic components are finite dimensional. A Harish-Chandra module is an
admissible $(\g,K)$-module which is finitely generated as a $U(\g)$-module.          \par
The space $C^\infty(G)$ of complex valued smooth functions on $G$ is equipped with the
left and right regular actions of $G;$ the subspace $C_c^\infty(G)$ of compactly supported functions
is invariant for these actions. The actions are continuous for the usual locally convex topologies
on these spaces and may be dualized by taking contragredients. Let $\E'(G)$ denote the space of
compactly supported distributions on $G,$ i.e., the topological linear dual of $C^\infty(G).$
          \par
Fix a (bi-invariant) Haar measure $dx$ on $G.$ Then the linear map
$$
C_c^\infty(G) \; \longrightarrow \;\E'(G),\quad f \;\longmapsto \; f\,dx
$$
is an injective intertwining operator for both $G$-action.
Accordingly, we will use this map to view $C_c^\infty(G)$ as a submodule of $\cE'(G).$
          \par
For any continuous representation $(\pi,V)$ and any $f\in C_c^\infty(G)$, let  $\pi(f)$ denote the endomorphism of $V$ defined by
$$\pi(f)v\coleq \int_Gf(x)\pi(x)v\,dx,\quad v\in V.$$
Then for all $v\in V^\infty$ and $\xi\in (V^\infty)^*$, the following equality holds
$$
\xi(\pi(f)v)=\langle f\,dx,\xi(\pi(\,\cdot\,)v)\rangle;
$$
the bracket on the right-hand side of the equation indicates the natural pairing between $\E'(G)$ and $C^\infty(G)$.
          \par
Let $(\pi_i,V_i)$, $i=1,2$, be two $(\g,K)$-modules. The space $\Hom(V_1,V_2)$ of
 (linear) homomorphisms from $V_1$ to $V_2$ is naturally endowed
with a $(\g\times\g,K\times K)$-module structure.
Indeed,
for any $T\in \Hom(V_1,V_2)$,
$$
\begin{array}{l}
(X_1,X_2)T=\pi_2(X_2)\circ T -  T \circ \pi_1(X_1),\quad X_1,X_2\in \g,\\
(k_1,k_2)T=\pi_2(k_2)\circ T\circ \pi_1(k_1^{-1}),\quad k_1,k_2\in K.
\end{array}
$$
Accordingly, the subspace $\Hom_{(\g,K)}(V_1,V_2)$ of $(\g,K)$-homomorphisms consists of
the elements of $\Hom(V_1,V_2)$ which are invariant under the diagonal action.
\begin{lemma}\label{ext1}
Let $(\pi,V)$ be an admissible representation of $G$ and let
$\theta_1$ and ${\theta_2}$ be finite subsets of $\dK.$ Then the linear map
$$
C_c^\infty(G) \; \longrightarrow \;\Hom(V_{\theta_1}, V_{\theta_2}),\qquad
f \; \longmapsto \; P_{{\theta_2}} \pi(f) P_{{\theta_1}}
$$
uniquely extends to a continuous linear map from $\E'(G)$ to $\Hom(V_{\theta_1}, V_{\theta_2}).$
\end{lemma}
\begin{proof}
Uniqueness of the extension follows by density of $C_c^\infty(G)\, dx$ in $\E'(G).$
Let $v\in V_{\theta_1}$ and $\xi\in V_{\theta_2}^*.$ Then by finite dimensionality of
$\Hom(V_{\theta_1}, V_{\theta_2})$ it suffices to show
that the linear map
$$
\mathcal{L}\,:\,C_c^\infty(G) \; \longrightarrow \; \C,\quad
f \; \longmapsto \;\xi(\pi(f)v)
$$
extends continuously to $\E'(G).$ Define the function $m \in C^\infty(G)$ by
 $m(x) = \xi(\pi(x)v).$ Then $\mathcal{L}(f) = \langle f\, dx , m\, \rangle$, for $f\in C_c^\infty(G).$
 Thus, $u \mapsto \langle u , m \rangle$ defines a continuous linear
extension of $\cL$ to $\E'(G).$
\end{proof}
We consider the convolution product $*$ on $C^\infty_c(G)$ given  by
$$
f * g\, (x) = \int_G f(y) g(y^{-1} x )\; dy,
$$
for $f ,g \in C^\infty_c(G)$ and $x \in G.$ It defines an algebra structure on $C_c^\infty(G).$
The  subspace $C_c^\infty(G,K)$ of left and right $K$-finite elements in $C_c^\infty(G)$ is closed under convolution, hence a subalgebra of $C_c^\infty(G).$
          \par
The convolution product has a unique extension
to a separately continuous bilinear map $\E'(G) \times \E'(G) \to \E'(G),$ denoted $(u,v) \mapsto u * v.$
This turns $\E'(G)$ into an algebra. It is readily seen that the subspace
$\E'(G,K)$ of left and right $K$-finite elements in $\E'(G)$ is closed under convolution,
hence a subalgebra. Likewise, the subspace $\E_K'(G)$ of distributions
with support in $K$ is a subalgebra, and so is the intersection
\begin{equation}
\label{e: Hecke algebra}
\H(G,K)\coleq \E_K'(G) \cap \E'(G,K).
\end{equation}
The latter is also called the Hecke algebra of the pair $(G,K)$ and is sometimes denoted by $\H$ for simplicity.
          \par
From Lemma \ref{ext1}, we obtain the continuous linear map
\begin{equation}
\label{e: pi on E'(G,K)}
\E'(G, K) \; \to \; \mathrm{End}(V_K)_{K\times K}, \quad
u \; \mapsto \; \pi(u)
\end{equation}
which intertwines the $(\g\times \g,K\times K)$-actions. Here the space on the right
is equipped with the weakest topology for which the $K \times K$-equivariant projections of finite
rank are continuous.
By application of Fubini's theorem we see that $\pi$ is a morphism on the convolution algebra $C_c^\infty(G,K),$ which is a dense
subalgebra of $ \E'(G,K).$

By separate continuity of $\ast$  and continuity of (\ref{e: pi on E'(G,K)}), it now follows that
the map (\ref{e: pi on E'(G,K)}) is a homomorphism of algebras.
          \par
Fix normalized Haar measure $dk$  on $K.$ Then each $\gf \in C(K)$ defines a distribution
$\gf\,dk $ in $\E'_K(G),$ given by
$$
\inp{\gf\, dk}{f} = \int_K \gf(k) f(k) \, dk.
$$
           \par
For a given representation  $\gamma\in\dK,$ we define the distribution  $\ga_\gg \in \H$
by $\ga_\gg = \dim(\gg) \chi_{\gg^\vee}\, dk,$ where $\chi_{\gg^\vee}$ denotes the
character of the contragredient $\gg^\vee$
of $\gg.$
Moreover, for $\theta \subset \dK$ a finite subset,
we define the element $\ga_\theta \in \H$ by
$$
\alpha_\theta\coleq \sum_{\gg\in \theta} \ga_\gg.
$$
The functions $\alpha_\theta,$ viewed as elements of $\H,$ will later be seen to
define an approximation of the identity in $\H.$
            \par
Let $\theta_1, \theta_2 \subset \dK$ be finite subsets.
We agree to write $\cE'(G,K)_{\theta_1\theta_2}$ for the space of distributions $\gf \in \cE'(G,K)$ satisfying
$$
\ga_{\theta_1} * \gf * \ga_{\theta_2} = \gf.
$$
Then $\cE'(G,K)_{\theta_1\theta_2}$ consists of the distributions in $\cE'(G,K)$ of left
$K$-type in $\theta_1$ and of right $K$-type in $\theta_2^\vee = \{\gg^\vee \mid \gg \in \theta_2\}.$
Similarly, we write
\begin{eqnarray*}
C^\infty_c(G,K)_{\theta_1\theta_2} & \coleq  & C_c^\infty(G,K)\; \cap\;  \cE'(G,K)_{\theta_1\theta_2},\;\;
\text{and}\\
\H_{\theta_1\theta_2} &\coleq  &\H \; \cap\; \cE'(G,K)_{\theta_1\theta_2}.
\end{eqnarray*}
          \par
so that, for all admissible $G$-representations $(\pi,V)$ and  $(\tau,U)$,
$$
\Hom(U, V)_{\theta_1\theta_2} \simeq \Hom(U_{\theta_2}, V_{ \theta_1}),
$$
naturally.
Viewing
$\Hom(U, V)_{K\times K} \simeq V_K \otimes (U^*)_K$ as a $(\fg\times \fg, K \times K)$-module  in a natural way,
we see that
$$
\Hom(U, V)_{\theta_1\theta_2} = \Hom(U, V)_{\theta_1\otimes \theta_2^\vee}.
$$
In particular,
it is readily seen that
$$
\pi(\cE'(G,K)_{\theta_1\theta_2}) \subset \End{V}_{\theta_1\theta_2} \simeq \End{V_K}_{\theta_1\theta_2}.
$$
Here we note that $\End{V_K}_{K \times K} \subset \End{V},$ naturally.
\begin{prop}\label{end-hecke}
For any admissible representation $(\pi,V)$ of $G$,
$$
\pi(C_c^\infty(G,K))=\pi(\E'(G,K))=\pi(\H).
$$
\end{prop}
\begin{proof}
We denote the three given subspaces of $\End{V_K}_{K\times K}$ by $E_\infty,$
$E_{\E'}$ and $E_\H$ respectively.
Let $\gb$ be the  $K$-equivariant bilinear form on $\End{V_K}_{K \times K}$ given by
$$
\gb(A, B) = \tr (A \after B).
$$
By admissibility of $\pi$ it follows that $\gb$ defines a non-degenerate pairing, which is perfect
when restricted to $\End{V_K}_{\theta\theta},$ for $\theta$ any finite subset of $\dK$.
Therefore, it suffices to show that the $\gb$-orthocomplements
$E_\infty^\perp,$ $E_{\cE'}^\perp$ and $\H^\perp$
are equal. Thus, let $T\in \End{V_K}_{K\times K}$. Then it suffices to show that the following
assertions are equivalent:
\begin{enumerate}
\item[(i)] $\tr{T\pi(x)}=0$, for all $x\in G$;
\item[(ii)]  $\tr{T\pi(f)}=0$, for all $f\in C_c^\infty(G,K)$;
\item[(iii)] $\tr{T\pi(f)}=0$, for all $f\in \E'(G,K)$;
\item[(iv)] $\tr{T\pi(f)}=0$, for all $f\in \H.$
\end{enumerate}
Obviously, (i) implies (ii). By density of $C_c^\infty(G,K)_{\theta\theta}$ in $\cE'(G,K)_{\theta\theta},$
for every finite subset $\theta\subset \dK,$ if follows that (ii) implies (iii).
Moreover, (iii) implies (iv). We will finish the proof by showing that (iv) implies (i).
          \par
Assume (iv).
For $x\in G$, we define $M(x)\coleq \tr{T\pi(x)}.$ By admissibility and $K\times K$-finiteness of $T$
it follows that $M$ is an analytic function on $G.$ From (iv) it follows that
$\langle f,M\rangle=0$ for any $f\in\H.$
Fix a finite subset $\theta \subset \dK$ such that
$M \in C_c^\infty(G,K)_{\theta\theta}.$ Then $M = \ga_\theta * M * \ga_\theta.$
Let $u\in U(\g)$, and set
$$
f\coleq \alpha_\theta\ast(L_u\delta_e)\ast\alpha_\theta\in\H,
$$
where $\delta_e$ is the Dirac measure at the unit element $e$ of $G.$ Then
$$
0=\langle f,M\rangle=\langle L_u\delta_e,M\rangle=(L_{\check{u}}M)(e).
$$
By analyticity this implies that $M$ vanishes on the identity component $G_0$ of $G.$
By $K$-stability of $\H$, we deduce that $M$ vanishes on $KG_0.$ Since $G$ is of the Harish-Chandra class,
this means that $M=0$ on $G.$
\end{proof}
         \par
\begin{cor}\label{end-hecke-ktype}
Let $(\pi,V)$ be an admissible representation of $G$ and assume that
$\theta_1$ and $\theta_2$ are finite subsets of $\dK$.
Then
$$
\begin{array}{rcl}
\pi(C_c^\infty(G,K)_{\theta_1\theta_2}) & = & \pi(\cE'(G,K)_{\theta_1\theta_2})\\
& = & \pi(\H_{\theta_1\theta_2})\\
& = & \pi(\H)\cap\End{V_K}_{\theta_1\theta_2}
\end{array}
$$
\end{cor}
\begin{proof}
This follows from Proposition \ref{end-hecke} by using $K$-equivariant projections.
\end{proof}

\section{Holomorphic families of representations and their derivatives}\label{s: holomorphic families}

Let $\fv$ be a finite dimensional real linear space. For any open subset $\Omega$ of its complexification
$\fv_\iC$, we denote by $\O(\Omega)$ the space of holomorphic $\C$-valued functions
on $\Omega$, endowed with the topology of uniform convergence on compact subsets.          \par
For $\mu\in\fv_\iC,$ we denote by $\O_\mu$  the algebra of germs at $\mu$ of
holomorphic functions defined on a neighborhood of $\mu.$ For any $\Omega$ as above,
$\mu\in\Omega$ and $f \in  \O(\Omega)$, the germ of $f$ at $\mu$ is denoted by
$\gamma_\mu(f)\in\O_\mu.$
             \par
Let $\cP = \cP(\fv_\iC)$ denote the algebra of polynomial functions $\fv_\iC \to \C.$ Then the map
$p \mapsto \gg_0(p)$ is an embedding of algebras, $\cP \hookrightarrow \cO_0.$
Accordingly, we shall view $\cP$ as a subalgebra of $\cO_0.$
             \par
The ring $\cO_0$ is local; its unique maximal ideal $\cM$ consists of the elements vanishing at $0.$
An ideal $\cI \ideal \cO_0$ is said to be cofinite if the quotient $\cO_0/\cI$ is finite
dimensional as a vector space over $\C.$ For $k \in \N,$ let  $\cP_k$ denote the space of polynomial functions
$\fv_\iC \to \C$ of degree at most $k.$ Then
\begin{equation}
\label{e: deco cO}
\cO_0 = \cP_k \oplus \cM^{k+1}.
\end{equation}
Therefore, the ideal $\cM^{k+1}$ is cofinite in $\cO_0.$

\begin{lemma}
\label{l: cofinite ideals}
Let $\cI$ be an ideal in $\cO_0.$ Then the following assertions are equivalent:
\begin{enumerate}
\itema the ideal $\cI$ is cofinite;
\itemb there exists a $k \in \N$ such that $\cM^{k+1} \subset \cI.$
\end{enumerate}
\end{lemma}

\begin{proof}
As $\cM^{k+1}$ is cofinite, (b) implies (a). Conversely, assume (a).

The space $V = \cO_0/\cI$ is a finite dimensional vector space, and an $\cO_0$-module for
left multiplication.  The associated algebra homomorphism $\cO_0 \to \End{V}$ is denoted
by $\gL.$ As $\cO_0$ is a commutative algebra and $V$ is finite dimensional, there exists a positive integer $p$ such that
the module $V$ decomposes as a finite direct sum of generalized weight spaces
$$
V_\chi\coleq \bigcap_{f \in \cO_0}  \ker (\gL(f) - \chi(f) {\rm id}_V)^p
$$
with $\chi \in \widehat \cO_0 \coleq  \Hom(\cO_0, \C).$ Since $\cO_0$ is a local ring with maximal ideal $\cM,$
the set $\widehat \cO_0$ of characters consists of the single element $\chi_0: g \mapsto g(0).$ It follows
that $V = V_{\chi_0},$ so that $(f - f(0))^p \in \cI$ for all $f \in \cO_0.$ In particular, $f^p \in \cI$ for all $f \in \cM.$
As an ideal, $\cM$ is generated by $n$ elements.
Hence,  $\cM^{k+1} \subset \cI,$ for $k  \geq np -1.$
\end{proof}

\subsection{The derivation process}\label{ss: derivation process}
For each vector $X \in \fv$ we denote by $\partial_X$ the first order differential operator
given by $\partial_X \gf (a) = \frac{d}{dz} \gf(a + z X)|_{z = 0}$ for $a \in \fv_\iC$ and $\gf$
a holomorphic function defined on a neighborhood of $a$ in $\fv_\iC.$
The map $X \mapsto \partial_X$ has a unique extension
to an algebra isomorphism $u \mapsto \partial_u$ from the symmetric algebra $S(\fv)$ of $\fv_\iC$
onto the algebra of constant coefficient (holomorphic) differential operators on $\fv_\iC.$
We will follow Harish-Chandra's convention to write
\begin{equation}
\label{e: HC notation for action of Sv}
\gf(a; u) \coleq  \partial_u \gf (a),
\end{equation}
for $\gf$ a holomorphic function defined on a neighborhood of $a.$

We define the pairing $\langle\,\cdot\, ,\,\cdot\,\rangle$ between $\O_0$ and
$S(\fv)$  by
\begin{equation}\label{sec3}
\begin{array}{rcl}
\O_0\times S(\fv) & \rightarrow & \C\\
(\varphi,u) & \mapsto & \varphi(0;u).
\end{array}
\end{equation}
For a given cofinite ideal $\cI \ideal \cO_0,$ let $S_\cI(\fv)$ denote the annihilator of $\cI$ in $S(\fv)$ relative
to this pairing.
For $k\in\N$, let $S_k(\fv)$ be the linear subspace of $S(\fv)$ consisting of the elements of order at most $k$.
\begin{lemma}\label{sec4}
Let $k \in \N.$ Then
\begin{enumerate}
\itema $S_k(\fv)=S_{\cM^{k+1}}(\fv);$
\itemb the pairing (\ref{sec3}) induces a perfect pairing $(\O_0/\cM^{k+1})\times S_k(\fv) \to \C.$
\end{enumerate}
\end{lemma}
\begin{proof}
The pairing $\langle\,\cdot\, ,\,\cdot\,\rangle$,
defined in \eqref{sec3}, vanishes on $\cM^{k+1}\times S_k(\fv)$.
Thus, $S_k(\fv)\subset S_{\cM^{k+1}}(\fv).$          \par
From the decomposition
$\O_0=\cM^{k+1}\oplus \P_k$, and non-degeneracy of the pairing, it follows
that $S_{\cM^{k+1}}(\fv)\hookrightarrow \P_k^*.$ In particular, the dimension of $S_{\cM^{k+1}}(\fv)$
does not exceed the dimension of $\P_k,$ which in turn equals the dimension of $S_k(\fv).$ Assertion
(a) now follows.

It also follows that
the induced embedding $S_k(\fv) = S_{\cM^{k+1}}(\fv)\hookrightarrow \P_k^* \simeq (\cO_0/ \cM^{k+1})^*$
is an isomorphism onto. By finite dimensionality, this implies assertion (b).
\end{proof}

\begin{lemma}\label{sec5}
Let $\cI$  be a cofinite ideal of $\O_0$. Then the pairing (\ref{sec3}) induces
a linear isomorphism $S_\cI(\fv)\simeq (\O_0/\cI)^*.$
\end{lemma}
\begin{proof}
By Lemma \ref{l: cofinite ideals} there exists a $k \in \N$ such that
$\cM^{k+1} \subset \cI.$
This inclusion induces an embedding
of $(\O_0/\cI)^*$ into $(\O_0/\cM^{k+1})^*.$
In view of Lemma \ref{sec4}, it follows that the pairing induces an
embedding  $(\O_0/\cI)^* \hookrightarrow S(\fv).$ Its image is contained in the annihilator
$S_\cI(\fv),$ by definition of the latter.
On the other hand, the pairing induces an inclusion $S_\cI(\fv)\hookrightarrow\O_0^*$
and elements of $S_\cI(\fv)$ vanish on $\cI$, so that $S_\cI(\fv)\hookrightarrow (\O_0/\cI)^*.$
The result follows.
\end{proof}

For $\mu \in \fv_\iC$ we denote by $T_\mu$ the translation in $\fv_\iC$ given by
$ \nu \mapsto \nu + \mu.$ We note that the pull-back map $T_\mu^*: \gf \mapsto \gf \after T_\mu$
induces a ring isomorphism from $\cO_\mu$ onto $\cO_0.$

Let $\cI\ideal \O_0$ be a cofinite ideal and let $\Omega$ be an open subset of $\fv_\iC$.
Then for every $f\in\O(\Omega)$ and each $\mu\in\Omega$ we define
\begin{equation}\label{e: JI}
J_\cI f(\mu)\coleq \mathrm{pr}_\cI(\, \gamma_0( T_\mu^*f)\,  )\;\in  \cO_0/\cI,
\end{equation}
where $\mathrm{pr}_\cI$ denotes the projection of $\cO_0$ onto $\cO_0/\cI.$

In the following lemma, which is a straightforward consequence of the definitions,
$\langle\, \cdot\,,\,\cdot\, \rangle$ denotes the pairing induced by
(\ref{sec3}), see Lemma \ref{sec5}.

\begin{lemma}\label{sec7}
Let $f \in \cO(\Omega).$ Then for all $\mu \in \Omega,$
$$
\langle J_\cI f(\mu),u\rangle = f(\mu;u),\qquad u\in S_{\cI}(\fv).
$$
\end{lemma}

\begin{cor}
\label{c: sec7}
The map $f\mapsto J_\cI f$ defines a continuous algebra homomorphism from $\O(\Omega)$ to $\O(\Omega,\O_0/\cI)$.
\end{cor}
\begin{proof}
Let $u\in S(\fv).$ For every $f\in\O(\Omega)$, the function $\partial_u f = f(\cdot;u)$ belongs to $\O(\Omega).$
Moreover, the map $\partial_u$ is a continuous linear endomorphism of $\cO(\Omega).$
In view of Lemma \ref{sec7}, it follows that for each $\xi \in (\cO_0/I)^*$ the map
$f \mapsto \xi \after [J_\cI(f)]$ is a continuous linear endomorphism of $\cO(\Omega).$
By finite dimensionality of $\cO_0/\cI,$ it follows that $J_\cI$ is continuous.

The assertion that $J_\cI$ is an algebra homomorphism follows from combining
the observations that $T_\mu^*,$ $\gg_0,$ and $\pr_{\cI}$ are algebra homomorphisms.
\end{proof}

\begin{ex}
\rm
Let $\xi\in\fv_\iC^*$. Denote by $e^\xi$ the holomorphic function on $\fv_\iC$ given by:
$$
e^\xi(\mu)\coleq e^{\xi(\mu)},\quad \mu\in\fv_\iC.
$$
In terms of the canonical identification of the symmetric algebra $S(\fv)$
with the algebra $\P(\fv_\iC^*)$ of polynomial functions on
$\fv_\iC^*$ we have $ \partial_u e^\xi = u(\xi) e^\xi,$ for $u \in S(\fv).$
Hence, if $\cI$ is a cofinite ideal in $\cO_0,$ then
for all  $\mu \in \fv_\iC$ and $u \in S_\cI(\fv),$
\vspace{1mm}
\begin{enumerate}
\item[(i)]
$  \langle\,  J_\cI e^\xi(\mu)  \, , \,   u  \, \rangle = u(\xi) \, e^{\xi(\mu)},$
\vspace{1mm}
\item[(ii)]
$\;J_\cI e^\xi \,(\mu) = e^{\xi(\mu)}\, \mathrm{pr}_\cI \circ\gamma_0(e^\xi).$
\end{enumerate}
\end{ex}

\begin{definition}\label{d: diffD}
\rm
Let $E$ be a finite dimensional $\O_0$-module. For every $f\in\O(\Omega)$ and all $\mu\in\Omega$,
we define $f^{(E)}(\mu)\in\mathrm{End}(E)$ by:
$$
f^{(E)}(\mu)e\coleq \gamma_0(T_\mu^*f)\cdot e, \quad e\in E.
$$
\end{definition}

\begin{ex}
\rm
If $E=\O_0/\cI$ for some cofinite ideal $\cI$ of $\O_0$, then, for any $f\in\O(\Omega),$

$$
J_\cI f(\mu)=
f^{(E)}(\mu)(1+\cI), \qquad \mu\in\Omega.
$$
\end{ex}
\medbreak\medbreak
Let $A$ be an algebra and $E$ an $A$-module. We denote by $\mathrm{ann}_A(E)$
the annihilator of $E$ in $A,$ i.e., the
kernel of the natural algebra homomorphism
$A\rightarrow\mathrm{End}(E)$.
If $E$ is finite dimensional, $\mathrm{ann}_A(E)$ is a cofinite ideal of $A$.

\begin{lemma}\label{sec13}
Let $\Omega \subset \fv_\iC$ be open, and let $E$ be a finite dimensional $\cO_0$-module.
Then for every $f \in \cO(\Omega)$ and all $\mu \in \Omega,$
\vspace{1mm}
\begin{enumerate}
\itema
$
f^{(E)}(\mu)e= J_{\mathrm{ann}_{\O_0}(E)}f(\mu)\cdot e,\quad \text{for all} e \in E;
$
\vspace{1mm}

\itemb
$
(T_\mu^* f)^{(E)} = T_\mu^* f^{(E)}.
$
\end{enumerate}
\end{lemma}
\begin{proof}
These formulas follow by straightforward computation.
\end{proof}

\begin{lemma}
\label{l: dual of End E}
Let $E$ be a finite dimensional $\cO_0$-module and let  $\eta \in \Endo(E)^*.$
Then there exists an element $u = u_\eta \in S(\fv)$ such that
$$
\eta \after f^{(E)} =  \partial_u f,
$$
for every open $\Omega \subset \fv_\iC$ and all $f \in \cO(\Omega).$
\end{lemma}

\begin{proof}
It suffices to prove this for $\eta = e^* \otimes e,$ with $e^* \in E^*$ and $e \in E.$
Then, for $f \in \cO(\Omega)$ and $\gl \in \Omega,$
\begin{equation}\label{e: linear-f as diff}
\eta \after f^{(E)}(\gl) = e^*(\gg_0(T^*_\gl f) \cdot e).
\end{equation}
Let $\cI$ be the (cofinite) annihilator of $e$ in $\cO_0.$  Then the linear functional $L: \gf \mapsto e^*(\gf \cdot e)$ on $\cO_0$ factors through a linear map
$\cO_0/\cI \to \C.$  Hence, in view of Lemma \ref{sec5}, there exists
an element $u \in S_\cI(\fv)$ such that $L(\gf) = \partial_u \gf(0)$ for all $\gf \in \cO_0.$
It follows that the expression on the right-hand side of \eqref{e: linear-f as diff} equals $\partial_u f(\gl).$
\end{proof}

We shall also need a kind of converse to the above lemma.

\begin{lemma}
\label{l: E associated with u}
Let $F \subset S(\fv)$ be a finite subset. Then there exists a finite dimensional $\cO_0$-module
$E,$ and linear functionals $\eta_u \in \Endo(E)^*,$ for $u \in F,$ such that
$$
\eta_u \after f^{(E)} = \partial_u f
$$
for every $u \in F,$ every open $\Omega \subset \fv_\iC$ and all $f \in \cO(\Omega).$
\end{lemma}

\begin{proof}
By taking direct sums of finite dimensional $\cO_0$-modules we may reduce to the case that $F$ consists
of a single element $u \in  S(\fv).$ Let $k$ be the order of $u.$ Then $\gf \mapsto \gf(0;u)$ defines
a linear functional $e^*$ on $\cO_0/\cI,$ for $\cI = \cM^{k+1}.$ We put $E = \cO_0/\cI$
and let $e$ denote the image of $1\in\cO_0$ in $E.$ Let $\eta = e^*\otimes e$ be the linear
functional on $\Endo(E)$ defined by $T \mapsto e^*(Te).$ Then for all $f \in \cO(\Omega)$ and all $\gl \in \Omega$ we have
$$
\eta \after f^{(E)}(\gl) = \eta(\gg_0 (T_\gl^* f)) = e^* (\gg_0 (T_\gl^* f)\cdot e) =
\partial_u (\gg_0(T_\gl^*f))(0) = \partial_u f(\gl).
$$
\end{proof}

We retain the assumption that $\Omega$ is an open subset of $\fv_\iC.$

\begin{cor}
\label{c: continuity of diffD}
Let $E$ be a finite dimensional $\O_0$-module. Then $f\mapsto f^{(E)}$ is a continuous algebra homomorphism from $\O(\Omega)$ to $\O(\Omega,\mathrm{End}(E))$.
\end{cor}

\begin{proof}
The map is an algebra homomorphism by Lemma \ref{sec13} (a) and Corollary \ref{c: sec7}.
The continuity is an immediate consequence of Lemma \ref{l: dual of End E}.
\end{proof}
We  agree to use the following notation for the map of Corollary \ref{c: continuity of diffD},
\begin{equation}
\label{e: defi diffD}
\diffD^{(E)} : \;f \mapsto f^{(E)}, \quad \cO(\Omega) \to \cO(\Omega, \End{E}).
\end{equation}
The following property is an immediate consequence of the definitions;
here we keep
in mind that $\End{E} \oplus \End{F} \embeds \End{E \oplus F},$ naturally.

\begin{property}
Let $E$, $F$ be two finite dimensional $\O_0$-modules. Then for every $f\in\O(\Omega),$
$$
f^{(E\oplus F)}=f^{(E)}\oplus f^{(F)}.
$$
\end{property}

To prepare for deriving more properties of the map $\diffD^{(E)},$ we formulate
a few  results on finite dimensional $\cO_0$-modules.

\begin{lemma}
Let $E$ be a finite dimensional $\O_0$-module. Then $E$ is cyclic if and only if there exists a cofinite ideal $\cI$ of $\O_0$ such that $E\simeq\O_0/\cI$.
\end{lemma}

\begin{proof}
Straightforward.
\end{proof}

\begin{cor}\label{sec18}
Let $E$ be a finite dimensional $\O_0$-module. Then
there exist finitely many cofinite ideals $\cI_1, \ldots, \cI_n$ of $\cO_0$ such that
$E$ is a quotient of the direct sum $\cO_0/ \cI_1 \oplus \cdots \oplus \cO_0 /\cI_n$
of $\cO_0$-modules.          \par
In particular, there exist $k,N  \in \N$ such that $E$ is a quotient of the $\cO_0$-module
$(\O_0/\cM^{k+1})^N$ for some $k,N \in\N.$
\end{cor}
\begin{proof}
The first assertion results from the previous lemma. The second follows from Lemma \ref{l: cofinite ideals}.
\end{proof}

Besides the decomposition (\ref{e: deco cO}), we have the following decomposition
of $\cP = \cP(\fv_\iC),$ for $k \in \N,$
\begin{equation}
\label{e: deco cP}
\P  =  \P_k\oplus (\cM\cap\P)^{k+1}.
\end{equation}
Hence, the
embedding $\cP \hookrightarrow \cO_0$ induces, for each $k \in \N,$ an isomorphism of algebras
\begin{equation}\label{sec23}
\iota_k: \;\;\P/(\cM\cap\P)^{k+1}\stackrel{\sim}{\longrightarrow}\O_0/\cM^{k+1}.
\end{equation}
It follows that $\cP/(\cM \cap \P)^{k+1}$ is a local ring, with unique maximal
ideal equal to $(\cM \cap \cP)/(\cM \cap \cP)^{k+1}.$ Thus, if $\widetilde \cI \subset \cP$
is an ideal with $(\cM \cap \cP)^{k+1} \subset \widetilde\cI,$ then $\widetilde \cI \subset \cM.$

\begin{lemma}\label{sec24}
Let $\widetilde{\cI}$ be an ideal of $\P.$  Then the following assertions are equivalent:
\begin{enumerate}
\itema
there exists a  $k \in \N$ such that $(\cM \cap \cP)^{k + 1} \subset
\widetilde \cI;$
\itemb
there exists an ideal $\cI\;\lhd\;\O_0 $ of finite codimension, such that
$\widetilde{\cI}=\cI\cap\P.$
\end{enumerate}
If any of these conditions is fulfilled, then the ideal $\cI$ in (b) is unique
and the embedding of $\cP$ into $\cO_0$ induces an isomorphism of algebras $\cP/\widetilde \cI \to \cO_0 / \cI.$
\end{lemma}

\begin{proof}
Assume (a).
The image $\widetilde \cI'$ of $\widetilde \cI$ in $\cP/(\cM \cap \cP)^{k+1}$
is an ideal.
Its image $\iota_k(\tilde \cI')$ is an ideal of $\cO_0/ \cM^{k+1}.$ Let $\cI$ be the
preimage of $\iota_k(\tilde\cI')$ in $\cO_0.$ Then the following diagram commutes
$$
\xymatrix@1{
\cP/(\cM \cap \cP)^{k+1}\;\;\ar[d]_p\ar[r]^{\iota_k} & \;\;\ar[d]^q\cO_0/\cM^{k+1}\\
\cP/\widetilde\cI\; \ar[r]^{\bfi} & \;\; {\cO_0/\cI}.
}
$$
Here $\bfi$ is induced by the inclusion map $\cP\hookrightarrow \cO_0.$
The kernel of $p$ equals $\widetilde \cI'$ and the kernel of $q$ equals $\iota_k(\widetilde \cI').$
As $\iota_k$ is an isomorphism of algebras, it follows that $\bfi$ is an isomorphism of algebras,
and (b) is immediate.

Conversely, assume (b). Let $\cI_j \ideal \cO_0$ be ideals such that $\widetilde \cI = \cI_j \cap \cP,$
for $j =1, 2.$ We will complete the proof by showing that (a) holds and that $\cI_1 = \cI_2.$

Since  $\cI_1$ and $\cI_2$ are cofinite, there exists a constant $k \in \N$ such that
$\cM^{k+1}\subset \cI_j,$ for both $j =1,2.$ Therefore,
$$
(\cM \cap \cP)^{k+1} \subset (\cM^{k+1} \cap \cP ) \subset \widetilde \cI
$$
and (a) follows.
Moreover, for each $j=1,2$ we have
the following commutative diagram:
$$
\xymatrix@1{
\cP/(\cM \cap \cP)^{k+1}\;\;\ar[d]_{p}\ar[r]^{\iota_k} & \;\;\ar[d]^{q_j}\cO_0/\cM^{k+1}\\
\cP/\widetilde\cI\; \ar[r]^{\bfi_j} & \;\; {\cO_0/\cI_j}.
}
$$
From the assumption on $\cI_j$ it follows that the map $\bfi_j$ is injective. Moreover, since
$\iota_k$ and $q_j$ are surjective, it follows that $\bfi_j$ is an isomorphism of algebras,
for $j =1, 2.$ This implies that $\iota_k(\ker p) = \ker (q_j).$ Since $\cI_j$ equals
the preimage of $\ker(q_j)$ in $\cO_0,$ for $j =1,2,$ it follows that $\cI_1 = \cI_2.$
\end{proof}
\par
Let $\FMod_{\O_0}$ denote the category of  finite dimensional $\O_0$-modules and $\FMod_\P$
the category of finite dimensional $\P$-modules $E$
for which there exists a $k \in \N$ such that
$$
(\cM\cap \cP)^{k+1} \subset \mathrm{ann}_\P(E).
$$
If $E,F$ belong to $\FMod_\P$ then the $\FMod_\P$-morphisms are defined to be the
$\cP$-module homomorphisms $E \to F.$
We observe that their kernels and images belong to the category $\FMod_\P$
as well.
            \par
If $E$ is a finite dimensional $\cO_0$-module, then its annihilator $\cI$ is cofinite.
Since $\cP \hookrightarrow \O_0,$ the space
$E$ is a finite dimensional $\cP$-module as well and its annihilator $\widetilde \cI \coleq  \ann_\cP(E)$ in $\cP$ is given by $\widetilde \cI = \cI \cap \cP.$
Furthermore, by Lemma \ref{l: cofinite ideals} there exists a $k \in \N$ such that
$$
(\cM \cap \cP )^{k+1} \subset  \ann_\cP(E).
$$
We conclude that there is a well-defined forgetful functor $\cF: \FMod_{\cO_0} \to \FMod_\cP.$

\begin{lemma}\label{sec25}
The forgetful functor $\cF: \FMod_{\cO_0} \to \FMod_\cP $ is an isomorphism of categories.
\end{lemma}

\begin{proof}
Let $E$ be a non-trivial $\cP$-module from the category $\FMod_\cP.$
We will first show that $E$ carries a unique compatible
structure of $\cO_0$-module.
             \par
The ideal $\widetilde \cI  = \mathrm{ann}_\P(E)$ satisfies condition (a)
of Lemma \ref{sec24}, hence equals
$\cP \cap \cI$ for a unique cofinite
ideal $\cI \ideal \cO_0.$ Via the isomorphism $\cP/\widetilde \cI \simeq \cO_0/\cI$ we equip
$E$ with the structure of $\cO_0$-module; this establishes existence.
            \par
For uniqueness, assume that $E$ is equipped with a compatible structure of $\cO_0$-module.
Let $\cI'$ be the annihilator of $E$ in $\cO_0.$ Then $\cI'$ is cofinite and, by compatibility, $\widetilde \cI = \cI' \cap \cP.$ By the first part of the proof, we
see that  $\cI' =  \cI.$ Furthermore, by compatibility it follows that
the diagram
$$
\xymatrix{
\ar[d] \cP/\widetilde \cI \ar[rd] & \\
\cO_0/\cI \ar[r] & \End{E}
}
$$
commutes. The vertical arrow is induced by the inclusion $\cP \to \cO_0$ and represents an isomorphism.
This establishes uniqueness.
                   \par
It follows from these considerations that each submodule
of $E$ in the category $\FMod_\P$ is a submodule for the associated compatible structure of
$\cO_0$-module as well.
                  \par
For $E$ a finite dimensional $\cP$-module from $\FMod_\P,$ let $\cG(E)$ denote the same space
with the uniquely defined compatible structure of $\cO_0$-module. If $E_0$ is a submodule in the category
$\FMod_\P,$ then by what we said before, $\cG(E_0)$ is an $\cO_0$-submodule of $\cG(E).$
The quotient $\cG(E)/\cG(E_0)$ is an $\cO_0$-module in a natural way, whose $\cO_0$-module
structure is compatible with the $\cP$-module structure of $E/E_0.$ Hence, $\cG(E)/\cG(E_0) = \cG(E/E_0).$
It now follows that every morphism $f: E \to F$ in the category $\FMod_\P$ is an $\cO_0$-module
homomorphism $\cG(E) \to \cG(F).$ We thus see that $\cG$ defines a functor $\FMod_\P \to  \FMod_{\cO_0},$
which obviously is a two-sided inverse to $\cF.$
\end{proof}

\begin{cor}
For every pair $E,\,F$ of finite dimensional $\cO_0$-modules,
$$\mathrm{Hom}_{\O_0}(E,F)=\mathrm{Hom}_\P(E,F).$$
\end{cor}

Our next objective is to consider tensor products in the categories $\FMod_\P$ and $ \FMod_{\cO_0}.$
                   \par
Let $n \in \N^*$ and consider the $n$-fold Cartesian product $\fv_\iC^n$ of
$\fv_\iC.$ Projection onto the $j$-th coordinate
is denoted by $\pr_j.$ Pull-back by $\pr_j$ defines an embedding of algebras
$\pr_j^*: \; p \mapsto p \after \pr_j,$ $\cP \to \cP(\fv_\iC^n).$ The multi-linear map
$$
(p_1, \ldots , p_n) \mapsto \prod_{j=1}^n \pr_j^*(p_j)
$$
induces an isomorphism of algebras $\cP^{\otimes n } \to \cP(\fv_\iC^n),$ via which we shall identify the
elements of these spaces. Accordingly,
$$
(p_1 \otimes \cdots \otimes p_n)(\mu_1, \ldots, \mu_n) = p_1(\mu_1) \cdots p_n(\mu_n),
$$
for $p_j \in \cP$ and $\mu_j \in \fv_\iC.$ The maximal ideal $\cM_{\cP^{\otimes n}}$ in $\cP(\fv_\iC^n)$ consisting
of the polynomials vanishing at $0$ is now given by $\cM_{\cP^{\otimes n}} = \sum_{i=1}^n \cM_{\cP, n, i},$
where
$$
\cM_{\cP, n , i} \coleq  \cP \otimes \cdots \otimes \cP \otimes
{\buildrel i \over {\overbrace{(\cM \cap \cP)}}
}\otimes \cP \otimes \cdots \otimes \cP.
$$
\begin{lemma}
\label{l: tensor products and cC}
Let $E_1, \ldots, E_n$ be finite dimensional $\cP$-modules from the category $\FMod_\P.$
Then $E_1 \otimes \cdots \otimes E_n$ is a $\cP^{\otimes n}$-module from the category $\FMod_{\P^{\otimes n}}.$
\end{lemma}
\begin{proof}
There exists a $k \in \N$ such that $(\cM\cap\cP)^{k+1}$ annihilates each of the modules
 $E_i,$ for $1 \leq i \leq n.$ It is now readily checked that
$$
(\cM_{\cP^{\otimes n}})^{n (k+1)}\subset \ann_{\cP^{\otimes n}}(E_1 \otimes \cdots \otimes E_n).
$$
\end{proof}
We consider the map
$$
\alpha_n:\;\; \fv_\iC^{ n}  \rightarrow  \fv_\iC, \quad
(\mu_1,\dots,\mu_n)  \mapsto \mu_1 + \cdots +\mu_n.
$$
Pull-back by $\alpha_n$ induces an algebra homomorphism
\begin{equation}\label{alphaene}
\alpha_n^*:\;\P\longrightarrow\P^{\otimes n}.
\end{equation}

\vspace{2mm}
\begin{lemma}
The homomorphism $\ga_n^*$ maps the maximal ideal $\cM\cap \cP$ of $\cP$ into the maximal ideal
$\cM_{\cP^{\otimes n}}$ of $\cP^{\otimes n}.$
\end{lemma}

\begin{proof}
The ideal $\cM\cap \cP$ is generated by the first order polynomials $\xi \in \fv_\iC^*.$
Now
\begin{equation}
\label{e: ga n on xi}
\ga_n^*(\xi) = \sum_{i = 1}^n    1 \otimes \cdots \otimes 1 \otimes {\buildrel i \over {{\xi}}} \otimes 1 \otimes \cdots \otimes 1,
\end{equation}
and the result follows.
\end{proof}

\begin{cor}\label{c: sec28}
Let $n\in\N^*$. Via the homomorphism (\ref{alphaene}),
every finite dimensional $\P^{\otimes n}$-module from
the category $\FMod_{\P^{\otimes n}}$ becomes a $\cP$-module from the category $\FMod_\P.$
\end{cor}

\begin{proof}
Let $E$ be a non-trivial module from the category $\FMod_{\P^{\otimes n}}.$ Then there exists a
$k \in \N$ such that
$$
(\cM_{\cP^{\otimes n}})^{k+1} \subset \ann_{\cP^{\otimes n}}(E).
$$
By application of the previous lemma, we obtain
$$
\ga_n^*((\cM\cap \cP)^{k+1}) \subset  \ann_{\cP^{\otimes n}}(E).
$$
The $\P$-module structure on $E$ is defined by:
$$
\xymatrix{
\P\ar[r]^{\alpha_n^*}&\P^{\otimes n}\ar[r]&\End{E}.
}
$$
From this we see that
$
\mathrm{ann}_\P(E) = \alpha_n^{*-1}(\ann_{\P^{\otimes n}}(E))
$
and we infer that
$$
(\cM\cap \cP)^{k+1} \subset \ann_{\cP}(E).
$$
\end{proof}

\begin{rem}
\label{r: tensor products in cC}
\rm
If $E$ is a finite dimensional $\cP$-module, we define
$$
m_E:\;\P \; \longrightarrow\; \End{E}
$$
by $m_E(p)e\coleq p\cdot e,$ for $p \in \cP$ and $e\in E$.

If $E_1, \cdots, E_n$ are finite dimensional modules from the category $\FMod_\P,$ then by
combining Lemma \ref{l: tensor products and cC} and Corollary \ref{c: sec28} we may equip the tensor product
$E_1 \otimes \cdots \otimes E_n$ with the structure of a module from the same category
$\FMod_\P.$ The module structure is given by the
rule
\begin{equation}
\label{e: rule for multiplication}
m_{E_1 \otimes \cdots \otimes E_n} = (m_{E_1} \otimes \cdots \otimes m_{E_n})\after \ga_n^*.
\end{equation}
In view of (\ref{e: ga n on xi}) this module structure is completely determined by the
rule
\begin{eqnarray}
\lefteqn{
m_{E_1 \otimes \cdots \otimes E_n}(\xi) =}\nonumber\\
\label{e: m on xi}
& = &
\sum_{i=1}^n \id_{E_1}  \otimes \cdots \otimes \id_{E_{i-1}} \otimes m_{E_i}(\xi) \otimes \id_{E_{i+1}}
\otimes \cdots \otimes \id_{E_n}
\end{eqnarray}
for $\xi \in \fv_\iC^*\subset \cP.$
                     \par
Accordingly, if $E_1, \ldots, E_n$
are finite dimensional $\cO_0$-modules, then in particular they
are $\cP$-modules from the category $\FMod_\P.$ We equip the module $E_1 \otimes \cdots \otimes E_n$
from $\FMod_\P$ with the unique compatible structure of $\cO_0$-module.
\end{rem}

\begin{lemma}\label{sec34}
Let $E_1,\ldots, E_n$ be finite dimensional $\O_0$-modules. Then, as $\O_0$-modules:
$$
E_1\otimes \cdots \otimes E_n\simeq E_1\otimes(E_2\otimes\cdots \otimes E_n).
$$
\end{lemma}

\begin{proof}
In view of Remark \ref{r: tensor products in cC}, it suffices to establish the identity
as an identity of $\cP$-modules. Thus, it suffices to show that
\begin{equation}
m_{{E_1} \otimes \cdots \otimes {E_n}} =  m_{{E_1} \otimes ({E_2\otimes\cdots \otimes E_n})}.
\end{equation}
Since $\fv_\iC^*$ generates $\cP,$ it suffices to check this identity on any
element $\xi \in \fv_\iC^*.$ This is easily done by using the identity
(\ref{e: m  on xi}).
\end{proof}

We return to  the setting of an open subset $\Omega \subset \fv_\iC$ and resume
our study of the map $\diffD^{(E)}: f \mapsto f^{(E)},$
$\cO(\Omega) \to \cO(\Omega , \End{E}),$
introduced in (\ref{e: defi diffD}), for $E$ a finite dimensional $\cO_0$-module.

\begin{lemma}
\label{l: restriction diffD to cP}
The restriction to $\P$ of the map
$\diffD^{(E)}$ is given by
$$
\diffD^{(E)}|_\cP = (\id_\P\otimes m_E)\circ\alpha_2^*.
$$
 In particular, this restriction maps $\P$ into $\P\otimes \End{E}$.
\end{lemma}

\begin{proof}
Let $p\in\P$. For any $\mu\in\fv_\iC$,  we have $T_\mu^*p\in\P,$ and for any $e\in E,$
$$
p^{(E)}(\mu)e  =  (T_\mu^*p)\cdot e = m_E(T_\mu^*p)e.
$$
Hence,
$$
(\ev_\mu \otimes \id_{\End{E}})\after \diffD^{(E)}|_\cP = m_E \after T_\mu^*|_\cP.
$$
Viewing $\ev_\mu\otimes\id_\P$ as a $\P$-valued function on $\P\otimes\P$,
we may identify the map  $T_\mu^*$  with $(\ev_\mu\otimes\id_\P)\circ\alpha^*_2$ on $\P.$
Using the  relation
$$
m_E\circ(\ev_\mu\otimes\id_\P)=(\ev_\mu\otimes\id_{\End{E}})\circ(\id_\P\otimes m_E),
$$
we obtain the assertion of the lemma.
\end{proof}

\begin{prop}\label{sec33}
Let $E_1$ and $E_2$ be two finite dimensional $\O_0$-modules. Then, for every $f\in\O(\Omega)$,
$$(\diffD^{(E_1)}\otimes \id_{\End{E_2}})f^{(E_2)}=f^{(E_1\otimes E_2)}.$$
\end{prop}

\begin{proof}
By density of $\P$ in $\O(\Omega)$ and continuity of the maps $\diffD^{(E_1\otimes E_2)}$
and $(\diffD^{(E_1)}\otimes \id_{\End{E_2}})\circ \diffD^{(E_2)},$ see
Corollary \ref{c: continuity of diffD}, it suffices to establish the validity of the identity
on a fixed element $p \in \cP.$
                    \par
By definition of $\diffD^{(E_1)},$ we have $\ev_0 \after \diffD^{(E_1)}q = m_{E_1}(q)$ for each $q \in \cP.$
Using Lemma \ref{l: restriction diffD to cP} we  obtain
$$
\begin{array}{rcl}
(\diffD^{(E_1)}\otimes \id_{\End{E_2}})p^{(E_2)}(0) & = & (m_{E_1}\otimes \id_{\End{E_2}})p^{(E_2)}\\
& = & (m_{E_1}\otimes \id_{\End{E_2}})\circ (\id_{\P}\otimes m_{E_2})(\alpha_2^*p)\\
& = & (m_{E_1}\otimes m_{E_2})\circ\alpha_2^*(p)\\
& = & m_{E_1\otimes E_2}(p)
=  p^{(E_1\otimes E_2)}(0).
\end{array}
$$
The result now follows by translation invariance (see Lemma \ref{sec13}(b)).
\end{proof}

\begin{prop}\label{p: iterateder}
Let $E_1,\dots,E_n$ be finite dimensional $\O_0$-modules.
Then, for every $f\in\O(\Omega)$,
$$
\left(\diffD^{(E_1)}\circ\cdots\circ \diffD^{(E_n)}\right)(f)=f^{(E_1\otimes\cdots\otimes E_n)}.
$$
\end{prop}
In the formulation of this proposition we have slightly abused notation, by using
the abbreviation $\diffD^{(E_k)}$ for
$\diffD^{(E_k)} \otimes \id_{\End{E_{k+1}}} \otimes \cdots \otimes \id_{\End{E_n}}.$

\begin{proof}
In view of Lemma \ref{sec34}, the result follows by repeated application of Proposition \ref{sec33}.
\end{proof}

Let $\lambda\in\fv_\iC$. By $\cI_\lambda$ we denote the subset of $\O_0$
consisting of all elements $\varphi\in\cM$ satisfying
\begin{equation}\label{e: idealcodim2}
\partial_\lambda \varphi \in \cM.
\end{equation}
By Leibniz's rule, $\cI_\lambda$ is an ideal of $\cO_0$ containing $\cM^2$.

\begin{lemma}
\label{l: codimension two}
Assume $\lambda\neq 0.$ Then $\cI_{\lambda}$ is an ideal in $\cO_0$ of codimension 2. If $\xi \in \fv_\iC^*$
is such that $\xi(\gl) \neq 0,$ then $\cO_0 = \C 1 \oplus \C \xi \oplus \cI_\gl.$
\end{lemma}

\begin{proof}
Fix a basis
$\xi_1,\dots,\xi_N$ for $\fv_\iC^*$ such that $\xi_1 = \xi$ and such that
$\xi_j(\gl) = 0$ for $j > 1.$
The monomials
$$
\xi^\beta\coleq \xi_1^{\beta_1}\cdots\xi_N^{\beta_N}, \qquad \gb\in \N^N,
$$
form a basis of the complex vector space $\cP.$
Consider a polynomial $p\in\P$ and write
$
p=\sum_\beta c_\beta \xi^\beta,
$
with $c_\gb \in \C.$ Then by an easy calculation we find that
$$
p(0) = c_0 \quad \textrm{and}\quad \partial_\lambda p(0)
= c_{(1,0,\ldots, 0)}\xi_1(\gl).
$$
It follows that $p$ belongs to $\cI_\gl$ if and only if
$c_\gb = 0$ for $\gb = 0$ and for $\gb = (1,0,\ldots, 0).$ This implies that
$\cP = \C 1 \oplus \C \xi_1 \oplus (\cI_\gl \cap \cP).$ By Lemma \ref{sec24} the inclusion map $\cP \to \cO_0$
induces an isomorphism of algebras $\cP/(\cI_\gl\cap \cP) \simeq \cO_0/\cI_\gl.$
All assertions follow.
\end{proof}

Let $\widetilde \cI_1, \ldots, \widetilde \cI_n$ be a collection of ideals from $\cP$
containing $(\cM \cap \cP)^{k+1}$ for some $k \in \N.$ Then the algebra homomorphism
$\ga_n^*: \cP \to \cP^{\otimes n}$ induces an algebra homomorphism
$$
\bar\ga_n^* : \cP \to (\cP/\widetilde \cI_1) \otimes \cdots \otimes (\cP/\widetilde \cI_n)
$$
by composition with the natural projection from $\cP^{\otimes n}$ onto the quotient algebra on the right.

\begin{lemma}\label{sec40}
Let $\lambda\in\fv_\iC\setminus \{0\}$, and let $\tilde{\cI} \ideal \cP$ be an ideal
with $(\cM\cap\P)^{k+1}\subset\tilde{\cI}$ for some $k\in\N$. Then the kernel of
$$
\bar\alpha_2^*:\;\;\P\longrightarrow \P/\tilde{\cI}\otimes(\P/\cI_\lambda\cap\P)
$$
is equal to $\{p\in\tilde{\cI}\mid \,
\partial_\gl p \in\tilde{\cI}\}.$
In particular, this kernel contains $\cM^{k+2} \cap \cP.$
\end{lemma}
\begin{proof}
The last assertion follows from the first one by application of the Leibniz rule. We turn
to the proof of the first assertion.
             \par
We fix a basis $\xi_1 \ldots, \xi_N$ of $\fv_\iC^*$ as in the proof of Lemma \ref{l: codimension two} and
adopt the notation of that proof. Let $X_1, \ldots, X_N$ be the dual basis of $\fv_\iC.$
We note that $\gl$ is proportional to $X_1.$
Our first goal is to obtain a suitable formula
for $\ga_2^*.$
                \par
Let $p \in \cP$ and $\mu \in \fv_\iC.$
Using Taylor expansion at $\mu$ we find, for all $\nu\in\fv_\iC$,
$$
p(\mu+\nu)=\sum_\beta\frac{p(\mu;{X}^\beta)}{\beta !}\,{\xi}^\beta(\nu)
$$
(where $X^\beta\coleq X_1^{\beta_1}\cdots X_N^{\beta_N}$ and $\beta !\coleq  \beta_1 !\cdots\beta_N !$).
From this formula we deduce that
$$
\alpha_2^*(p)
=\sum_\beta\frac{p(\dotvar ;X^\beta)}{\beta !}\otimes {\xi}^\beta.
$$
In view of the characterization of $\cI_\gl$ in the proof of Lemma \ref{l: codimension two},
we now see that
\begin{equation}\label{eq4}
\alpha^*_2(p)=p\otimes 1  + p(\dotvar ; X_1)\otimes\xi_1   \quad\mathrm{mod}\;\;
\cP \otimes(\cI_\lambda\cap\P).
\end{equation}
Furthermore, again by Lemma \ref{l: codimension two}, we have $\cP = \C 1 \oplus \C \xi_1 \oplus (\cI_\gl \cap \cP),$
so that $\cP^{\otimes 2} = [\cP \otimes (\C 1 \oplus \C \xi_1)] \oplus [\cP \otimes (\cI_\gl \cap \cP)].$
Accordingly,
the natural projection map $\cP^{\otimes 2}\to \cP/\tilde \cI \otimes \cP / (\cI_\gl \cap \cP)$
has kernel equal to
\begin{equation}
\label{e: kernel pr}
[\widetilde\cI \otimes (\C 1 \oplus \C \xi_1)]\, \oplus\,[ \cP \otimes (\cI_\gl \cap \cP)].
\end{equation}
Combining (\ref{eq4}) and (\ref{e: kernel pr}), we finally see that  $\bar \ga_2^*(p) = 0$ if and only if
$$
p\otimes 1 + p(\dotvar ; X_1)\otimes\xi_1 \;\; \in \widetilde \cI \otimes (\C 1 \oplus \C \xi_1),
$$
which in turn is equivalent to  $p \in \widetilde \cI$ and $\partial_\gl p
\in \widetilde \cI.$
\end{proof}

\begin{lemma}
\label{sec41}
Let $\lambda_1,\dots,\lambda_n\in\fv_\iC\setminus\{0\}$. If $n>1$, the kernel of
$$
\bar{\alpha}_n^*:\P\longrightarrow(\P/\cI_{\lambda_1}\cap\P)\otimes\cdots\otimes(\P/\cI_{\lambda_n}\cap\P)
$$
equals the kernel of $\bar{\alpha_2}^*:\P\longrightarrow(\P/\ker{\bar{\alpha}_{n-1}^*})\otimes(\P/\cI_{\lambda_n}\cap\P),$
where $\bar{\alpha}_{n-1}^*$ denotes the composition on the left of the algebra homomorphism $\alpha^*_{n-1}$ with
 the projection $\P^{\otimes (n-1)}\rightarrow (\P/\cI_{\lambda_1}\cap\P)\otimes\cdots\otimes(\P/\cI_{\lambda_{n-1}}\cap\P)$.
\end{lemma}
             \par
\begin{proof}
The map $\bar{\alpha}_{n-1}^*$ induces an embedding (denoted by the same symbol)
$$
\bar{\alpha}_{n-1}^*:\P/\ker{\bar{\alpha}_{n-1}^*}\hookrightarrow
(\P/\cI_{\lambda_1}\cap\P)\otimes\cdots\otimes(\P/\cI_{\lambda_{n-1}}\cap\P).
$$
Moreover, the following diagram commutes
\vspace{1mm}
$$
\xymatrix{
\P\ar[d]_{\bar{\alpha}_2^*}\ar[r]^{\!\!\!\!\!\!\!\!\!\!\!\!\!\!\!\!\!\!\!\!\!\!\!\!\!\!\!\!\!\!\!\!\!\!\!\!\!\!\!\!\!\!\bar{\alpha}_n^*} & (\P/\cI_{\lambda_1}\cap\P)\otimes\cdots\otimes(\P/\cI_{\lambda_n}\cap\P)\\
(\P/\ker{\bar{\alpha}_{n-1}^*})\otimes(\P/\cI_{\lambda_n}\cap\P)\ar@{^{(}->}[ur]_{\qquad\bar{\alpha}_{n-1}^*\otimes\id_{\P/(\cI_{\lambda_n}\cap\P)}} &
}
$$
\vspace{2mm}
Thus, $\ker{\bar{\alpha}_n^*}=\ker{\bar{\alpha}_2^*}$.
\end{proof}

\begin{lemma}\label{sec42}
Let $\lambda_1,\dots,\lambda_n\in\fv_\iC\setminus\{0\}$. Then the kernel of
$$
\bar{\alpha}_n^*:\P\longrightarrow(\P/\cI_{\lambda_1}\cap\P)\otimes\cdots\otimes(\P/\cI_{\lambda_{n}}\cap\P)
$$
consists of the elements $p$ of $\cM\cap\P$ satisfying  $p(\dotvar;\lambda_{j_1}\cdots\lambda_{j_l})\in\cM$, for
all $1\leq l\leq n$, $1\leq j_1<\cdots<j_l\leq n.$
In particular, this kernel contains
$\cM^{n +1} \cap \cP.$
\end{lemma}

\begin{proof}
The final statement follows from the first by repeated application of the Leibniz rule.
We prove the first statement by induction on $n.$ If $n=1$,
then $\bar{\alpha}_1^*$ is just the projection
$\P\rightarrow\P/\cI_{\lambda_1}\cap\P$, and
$\ker(\bar{\alpha}_1^*)=\cI_{\lambda_1}\cap\P=\{p\in\cM\cap\P\mid \,p(\dotvar ;\lambda_1)\in\cM\}$.
Let $n\geq 2$ and assume the statement to be valid for
smaller values of $n.$
Using Lemmas \ref{sec40} and \ref{sec41}, we find that
$$
\ker(\bar{\alpha}_n^*)=
\{p\in \cP \mid p \in \ker\bar{\alpha}_{n-1}^*\;\; \textrm{and}\;\;
\partial_{\gl_n} p \in\ker\bar{\alpha}_{n-1}^*\}.
$$
By the induction hypothesis it follows that  $\ker (\bar \ga_n^*)$ consists of all polynomials $p \in \cM \cap \cP$
such that for all $1\leq l \leq n-1$ and all $1\leq j_1 < \cdots < j_l \leq n-1$ the polynomials
$p(\dotvar; \lambda_{j_1}\cdots\lambda_{j_l})$ and $p(\dotvar ; \lambda_{j_1}\cdots\lambda_{j_l}\gl_n) =
\partial_{\gl_n}p(\dotvar ; \lambda_{j_1}\cdots\lambda_{j_l})$
belong to $\cM.$ This implies the result.
\end{proof}

\begin{prop}
\label{sec43}${}$
\begin{enumerate}
\itema
Let $E_1,\dots,E_n$ be finite dimensional $\O_0$-modules.
There exist $N,\,k\in\N$ such that $E_1\otimes\cdots\otimes E_n$ is  a quotient of the $\O_0$-module
$\C^N\otimes \O_0/\cM^{k+1}$ (action on the second tensor component).
\itemb
Let $\cI$ be a cofinite ideal of $\O_0$. There
exist $\gl_1,\dots,\gl_n\in\fv_\iC$ such that the algebra homomorphism
$$
\bar{\alpha}_n^*:\P\longrightarrow\;(\P/\cI_{\lambda_1}\cap\P)\otimes \cdots\otimes(\P/\cI_{\lambda_n}\cap\P)
$$
satisfies $\ker{\bar{\alpha}_n^*}\subset \cI.$
In particular, the $\cO_0$-module
$\cO_0/\cI$ is a subquotient of $(\cO_0/\cI_{\lambda_1})\otimes \cdots\otimes(\cO_0/\cI_{\lambda_n}).$
\end{enumerate}
\end{prop}

\begin{proof}
Assertion (a) follows from Corollary \ref{sec18} and \eqref{sec23}.
For the second assertion, we note that by Lemma \ref{sec25},
there exists a number $k\in\N$
such that $(\cM\cap\P)^{k+1}\subset \cI\cap\P\subset\cM$.
Fix a basis $\{X_1,\dots, X_N\}$ of $\fv$. Put $n\coleq kN$ and
$$
\gl_{kj + i}\coleq  X_j \qquad \textrm{for}\;\;\; 0 \leq j \leq N-1\;\; \textrm{and}\;\; 1 \leq i \leq k.
$$
By Lemma \ref{sec42}, the ideal $\ker{\bar{\alpha}_n^*}$ is equal to
$$
\{p\in\cM\cap\P\mid \,p(\dotvar ;\lambda_{j_1}\cdots\lambda_{j_l})
\in\cM,\,1\leq l\leq n,\,1\leq j_1<\cdots<j_j\leq n\}.
$$
Hence, if $p\in\ker{\bar{\alpha}_n^*}$, then
$$
p(0; X^\beta)=0,\quad |\beta|\leq k.
$$
This implies that $p\in (\cM\cap\P)^{k+1}$. Thus, $\ker{\bar{\alpha}_n^*}\subset \cI.$
In particular,
it follows that the $\cP$-module $\cP/(\cI \cap \cP)$ is a subquotient of
$(\cP/\cI_{\gl_1} \cap \cP) \otimes \cdots \otimes (\cP/\cI_{\gl_n} \cap \cP).$ The final assertion
now follows by application of Lemmas \ref{sec24} and \ref{sec25}.
\end{proof}

\subsection{Holomorphic families of continuous representations}
\label{s: holomf}
We retain our assumption that $\Omega$ is an open subset of the finite
dimensional complex linear space $\fv_\iC.$
\par
If $X$ is a locally compact Hausdorff space, and $V$ a locally convex (Hausdorff) space,
then by $C(X,V)$ we denote the space of $V$-valued continuous functions on $X$, equipped
with the topology of uniform convergence on compact subsets.
If $V$ is quasi-complete, then $C(X, V)$ is quasi-complete as well.

Let $V$ be a quasi-complete locally convex space.
The subspace $\cO(\Omega,V)$ of $V$-valued holomorphic functions on $\Omega$ is
closed in $C(\Omega,V)$ hence quasi-complete of its own right.
For our further considerations, it is important to note that
the algebraic tensor product $\cP(\fv_\iC)\otimes V$ is dense in $\O(\Omega,V),$
see Lemma \ref{l: density of P in O}.
 \par
             If $V,W$ are locally convex spaces, we write $\Hom(V,W)$
for the space of continuous complex linear maps $V \to W.$ This space, equipped
with the strong operator topology is locally convex again.
Moreover, if $V$ is barrelled and $W$ quasi-complete, then $\Hom(V,W)$ is quasi-complete as well, see appendix, text before Lemma \ref{l: continuity of composition}.
            \par
As usual, we write $\Endo(V)$ for $\Hom(V,V).$ Let $U$ be a third locally convex space.
Then the composition map
$$
\gb: \Hom(U,V) \times \Hom(V,W) \to \Hom(U,W),\quad (A,B) \mapsto B\after A
$$
is bilinear and separately continuous.
Moreover, if $V$ is barrelled, then by the principle of uniform boundedness, $\gb$
is continuous
relative to (i.e., when restricted to) sets of the form $\Hom(U,V) \times C,$ with $C \subset \Hom(V,W)$
compact; see
appendix, Lemma \ref{l: continuity of composition}.
Now assume that $U,V$ and $W$ are quasi-complete locally convex spaces.
If $U$ and $V$ are barrelled, then it follows from the material in the appendix,
see Lemma \ref{l: composition of holomorphic homomorphisms},
that the natural pointwise composition
defines a  bilinear map
$$
\gb_*: \cO(\Omega,\Hom(U,V)) \times \cO(\Omega, \Hom(V,W)) \to \cO(\Omega,\Hom(U,W)).
$$

\begin{definition}\label{defhfr}
\rm
A holomorphic family of continuous representations of $G$ over $\Omega$ is a pair $(\pi,V)$ such that
the following conditions are fulfilled:
\begin{enumerate}
\itema $V$ is a Fr\'echet space.
\itemb $\pi$ is a continuous map from $G$  to $\O(\Omega,\End{V})$ satisfying:
\begin{enumerate}
\item[(1)] $\pi(g_1g_2)=\pi(g_1)\pi(g_2),\quad g_1,g_2\in G,$\\
$\pi(e_G)= 1_\pi$, where $1_\pi(\mu) = \id_V$ for all $\mu \in \Omega;$
\item[(2)] for  every  $k\in K$,  the $\End{V}$-valued function $\pi(k)$ is constant
on $\Omega$.
\end{enumerate}
\end{enumerate}
A holomorphic family of smooth representations of $G$ over $\Omega$ is a family $(\pi,V)$ as above such
that in addition
\begin{equation}
\label{e: condition smooth family}
(g,\gl)\mapsto \pi(g)(\gl)v,\quad G \times \Omega \to V
\end{equation}
is smooth for every $v \in V.$
\end{definition}

\begin{rem}
\rm
Given $g \in  G$ and $\gl \in \Omega$ we agree to write $\pi_\gl(g)\coleq  \pi(g)(\gl).$
The condition that $V$ is Fr\'echet ensures that the principle of uniform boundedness is valid.
By application of  this principle it follows that the map
\begin{equation}
\label{e: holomorphic family written alternatively}
G \times \Omega \times V \to V,\qquad (g,\gl,v) \mapsto \pi_\gl(g)v
\end{equation}
is continuous and holomorphic in $\gl.$
More generally we could have given the definition of a holomorphic family of continuous representations under the weaker
assumption that $V$ be quasi-complete and barrelled. All results of the present section and the next are in fact valid under
this weaker assumption.
However, in \S \ref{s: parabolic induction} the assumption that $V$ is Fr\'echet
will really be needed.
\end{rem}

\begin{definition}
\rm
Let $(\pi,V_\pi)$ and $(\rho,V_\rho)$ be two holomorphic families of
continuous representations of $G$ over $\Omega$.
A holomorphic family of intertwining operators,
$T$, between $(\pi,V_\pi)$ and $(\rho,V_\rho)$
is an element of $\O(\Omega,\Hom(V_\pi,V_\rho))$ satisfying:
$$
T\pi(g)=\rho(g)T,\qquad g\in G.
$$
\end{definition}

\begin{definition}
\rm
The category $\HFG$ is defined as follows.
\begin{enumerate}
\item[(a)] The objects are the holomorphic families of continuous representations of $G$ over  $\Omega$.
\item[(b)] The morphisms are the holomorphic families of intertwining operators between two objects.
\end{enumerate}
For any $(\pi,V_\pi)\in\HFG$, the identity morphism is the holomorphic family $1_\pi$
of interwining operators defined by:
$$
1_\pi(\gl)\coleq \id_{V_\pi},\quad \gl\in\Omega.
$$
The composition $T'\after T$ of two (composable) morphisms
is given by pointwise composition: $(T'\after T)_\gl = T_\gl' \after T_\gl.$ It is  again a holomorphic family of intertwining
operators, by virtue of Lemma \ref{l: composition of holomorphic homomorphisms}.
\end{definition}
If $V$ and $W$ are locally convex spaces, and $E$ a finite dimensional complex linear space,
then the map $(A_1, A_2) \mapsto A_1 \otimes A_2$ induces a linear isomorphism
\begin{equation}\label{e: end-hom}
\End{E}\otimes \Hom(V,W)\simeq \Hom(E\otimes V,E\otimes W)
\end{equation}
which we shall use for identifying these spaces. Accordingly, if $V$ and $W$ are quasi-complete,
and $V$ barrelled, then
$$
\diffD^{(E)}: T \mapsto T^{(E)}
$$
defines a continuous linear map
$$
\cO(\Omega, \Hom(V,W)) \to \cO(\Omega, \Hom(E \otimes V, E \otimes W))
$$
(see Remark \ref{r: J is M}).

\begin{definition}
\rm
Let $E$ be a finite dimensional $\O_0$-module.
\begin{enumerate}
\item[(a)] For $(\pi,V)\in\HFG$, we define $\pi^{(E)}$ to be the continuous map
$$
\pi^{(E)} = \diffD^{(E)}\after \pi:\;\; g\mapsto\pi(g)^{(E)}, \quad
G \to \O(\Omega,\End{E\otimes V}).
$$
\item[(b)] For any morphism $T:(\pi,V_\pi)\to(\rho,V_\rho)$ of $\HFG$, we define
$$
T^{(E)}\coleq  \diffD^{(E)}(T)  \in \O(\Omega,\Hom(E\otimes V_\pi,E\otimes V_\rho)).
$$
\end{enumerate}
\end{definition}

\begin{prop}\label{p: the differentiation is a functor}
Let $E$ be a finite dimensional $\O_0$-module. Then $\diffD^{(E)}$ defines a functor from $\HFG$ to itself.
\end{prop}
\begin{proof}
Let $(\pi, V) \in \HFG.$ We will first show that
$(\pi^{(E)},E\otimes V)\in \HFG$. Let $\beta$ denote
the composition map in $\End{V}$.
It is a bilinear map, which preserves holomorphy on $\Omega,$
by Lemma \ref{l: composition of holomorphic homomorphisms}.
Moreover, by Corollary \ref{c: multiplicativity of diffD} we have, for all $g_1,g_2\in G$,
$$
\begin{array}{rcl}
\pi^{(E)}(g_1g_2) & = & \pi(g_1g_2)^{(E)}\\
& = & (\pi(g_1)\pi(g_2))^{(E)}\\
& = & \pi(g_1)^{(E)}\pi(g_2)^{(E)}\\
& = & \pi^{(E)}(g_1)\pi^{(E)}(g_2).
\end{array}
$$
Also, $\pi^{(E)}(e_G)=1_\pi^{(E)}=1_{\pi^{(E)}}$.
For  $k\in K$, the $\Endo(V)$-valued function $\pi(k)$ is a constant  on $\Omega,$ and therefore, so is
the $\Endo(E\otimes V)$-valued function $\pi^{(E)}(k) = \pi(k)^{(E)}.$
Hence (b.1) and (b.2) of Definition \ref{defhfr} are fulfilled.
            \par
It remains to show that, for any morphism $T:(\pi,V_\pi)\to (\rho,V_\rho)$
of $\HFG$, $T^{(E)}$ is a morphism of
$\HFG$ from $(\pi^{(E)}, E\otimes V_\pi)$
to $(\rho^{(E)},E\otimes V_\rho)$.
For this it suffices to show that, for $g\in G$,
$$
T^{(E)} \pi^{(E)}(g)=\rho^{(E)}(g)T^{(E)}.
$$
This follows from Corollary \ref{c: multiplicativity of diffD}, see also the first part of the present proof.
\end{proof}
The category $\HFG$ has a null object, $(0,\{0\})$, and one can define a biproduct in $\HFG$ as follows. Let $(\pi,V_\pi),\,(\rho,V_\rho)\in\HFG$. Set, for any $g\in G$, $\gl\in\Omega$ and $(v,w)\in V_\pi\oplus V_\rho$,
$$(\pi\oplus\rho)_\gl(g)(v,w)\coleq (\pi_\gl(g)v,\rho_\gl(g)w).$$
Then $(\pi\oplus\rho,V_\pi\oplus V_\rho)$ defines an object of $\HFG.$
                \par
We define the full subcategory $\HFGinfty$ of $\HFG$ by stipulating that the objects are the holomorphic families of smooth representations over $\Omega$ (the set of morphisms between objects in $\HFGinfty$ coincides with the set of morphisms between the objects viewed as objects
for the bigger category $\HFG.$ Likewise, the full subcategory
$\HFGinftyadm$ of $\HFGinfty$ consists of the objects
$(\pi, V)$ of $\HFGinfty$ with $V_K$ admissible.
    \par
If $(\pi, V) \in \HFG,$ then the identity morphism $1_\pi$ is constant as a $\End{V}$-valued function
on $\Omega.$ This enables us to define a particular subcategory.
                \par
\begin{definition}\label{d: sec59}
\rm
The subcategory $\HFGO$ of $\HFG$ is defined as follows.
\begin{enumerate}
\itema The objects are the objects of $\HFG.$
\itemb If $(\pi, V)$ and $(\rho,W)$ are objects of $\HFG,$ then the associated
collection of $\HFGO$-morphisms consists of all $\HFG$-morphisms $T$ in $\cO(\Omega, \Hom(V,W))$
which are constant as a function on $\Omega.$
\end{enumerate}
In a similar fashion, we define the subcategories
$\HFGOinfty$ and $\HFGOinftyadm$ of $\HFGinfty$ and $\HFGinftyadm$, respectively.
\end{definition}
Note that  $\HFGOinftyadm$ is a full subcategory of
$\HFGOinfty,$ which in turn is a full subcategory of
$\HFGO.$

\begin{rem}
\label{r: invariance under functor}
\rm
For every $E\in \FMod_{\cO_0}$, the functor $J^{(E)}: \HFG \to \HFG$
leaves all subcategories $\HFGinfty,$ $\HFGinftyadm,$ $\HFGO$, $\HFGOinfty$ and $\HFGOinftyadm$ invariant.
\end{rem}

\begin{lemma}
Let $\psi: E \to E'$ be a morphism in $\FMod_{\cO_0}$ and let $(\pi, V_\pi)$
be a holomorphic family of representations. Then $\psi \otimes 1_\pi$
intertwines the families $\pi^{(E)}$ and $\pi^{(E')}.$
\end{lemma}

\begin{proof}
If $f \in \cO(\Omega),$ then from the definitions it readily follows
that
$$
\psi \after f^{(E)}(\mu) = f^{(E')}(\mu) \after \psi,\quad \mu \in \Omega.
$$
From this and the identification $\End{E}\otimes\End{V_\pi}\simeq\End{E\otimes V_\pi}$, it follows that for all $g \in G$ the map $\psi \otimes 1_{\pi}: E \otimes V_\pi \to E'\otimes V_\pi$
intertwines $\pi_\mu^{(E)}(g) = \pi(g)^{(E)}(\mu)$ with $\pi_\mu^{(E')}(g) = \pi(g)^{(E')}(\mu).$
\end{proof}

The above lemma justifies the following definition.
\begin{definition}
\rm
Given any $(\pi,V_\pi)\in\obj(\HFG)$, we define the functor $X_\pi$ from $\FMod_{\O_0}$ to $\HFGO$ as follows.
\begin{enumerate}
\itema
For an object $E\in\FMod_{\O_0},$ the associated  object of $\HFGO$
is given by
$$
X_\pi(E)\coleq (\pi^{(E)},E\otimes V_\pi).
$$
\itemb
 For a morphism $\psi: E\rightarrow E'$ of $\FMod_{\O_0}$, the associated morphism of $\HFGO$
is given by
$$
\quad X_\pi(\psi)\coleq \psi\otimes 1_\pi:\;\; (\pi^{(E)},E\otimes V_\pi)\,\longrightarrow\, (\pi^{(E')},E'\otimes V_\pi).
$$
\end{enumerate}
\end{definition}

\begin{rem}
\rm
It is readily checked that $X_\pi$ is a functor.
Indeed, $X_\pi$ respects composition of morphisms, and $X_\pi(\id_E) = \id_{E\otimes V_\pi}= 1_{\pi^{(E)}}.$
\end{rem}

As $\FMod_{\cO_0}$ is an abelian category, we have the usual notion of finite direct sums and exact
sequences in $\FMod_{\cO_0}.$
\par
The category $\Vect$ of complex vector spaces is abelian.
If $T: (\pi, V) \to (\rho, W)$ is a morphism in $\HFGO,$ then there exists a unique
linear map $T_0: V \to W$ such that $T(\gl) = T_0$ for all $\gl \in \Omega.$ By abuse of notation
we will write $T$ for $T_0.$ We thus have a forgetful
functor $\HFGO \to \Vect.$
    \par
The category $\HFGO$ is not abelian. Nevertheless, we may use the
forgetful functor to define exact sequences
\begin{definition}\rm
A sequence $((\pi_k, V_k), T_k), p \leq k \leq q,$ in the category $\HFGO$,
 where $p, q \in \Z, p < q,$ will be called exact if its image under the
 forgetful functor $\HFGO \to \Vect$  is exact, i.e., the image of $T_{k-1}$ equals the kernel of $T_k$ for all $p < k \leq q.$
 \end{definition}

\begin{lemma}
\label{l: exactness properties X pi}
Let $(\pi, V_\pi) \in \obj(\HFG).$
Then the functor $X_\pi: \FMod_{\cO_0} \to \HFGO$
has the following properties.
\begin{enumerate}
\itema
It sends every short exact sequence $0 \to E \to E' \to E'' \to 0$ to a  similar short exact sequence in $\HFGO.$
\itemb
It sends every exact sequence of the form $0 \to E \to E'$ in $\FMod_{\cO_0}$ to an exact sequence of similar form in $\HFGO.$
\itemc
It sends every exact sequence of the form $E' \to E''\to 0$ in $\FMod_{\cO_0}$ to an exact sequence of similar form in $\HFGO.$
\itemd
It sends a direct sum of the form $E = E_1 \oplus E_2$ in
$\FMod_{\cO_0}$ to a similar direct sum in $\HFGO.$
\end{enumerate}
\end{lemma}

\begin{proof}
Let $\FDVect$ denote the (abelian) category of finite dimensional complex linear spaces.
Then we have a forgetful functor $\cF$ from $\FMod_{\cO_0}$ to $\FDVect.$  A sequence in $\FMod_{\cO_0}$ is exact if and only if its image under $\cF$ is exact in $\FDVect.$
According to the definition above, the forgetful functor $\cF': \HFGO \to \Vect$ has a similar property.
                \par
Given $U \in \Vect$ we define the functor $X_U: \FDVect \to \Vect$ by $X_U(E) = E \otimes U$
for an object $E$ of $\FDVect.$ A morphism $f: E \to E'$ in $\FMod_{\cO_0}$ is mapped to $X_U(f)\coleq f \otimes \id_U:
E \otimes U \to E' \otimes U.$ It is readily seen that the functor $X_U$ is exact, and has the obvious properties analogous to (a) - (d).
            \par
Since
$
\cF' \after X_\pi = X_{\cF'(\pi,V)} \after \cF,
$
assertions (a), (b) and (c) of the lemma follow.
For assertion (d) it remains to be shown that each natural embedding
${\rm i}_j: E_j \to E$, for $j =1,2,$
is mapped to an embedding
$X_\pi(\rmi_j)$ from $X_\pi(E_j)$ onto a closed subspace of
$X_\pi(E).$ Let $p_j: E \to E_j$ be the natural projection.
Then by exactness of the sequence $0 \to E_1 \to E \to E_2 \to 0$
it follows from the established assertion (a) that $0 \to X_\pi(E_1) \to X_\pi(E) \to X_\pi(E_2) \to 0$ is exact.
This implies that $X_\pi(\rmi_1)$ has closed image in $X_\pi(E).$
Likewise, $X_\pi(\rmi_2)$ is seen to have close image.
\end{proof}

\begin{rem}
In view of Remark \ref{r: invariance under functor}, Lemma 2.44
has an obvious generalization to objects from $\HFGinfty$ and
from $\HFGinftyadm.$
\end{rem}

\begin{prop}
\begin{enumerate}
\item[(a)]
Let $E_1,\dots,E_n\in \FMod_{\O_0}$ and set $E\coleq E_1\otimes\cdots\otimes E_n.$
Then there exist $N,\, k\in\N$ such that, for any object $(\pi,V_\pi)$ in $\HFG$  (resp.~$\HFGinfty,\,\HFGinftyadm$), the family
$(\pi^{(E)},E\otimes V_\pi)$ is a quotient of
$$
(\id_{\iC^N}\otimes \pi^{(\O_0/\cM^{k+1})},\C^N\otimes(\O_0/\cM^{k+1}\otimes V_\pi))
$$
in the category $\HFGO$ (resp.~$\HFGOinfty,\,\HFGOinftyadm$).
\item[(b)]
Let $E\in\FMod_{\O_0}$.
Then there exist  $\lambda_1,\dots,\lambda_n\in\fv_\iC$
such that, for any object $(\pi,V_\pi)$ in $\HFG$ (resp.~$\HFGinfty,\,\HFGinftyadm$), the family
$
(\pi^{(E)},E\otimes V_\pi)$ is a subquotient of
$$
(\pi^{((\O_0/\cI_{\lambda_1})\otimes\cdots\otimes(\O_0/\cI_{\lambda_n}))},
(\O_0/\cI_{\lambda_1})\otimes\cdots\otimes(\O_0/\cI_{\lambda_n})\otimes V_\pi)
$$
in the category $\HFGO$ (resp.~$\HFGOinfty,\,\HFGOinftyadm$).
\end{enumerate}
\end{prop}
\begin{proof}
This follows from Proposition \ref{sec43} combined with
 Lemma \ref{l: exactness properties X pi}.
\end{proof}

\subsection{Holomorphic families of admissible $(\g,K)$-modules}
For the purpose of this paper, it is convenient to introduce the following notion of holomorphic families of admissible $(\g,K)$-modules. Recall that $\Omega$ is an open subset of the finite dimensional
complex linear space $\fv_\iC.$
             \par
\begin{definition}\label{defhfm}
\rm
A holomorphic family of admissible $(\g,K)$-modules over $\Omega$ is a triple $(\pi_1,\pi_2,V)$
satisfying the following conditions.
\begin{enumerate}
\itema $V$ is a complex vector space.
\itemb $\pi_1$ is a map from $U(\g)\times \Omega$ to $\End{V}$ such that
\begin{enumerate}
\item[(1)] for each $\lambda\in\Omega$,  the map $\pi_1(\cdot,\lambda)$ is a Lie algebra homomorphism;
\item[(2)] for  each $u\in U(\g)$ and  every $v\in V,$ the vector subspace of $V$ generated by the $\pi_1(u,\lambda)v$, $\lambda\in\Omega$, is finite dimensional and the map $\lambda\mapsto \pi_1(u,\lambda)v$ is holomorphic from $\Omega$ into this subspace.
\end{enumerate}
\itemc $\pi_2$ is a Lie group homomorphism from $K$ to $\mathrm{GL}(V)$ such that:
\begin{enumerate}
\item[(1)]for each $v\in V,$ the vector subspace of $V$ generated by the $\pi_2(k)v$, $k\in K,$ is finite dimensional;
\item[(2)] for all $ v\in V$,  $\gl \in \Omega$, $u\in U(\g)$, $k\in K$ and  $X\in\k,$
$$
\begin{array}{rcl}
\pi_1(\mathrm{Ad}(k)u,\gl)v&=&\pi_2(k)\pi_1(u,\gl)v,\\
\frac{d}{dt}\left[\pi_2(\exp{(tX)})v\right]_{\vert t=0}&=&\pi_1(X,\gl)v;
\end{array}
$$
\end{enumerate}
\itemd for every  $\delta\in\hat{K}, $ the $K$-isotypic component $\pi_{2,\delta}$ of $\pi_2$ of type $\delta$ is finite dimensional.
\end{enumerate}
\end{definition}

\begin{rem}
\rm
Let $(\pi_1, \pi_2, V)$ be as in the above definition. Given $\gl \in \Omega$ we agree to write $\pi_{1\gl}$ for the map $\pi_1(\dotvar, \gl): U(\fg) \to \End{V}.$ Then
$(\pi_{1\gl}, \pi_2, V)$ is an admissible $(\fg, K)$-module.
\end{rem}

We also need the following notion of holomorphic family of intertwining operators.

\begin{definition}
\rm
Let $(\pi_1,\pi_2,V)$ and $(\rho_1,\rho_2,W)$ be two holomorphic families of admissible $(\g,K)$-modules over $\Omega$. A holomorphic family of intertwining operators between $(\pi_1,\pi_2,V)$ and $(\rho_1,\rho_2,W)$ is a function $T: \Omega \to \Hom(V,W)$ satisfying
the following conditions.
\begin{enumerate}
\itema $T(\lambda)\pi_1(u,\lambda)=\rho_1(u,\lambda)T(\lambda),$ for all $\lambda\in \Omega$ and $u\in U(\g)$;
\itemb $T(\lambda)\pi_2(k)=\rho_2(k)T(\lambda),$ for all $\lambda\in \Omega$ and $k\in K$;
\itemc  for any finite dimensional subspace $\tilde{V}$ of $V$, there exists a finite dimensional subspace $\tilde{W}$ of $W$ such that $T(\lambda)(\tilde V) \subset \tilde{W}$ for all $\lambda \in \Omega,$ and the associated function $\lambda \mapsto T(\gl)|_{\tilde V}$ belongs to $\cO(\Omega, \Hom(\tilde V, \tilde W)).$
\end{enumerate}
\end{definition}

Let $\HFgKadm$ denote the corresponding category of holomorphic families of admissible $(\fg,K)$-modules.

Let $(\pi, V)$ be an object in the category $\HFGinftyadm.$ Then for each  $\gl \in \Omega,$ $u \in U(\fg)$ and $v \in V$ we may define
$$
\pi(u)(\gl)v:= L_{u^\vee}(g\mapsto \pi(g)(\gl)v)|_{g =e}.
$$
We put $\pi_1(u,\gl) = \pi(u)(\gl)|_{V_K}$ and $\pi_2(k) = \pi(k)|_{V_K}.$

\begin{lemma}
Let $(\pi, V) \in \HFGinftyadm$ and
let $\pi_1, \pi_2$ be defined as above. Then
$$
(\pi, V)_K: =(\pi_1, \pi_2, V_K)
$$
is a holomorphic family
in $\HFgKadm.$ Moreover, $(\dotvar)_K: (\pi, V) \mapsto (\pi, V)_K$
defines a functor $\HFGinftyadm \to \HFgKadm.$
\end{lemma}

\begin{proof}
It follows from the smoothness of the map
(\ref{e: condition smooth family}), the continuity of the map $L_{u^\vee}: C^\infty(G) \to C^\infty(G)$  and the holomorphy with respect
to $\gl,$ that $\gf: \gl \mapsto \pi(u)(\gl)v$ defines a holomorphic
function $\Omega \to V.$ Let now $v \in V_K$ and let
$\vartheta_2 \subset \widehat K$ denote the set of $K$-types appearing in the $\pi(K)$-span of $v.$ Moreover, let $\vartheta_1$ be the set of
$K$-types which appear in the $\Ad(K)$-span of $u.$
Finally, let $\vartheta$ be the union of the sets of $K$-types
of $\gd_1 \otimes \gd_2,$ for $\gd_j \in \vartheta_j,$ $j =1,2.$
Then $\gf$ has image contained in the finite dimensional subspace
$V_\vartheta \subset V_K.$ It follows that $\gf$ is holomorphic
as a function $\Omega \to V_\vartheta.$
This shows that $(\pi_1, \pi_2 , V_K)$ satisfies condition (b)(2)
of Definition \ref{defhfm}. The other conditions of that definition
are pointwise in $\gl,$ and therefore consequences of
the standard theory of assigning the $(\fg, K)$-module of $K$-finite
vectors to an admissible smooth Fr\'echet representation, see
for instance \cite[Lemma 3.3.5]{wallbook1}.
The latter assignment is a functor from the category of admissible
smooth Fr\'echet representations to the category of Harish-Chandra modules. This implies that $(\pi, V) \mapsto (\pi,V)_K$ has functorial
properties which are pointwise in $\gl.$ This in turn is readily seen
to imply that $(\dotvar)_K$ is a functor as stated.
\end{proof}
\medbreak
We will now discuss the functor $J^{(E)}$ on the level
of holomorphic families of admissible $(\fg, K)$-modules.

 \begin{definition}
\rm
Let $E\in\FMod_{\cO_0}.$
For any $(\pi_1,\pi_2,V)\in\HFgKadm$ the triple
$(J^{(E)}\pi_1, J^{(E)}\pi_2, J^{(E)}V)$ (also denoted $(\pi_1^{(E)},\pi_2^{(E)},V^{(E)})$) is defined as follows.
\begin{enumerate}
\itema $J^{(E)}V\coleq E\otimes V.$
\itemb $J^{(E)}\pi_1$ is the map from $U(\g)\times \Omega$ to $\End{E\otimes V}$ given
by
$$
J^{(E)}\pi_1(u,\lambda)e\otimes v\coleq (\pi_1(u,\dotvar)v)^{(E)}(\lambda)e,
$$
for  $u\in U(\g),\,\lambda\in \Omega,\, e\in E,$ and $v\in V.$
\itemc $J^{(E)}\pi_2\coleq \id_E\otimes \pi_2.$
\end{enumerate}
\end{definition}
We now have functors $J^{(E)}$ on $\HFGinftyadm$ and $\HFgKadm.$
They are linked as follows.
\begin{lemma}{\ }
\begin{enumerate}
\itema  For any $E\in\FMod_{\cO_0}$, the assignment $J^{(E)}$ defines a functor from $\HFgKadm$ to itself.
\itemb The following diagram  commutes:
$$
\xymatrix{
\ar[d]_{J^{(E)}}\HFGinftyadm \ar[rr]^{(\dotvar)_{K}} &&\ar[d]^{J^{(E)}}\HFgKadm \\
\HFGinftyadm \ar[rr]_{(\dotvar)_{K}} && \;\;\HFgKadm\;\;.
}
$$
\end{enumerate}
\end{lemma}
\par
\begin{proof}
Assertion (a) follows by similar arguments as in the proof of Proposition \ref{p: the differentiation is a functor}.
For (b), assume that $(\pi,V)$ is a family in $\HFGinftyadm.$
Put $(\pi,V)_K = (\pi_1, \pi_2, V_K),$ then
$$
J^{(E)}((\pi, V)_K)
= (\pi_1^{(E)}, \pi_2^{(E)}, (V_K)^{(E)}).
$$
On the other hand, $J^{(E)}(\pi, V) = (\pi^{(E)}, V^{(E)})$ and
$$
(J^{(E)}(\pi, V))_K = ((\pi^{(E)})_1, (\pi^{(E)})_2, (V^{(E)})_K)
$$
Now $V^{(E)} = E\otimes V$ is equipped with the $K$-action on the second
component, so that $(V^{(E)} )_K = E\otimes V_K = (V_K)^{(E)}.$
Moreover, $\pi_2^{(E)}(k) = 1_E \otimes \pi_2(k) = 1_E \otimes \pi(k,\gl)|_{V_K} = \pi^{(E)}(k,\gl)|_{E \otimes V_K} = (\pi^{(E)})_2(k).$ It remains to establish the identity
\begin{equation}
\label{e: identity pi 1 E}
\pi_1^{(E)} = (\pi^{(E)})_1.
\end{equation}
Since both are representations
of $U(\fg)$ in $E \times V_K,$ it suffices to check the identity on a fixed element $X \in \fg.$ Fix $e\in E$ and $v \in V.$ We first observe that
$$
\pi_1^{(E)}(X)v = (\pi_1(X, \dotvar)v)^{(E)} = [\frac{d}{dt}\pi(\exp tX)(\dotvar)v]^{(E)}|_{t = 0}.
$$
In view of the natural identification $\Endo (E) \otimes V \simeq \Hom(E, E \otimes V),$ we note that
\begin{eqnarray*}
(\pi^{(E)})_1(X, \dotvar)v &=&  \frac{d}{dt}[\pi^{(E)}(\exp tX)(\dotvar)v]|_{t = 0}\\
&=&
\frac{d}{dt}[(\pi(\exp tX)(\dotvar)v)^{(E)}]|_{t = 0}.
\end{eqnarray*}
The identity (\ref{e: identity pi 1 E}) now follows by application of the lemma below.
\end{proof}

\begin{lemma}
Let $V$ be a quasi-complete locally convex space, and let
$\gf: \R \times \Omega \to V$ be a $C^1$-map which is holomorphic
in the second variable. Then
\begin{equation}
\label{e: d dt and (E)}
\frac{d}{dt} [\gf(t, \dotvar)^{(E)}] = [\frac{\partial \gf}{\partial t}(t, \dotvar)]^{(E)}.
\end{equation}
\end{lemma}

\begin{proof}
Let $S$ be the space of $C^1$-maps $\R \times \Omega \to V,$
equipped with the usual quasi-complete topology. Let
$S_0$ be the subspace consisting of functions in $S$ which are
holomorphic in the second variable. Then $S_0$ is closed in $S,$ hence
quasi-complete. The identity (\ref{e: d dt and (E)})
at $(t,\gl)$ can be viewed as an identity of continuous linear functionals on $S_0.$ Hence, it suffices
to check the identity on the dense subspace $C^1(\R) \otimes \cO(\Omega) \otimes V.$ This amounts to checking whether
$$
(\frac{d}{dt} \otimes I \otimes I)
\after
(I \otimes J^{(E)} \otimes I)
=
(I \otimes J^{(E)} \otimes I)
\after
(\frac{d}{dt} \otimes I \otimes I).
$$
The latter is obvious.
\end{proof}

\subsection{Parabolic induction}\label{s: parabolic induction}
Let $\fg = \fk \oplus \fp$ be a Cartan decomposition associated with the maximal compact
subgroup $K$ and let $\Cartan$ be the associated involution of $G.$ Let $\fa \subset \fp$
be a maximal abelian subspace, and let $A = \exp \fa.$ Let $\cP(A)$ denote the collection
of parabolic subgroups of $G$ containing $A.$ Let $\cL(A)$ denote the collection
of $\Cartan$-stable Levi components of parabolic subgroups from $\cP(A).$
                   \par
Let $\fv$ be a finite dimensional real linear space, and $\Omega$ an open subset
of its complexification.
For $L \in \cL(A)$ we denote by $\HFL$ the category defined as in
Definition \ref{defhfr},
with
the group $L$ in place of $G.$
\par
The parabolic induction functor from $\HFLinfty$ to $\HFGinfty$ is defined as follows.
Let $(\xi,V_\xi)\in\HFLinfty$  be a holomorphic family of smooth representations of $L$
defined over $\Omega.$ Denote by $\bar{\pi}_{P,\xi_\gl}$ the right regular representation
of $G$ on $C^\infty(G:P:\xi_\gl),$
where
$$
\begin{array}{rl}
C^\infty(G:P:\xi_\gl)\coleq \{\psi\in C^\infty(G,V_\xi): & \psi(nmg)=\xi_\gl(m)\psi(g),\\
& (g,m,n)\in G\times L\times N_P\}.
\end{array}
$$
Let
\begin{equation}
\label{e: defi CinftyKxi}
\begin{array}{rl}
C^\infty(K:\xi)\coleq \{\psi\in C^\infty(K,V_\xi): & \psi(mk)=\xi(m)\psi(k),\\
 & (k,m)\in K\times (K\cap L)\}.
\end{array}
\end{equation}
Restriction of functions to $K$ induces a continuous linear isomorphism
between these spaces.
Let $\pi_{P,\xi_\gl}$ denote the representation of $G$ on $C^\infty(K:\xi)$,
obtained from $\bar{\pi}_{P,\xi_\gl}$ by transfer of structure,
and set, for $g\in G$, $\gl\in \fv_\iC$,
$$
\pi_{P,\xi,\gl}(g)\coleq \pi_{P,\xi_\gl}(g).
$$
Then it is readily seen that $(\pi_{P,\xi},C^\infty(K:\xi))$ belongs to
$\HFGinfty.$  Here, the information that $V_\xi$ is
a Fr\'echet space is needed to conclude that $C^\infty(K\col \xi)$ is
Fr\'echet, in particular barrelled.

Let $W_1,\,W_2$ and $W_3$ be quasi-complete locally convex spaces,
and let $\alpha$ be a continuous linear map from $W_2$ to $W_3$. Then the map
$$
L_\alpha: \; \varphi\longmapsto \alpha\circ\varphi, \;\;\;\Hom(W_1,W_2)\longrightarrow \Hom(W_1,W_3)
$$
is continuous linear. It readily follows that the map
$$
\id_{\cO(\Omega)}\, \hatotimes \,L_\ga: \;\;\cO(\Omega, \Hom(W_1, W_2)) \longrightarrow
\cO(\Omega, \Hom(W_1, W_3))
$$
given by $f \mapsto L_\ga \after f$
is continuous linear. The notation for this map is explained by the fact that
it may be viewed as the unique
continuous linear extension of $\id_{\cO(\Omega)} \otimes L_\ga.$

\begin{lemma}
Let $L_\ga$ be as above and let $E$ be a finite dimensional $\cO_0$-module.
Then we have the following identity of maps
$$
\cO(\Omega, \Hom(W_1,W_2)) \too \cO(\Omega, \End{E} \otimes \Hom(W_1, W_3)):
$$
$$
\diffD^{(E)}\circ (\id_{\O(\Omega)}\, \hatotimes\, L_\alpha)=
(\id_{\O(\Omega)}\,\hatotimes \,\id_{\End{E}}\,\hatotimes\, L_\alpha)\circ {\diffD}^{(E)}.
$$
\end{lemma}
\begin{proof}
By continuity of the expressions on both sides of the identity
it suffices to prove the identity
on the dense subspace $\O(\Omega)\otimes \Hom(W_1,W_2)$ of $\O(\Omega,\Hom(W_1,W_2)).$
But then the identity becomes obvious, in view of Remark \ref{r: J is M},  last line.
\end{proof}

The following result may be phrased as `derivation commutes with induction'.

\begin{prop}
Let $E$ be a finite dimensional $\O_0$-module, $P$ a parabolic subgroup with Levi
component $L$, and $(\xi,V_\xi)\in\HFL^\infty$.
Then
$$\pi_{P,\xi}^{(E)}=\pi_{P,\xi^{(E)}}.$$
\end{prop}

\begin{proof}
Let $g\in G$ and $k\in K$. Put $\alpha(k)\coleq \ev_k:C^\infty(K:\xi)\to V_\xi$.
Let
$$L_{\alpha(k)}:\End{C^\infty(K:\xi)}\to \Hom(C^\infty(K:\xi),V_\xi)$$
be defined as in the previous lemma. Accordingly,
\begin{equation}\label{comm0}
(\id_{\O(\Omega)}\hatotimes\id_E\hatotimes L_{\alpha(k)})(\pi_{P,\xi}^{(E)}(g)) =
\left((\id_{\O(\Omega)}\hatotimes L_{\alpha(k)})(\pi_{P,\xi}(g))\right)^{(E)}.
\end{equation}
Write $g=np(g)\kappa(g)$ uniquely via the decomposition $G=N_PA_P(M_P\cap\exp\p)K$,
where $n\in N_P$, $p(g)\in A_P(M\cap\exp\p)$ and $\kappa(g)\in K$. We
then have the following identity:
$$
(\id_{\O(\Omega)}\hatotimes L_{\alpha(k)})(\pi_{P,\xi}(g))
             = \xi(p(kg))\circ L_{\alpha(\kappa(kg))}.
$$
Hence, it follows from \eqref{comm0} that
\begin{eqnarray*}
\lefteqn{
\!\!\!\!\!\!\!\!\!\!\!\!\!\!
(\id_{\O(\Omega)}\hatotimes\id_E\hatotimes
            L_{\alpha(k)})(\pi_{P,\xi}^{(E)}(g))
            \;\;\;\;\;\;\;\;\;\;\;\;\;\; \;\;\;\;\;\;\;\;\;\;\;\;\;\;}\\
&=&
\left(\xi(p(kg))\circ L_{\alpha(\kappa(kg))}\right)^{(E)}\\
&=&
\left(\xi^{(E)}(p(kg))\circ L_{\alpha(\kappa(kg))}\right)\\
&=&
(\id_{\O(\Omega)}\hatotimes\id_E\hatotimes L_{\alpha(k)})(\pi_{P,\xi^{(E)}}(g))
\end{eqnarray*}
and the statement follows.
\end{proof}

\section{The Arthur-Campoli relations}\label{s: Arthur's PW}
Given a parabolic subgroup $P\in\cP(A)$ we denote its Langlands decomposition by $P=M_PA_PN_P.$

Let $(\xi, V_\xi)$ be a smooth, irreducible and admissible
Fr\'echet representation of $M_P.$ Given $\gl \in \faPdc$ we denote by $\xi \otimes \gl$ the smooth representation of $L_P: = M_P A_P$ in $V_\xi$ defined by
$$
\xi \otimes \gl\, (ma) = a^{\gl + \rho_P} \xi(m).
$$
As usual, here $\rho_P \in \faPd$ is defined by $\rho_P =\half
\tr(\ad(\dotvar)|_{\fn_P}).$
The associated induced representation $\bar\pi_{P,\xi \otimes \gl}$ of $G$
in $C^\infty(G\col P\col \xi \otimes \gl)$ is defined as in
Section \ref{s: parabolic induction}, with $\fv = \fa^*, \Omega = \fadc$
and $\xi_\gl = \xi \otimes \gl.$ The family of these representations
may be viewed as a holomorphic family $\pi_{P, \xi\otimes(\dotvar)}$ of smooth representations of $G$
over $\Omega = \fadc$ on the fixed space $C^\infty(K\col \xi),$ defined in (\ref{e: defi CinftyKxi}). As in the mentioned
section, we agree to write $\pi_{P, \xi, \gl} = \pi_{P, \xi \otimes \gl}.$ If $P$ is a minimal parabolic subgroup, the representations
$\pi_{P,\xi, \gl}$ just defined are called representations
of the smooth minimal principal series of $G.$

Let now $P \in \cP(A)$ be arbitrary again. Then $M_P$ is a group
of the Harish-Chandra class, with maximal compact subgroup
$K_P:= M_P \cap K.$
Two continuous admissible $M_P$-representations of finite length in a quasi-complete locally convex space are said to be infinitesimally equivalent if their Harish-Chandra
modules are equivalent as $(\fm_P, K_P)$-modules.
For each equivalence class $\omega$ of irreducible unitary representations of $M_P,$ we fix a smooth admissible Fr\'echet representation $\xi = \xi_\omega$ which is infinitesimally equivalent to a representation of class $\omega,$ and which
is topologically equivalent to a closed subrepresentation of a representation of the smooth minimal principal series
of $M_P.$ Indeed, this is possible by the subrepresentation
theorem for the group $M_P.$ The set of all these chosen representations $\xi_\omega$ is denoted by $\dMP.$ Thus,
$\omega \mapsto \xi_\omega$ defines a bijection from the
set of equivalence classes of irreducible unitary representations
of $M_P$ onto $\dMP.$

\begin{rem}
\rm
In view of the theory of the Casselman--Wallach globalization functor, see \cite[Section 11]{wallbook2}, the representation
$\xi_\omega$ is a smooth Fr\'echet globalization of moderate growth of the Harish-Chandra module of any representative of $\omega.$
This characterization makes the choice of $\dMP$ more natural,
but will not be needed in the present paper.
\end{rem}

We denote by $\dMPds$ the subset of $\dMP$ consisting of the
representations $\xi_\omega,$ with $\omega$ a discrete series representation.
In particular, if $P \in \cP(A)$ is minimal, then $M_P$ equals the centralizer $M$ of $\fa$ in $K,$ and $M_{ds}^\wedge=\dM$. For each $(\xi, V_\xi)\in \dMPds$ we put
\begin{equation}\label{e: defi SSPxi}
\SS(P,\xi)\coleq \End{C^\infty(K:\xi)}_{K\times K},
\end{equation}
and we define the algebraic direct sum of linear spaces
\begin{equation}\label{e: defi cSP}
\cS(P)\coleq \oplus_{\xi \in \hatMPds} \cS(P, \xi).
\end{equation}

\subsection{The Arthur-Campoli relations}
Fix a minimal parabolic subgroup $P_0$ in $\cP(A)$ and let $P_0=MAN_0$ be
its Langlands decomposition.
In \cite[III, \S 4]{arthurpw}, Arthur defines a Paley-Wiener space involving all minimal parabolic subgroups containing $A.$ This definition is given in terms
of on the one hand Paley-Wiener growth conditions and on the other the so-called
Arthur-Campoli conditions. In  \cite[Thm.~3.6]{BSpwspace} it is shown
that the Arthur Paley-Wiener space is isomorphic to one defined in terms of the single minimal
parabolic subgroup $P_0.$
We shall now describe the Arthur-Campoli relations in the context
of the operator valued Fourier transform $f \mapsto \hat f(P_0).$
                     \par
For $f \in C^\infty_c(G,K)$ and $\xi \in \dM,$
the Fourier transform $\hat f(P_0, \xi) \in \cO(\fadc)\otimes \SS(P_0, \xi)$ is defined by
\begin{equation}
\label{e: defi Fourier}
\hat{f}(P_0,\xi,\gl)\coleq \int_G f(x)\pi_{P_0,\xi,\lambda}(x)\,dx,\quad \gl\in\fa_\iC^*.
\end{equation}
Then $f \mapsto \hat f(P_0)$ maps $C_c^\infty(G,K)$ into $\cO(\fa_\iC^*) \otimes \cS(P_0).$
                   \par
We define $\data$ to be the set of $4$-tuples of the form $(\xi,\psi,\lambda,u)$
with $\xi \in \dM,$
$\psi \in \cS(P_0, \xi)^*_{K\times K}$,
 $\gl \in \fadc$ and $u \in S(\fa^*).$
An Arthur-Campoli sequence in $\data$ is defined to be a finite family
$(\xi_i, \psi_i, \lambda_i, u_i)$
in $\data$ such that
$$
\sum_i\langle\pi_{P_0,\xi_i,\lambda_i;u_i}(x),\psi_i\rangle=0, \qquad x \in G.
$$
By integration over $x$ it follows that this condition is equivalent to the condition
that
\begin{equation}\label{e: acs}
\sum_i \langle \hat f(P_0, \xi_i, \gl_i;u_i), \psi_i \rangle = 0,
\qquad  f \in C_c^\infty(G,K).
\end{equation}

\begin{definition}
\label{d: AC relations}
\rm
A function  $\varphi \in \O(\a_\iC^*)\otimes\SS(P_0)$ is said to
satisfy the Arthur-Campoli relations if
$$
\sum_i\langle\varphi_{\xi_i,\lambda_i;u_i},\psi_i\rangle=0
$$
 for any Arthur-Campoli sequence $(\xi_i,\psi_i,\lambda_i,u_i)$ in $\data.$
\end{definition}

\subsection{Reformulation of the Arthur-Campoli relations}
In the following, $\O_0$ denotes the ring of germs at $0$ of
holomorphic functions defined
on a neighborhood of $0$ in $\a_\iC^*.$\par
Let
$E$ be a finite dimensional $\O_0$-module and let  $\Xi \subset \dM$
and $\Lambda \subset \fa_\iC^*$ be
finite sets.
We define the representation $\pi_{E,\Xi,\Lambda}$ of $G$ by
\begin{equation}
\label{e: defi pi E Xi gL}
\pi_{E,\Xi,\Lambda}
= \oplus_{(\xi, \gl) \in \Xi \times \gL}\; \pi_{P_0, \xi, \gl}^{(E)}.
\end{equation}
Note that this representation is admissible and of finite length.
Its underlying space is given by
\begin{equation}
\label{e: defi V E Xi gl}
V_{E, \Xi, \gL}
= \oplus_{(\xi, \gl)\in \Xi \times \gL} \;\;  E \otimes C^\infty(K\col \xi).
\end{equation}
For each element $\gf \in \cO(\fa_\iC^*) \otimes \cS(P_0)$
we define the $K\times K$-finite endomorphism
$\varphi_{E,\Xi,\Lambda}$ of $V_{E, \Xi, \gL}$
by taking the similar direct sum
\begin{equation}
\label{e: defi gf E Xi gL}
\gf_{E, \Xi, \gL} \coleq \oplus_{(\xi,\gl) \in \Xi \times \gL} \;\; \gf^{(E)}_{\xi}(\gl).
\end{equation}
We note that $\pi_{E, \Xi, \gL}(C_c^\infty(G,K))$
is a subset of the space $\Endo(V_{E,\Xi, \gL})$
and agree to write
$\pi_{E, \Xi, \gL}(C_c^\infty(G,K))^\perp$ for its annihilator in the space
$$
\Endo(V_{E, \Xi, \gL})^*_{K\times K}.
$$
\begin{prop}\label{p: reformulation ac}
Let $\phi\in \O(\a_\iC^*)\otimes\SS(P_0).$
Then the following conditions are equivalent:
\begin{enumerate}
\itema
$\phi$ satisfies the Arthur-Campoli relations;
\itemb
for every finite dimensional $\O_0$-module $E,$ every pair of finite sets
$\Xi \subset \dM,\,$ $\Lambda \subset \a_\iC^*$ and all
$\Psi\in\pi_{E,\Xi,\Lambda}(C^\infty_c(G,K))^\perp,$
$$
\langle\phi_{E,\Xi,\Lambda},\Psi\rangle=0.
$$
\end{enumerate}
\end{prop}
We prove the result through a number of lemmas.
In the following results a complication is caused by the circumstance that
$\Endo(V_{E,\Xi,\gL})_{K\times K}$
is not the direct sum of the spaces
$\Endo(V_{E, \xi, \gl})_{K\times K},$ but rather that of
the spaces
$\Hom(V_{E, \xi_1, \gl_1}, V_{E, \xi_2, \gl_2})_{K\times K},$
for $\xi_1, \xi_2 \in \Xi$ and $\gl_1, \gl_2 \in \gL.$
For $(\xi, \gl) \in \Xi \times \gL,$ let
$$
\inj_{\xi, \gl}: \;  \Endo(V_{E, \xi, \gl})_{K\times K} \longrightarrow \Endo(V_{E, \Xi, \gL})_{K\times K}
$$
denote the associated embedding, and let
$$
\pr_{\xi, \gl}:\;\; \Endo(V_{E, \Xi, \gL})_{K\times K} \longrightarrow \Endo(V_{E, \xi, \gl})_{K\times K}
$$
denote the associated projection map.

\begin{lemma}
\label{l: ac data associated with Psi}
Let $E, \Xi, \gL$ be as above. Then for each $\Psi \in \Endo(V_{E, \Xi, \gL})^*_{K\times K}$
there exists a finite sequence
$(\xi_i, \psi_i, \gl_i, u_i)$ in $\data$ such that  for all $\gf \in \cO(\fa^*_\iC) \otimes \cS(P_0),$
\begin{equation}
\label{e: ac data associated with Psi}
\inp{\gf_{E, \Xi, \gL}}{\Psi} = \sum_{i} \inp{\gf_{\xi_i, \gl_i ; u_i}}{\psi_i}.
\end{equation}
\end{lemma}

\begin{proof}
Put
$$
\cS(P_0, \Xi, \gL) \coleq
\oplus_{(\xi,\gl)\in\Xi\times\gL}\cS(P_0,\xi).
$$
We observe that $\Endo(V_{E, \Xi, \gL})^*_{K\times K}$ may be viewed as the direct sum of the $K \times K$-submodule $\Endo(E)^* \otimes \cS(P_0, \Xi, \gL)^*_{K\times K}$
and a unique
$K\times K$-submodule $\End{E}^*\otimes \cT$ (consisting of the  `cross terms').
Since every element
of $\End{E}^*\otimes \cT$ annihilates $\gf_{E, \Xi, \gL},$ we may assume
that $\Psi \in \End{E}^*\otimes\cS(P_0, \Xi, \gL)^*_{K\times K}.$
By linearity, we may then reduce to the situation that $\Xi = \{\xi\}$ and $\gL = \{\gl\}.$
Again by linearity we may assume that $\Psi$ is of the form $\eta \otimes \psi,$ with
$\eta \in \Endo(E)^*$ and
$\psi \in \cS(P_0, \xi)^*_{K\times K}.$
Let $u \in S(\fa^*)$ be associated with
$\eta$ as in Lemma \ref{l: dual of End E}. Then for all $\gf \in \cO(\fa_\iC^*) \otimes \cS(P_0),$
$$
\inp{\gf_{E, \xi, \gl}}{\Psi} = \inp{\gf_{\xi, \gl}^{(E)}}{\eta \otimes \psi} = \inp{\gf_{\xi, \gl ; u}}{\psi}.
$$
This  finishes the proof.
\end{proof}

\begin{lemma}
\label{l: E associated with ac data}
Let $(\xi_i, \psi_i, \gl_i, u_i)$ be a finite sequence in $\data.$ Then there exists a finite dimensional
$\cO_0$-module $E$, and finite sets $\Xi \subset \dM,$ $\gL \subset \fadc$ and an element
$\Psi \in \Endo(V_{E, \Xi, \gL})^*_{K\times K}$ such that for all
$\gf \in \cO(\fa^*_\iC) \otimes \cS(P_0),$
\begin{equation}
\label{e: E associated with ac data}
 \sum_{i} \inp{\gf_{\xi_i, \gl_i ; u_i}}{\psi_i} = \inp{\gf_{E, \Xi, \gL}}{\Psi}.
\end{equation}
\end{lemma}

\begin{proof}
Let $E$ and $\eta_i \in \Endo(E)^*$ be associated to the finite sequence
$u_i$ as in Lemma \ref{l: E associated with u}.
Then $\Psi_i = \eta_i \otimes \psi_i$ belongs to
$\Endo(E)^* \otimes \cS(P_0, \xi_i)^*_{K\times K} \simeq \Endo(V_{E, \xi_i,\gl_i})_{K\times K}^*.$
Let $\Xi$ be the finite set of all $\xi_i$ and let $\gL$ be the finite set of all $\gl_i.$
Let $\pr_{\xi_i, \gl_i}$ be defined as above.
We define $\Psi \in \Endo(V_{E, \Xi, \gL})^*_{K \times K}$ by
$$
\Psi =\sum_i  \pr_{\xi_i, \gl_i}^* ( \eta_i \otimes \psi_i ).
$$
Then for all $\gf \in \cO(\fadc) \otimes \cS(P_0)$ we have
\begin{eqnarray*}
\sum_i \inp{\gf_{\xi_i,\gl_i ; u_i}}{\psi_i}
       &=& \sum_i \inp{\gf^{(E)}_{\xi_i,\gl_i}}{\eta_i \otimes \psi_i}\\
& =& \sum_i \inp{\pr_{\xi_i, \gl_i}\gf_{E, \Xi, \gL}}{\eta_i \otimes \psi_i} \\
& = & \inp{\gf_{E, \Xi, \gL}}{\Psi}.
\end{eqnarray*}
\end{proof}

\begin{proof}[Proof of Proposition \ref{p: reformulation ac}]
Let $\phi$ be as stated and assume (a).
Let $E,\,\Xi,\, \Lambda$ and $\Psi$ be as asserted in (b).
In particular, $\Psi \in \Endo(V_{E, \Xi, \gL})^*_{K\times K}.$
Let $(\xi_i, \psi_i, \gl_i, u_i)$ be a sequence in $\data,$
associated to $\Psi$ as in Lemma \ref{l: ac data associated with Psi}.
Then using the relation (\ref{e: ac data associated with Psi}) with
$\gf = \hat f(P_0)$ for $f \in C_c^\infty(G,K),$
we see that $(\xi_i, \psi_i, \gl_i, u_i)$ is an Arthur-Campoli
sequence (see \eqref{e: acs}).
Hence,
$$
\inp{\phi_{E, \Xi, \gL}}{\Psi} = \sum_i \inp{ \phi_{\xi_i,\gl_i; u_i}}{\psi_i} = 0.
$$
We have proved (b).
                   \par
Conversely, assume (b) and let $(\xi_i, \psi_i, \gl_i, u_i)$ be an
Arthur-Campoli sequence in $\data.$
Let $E, \Xi, \gL, \Psi$ be associated with this sequence as in
 Lemma \ref{l: E associated with ac data}.
Then it follows from
\eqref{e: E associated with ac data} with $\gf = \hat f(P_0)$ for
$f \in C_c^\infty(G,K)$ that $\Psi$ belongs
to $\pi_{E, \Xi, \gL}(C_c^\infty(G,K))^\perp.$ This implies that
$$
\sum_i \inp{\phi_{\xi_i,\gl_i;u_i}}{\psi_i} = \inp{\phi_{E, \Xi, \gL}}{\Psi} = 0.
$$
Hence (a).
\end{proof}

\section{Delorme's interwining conditions}\label{s: Delorme's PW}

\subsection{Successive derivatives}\label{s: succder}
Let $\v$ be a finite dimensional vector space over $\R,$
 let $\Omega \subset \fv_\iC^*$ be an open subset and let
$V$ be a quasi-complete locally convex space.
 Following Delorme \cite{delormepw}, we define, for
$\Phi \in \cO(\Omega, \End{V})$ and $\eta \in \fv_\iC^*,$
 the holomorphic function  $\Phi^{(\eta)}: \Omega \to \End{V\oplus V}$
by
$$
\Phi^{(\eta)}(\lambda)(v_1,v_2)\coleq
\left(\Phi(\lambda)v_1+\frac{d}{dz}(\Phi(\lambda+z\eta)v_2)_{|z=0}\;,\;\Phi(\lambda)v_2\right),
$$
for $\gl \in \Omega$ and $v_1,v_2 \in V.$
Still following \cite{delormepw}, we define,
for any finite sequence $\eta=(\eta_1,\dots,\eta_N)$ in $\v_\iC^*$,
the iterated derivative
$$
\Phi^{(\eta)} \coleq  (\cdots (\Phi^{(\eta_N)})^{(\eta_{N-1})}\cdots)^{(\eta_1)}
$$
of $\Phi$ along $\eta.$
Then $\Phi^{(\eta)}$ is a holomorphic function on $\Omega$
with values in $\End{V^{(\eta)}},$ where $V^{(\eta)}$ denotes
the direct sum of $2^N$ copies of $V.$
                    \par
Now assume that $V$ is Fr\'echet (or more generally, barrelled).
If $\pi$ is a holomorphic family of continuous representations
of $G$ in $V$ over the parameter set $\fv_\iC^*$ then it follows
by application of the methods of \cite{delormepw} that for each $\gl\in \Omega,$
$$
\pi^{(\eta)}_\gl(x): = \pi(x)^{(\eta)}(\gl), \qquad x \in G,
$$
defines a continuous representation of $G$ in $V^{(\eta)}.$

\subsection{The intertwining conditions}
\label{s: the intertwining conditions}
Let $\cD$ be the set of $4$-tuples $\delta=(P,\xi,\lambda,\eta)$, with
$P\in\cP(A),$
$\xi\in\dMPds$, $\lambda\in\a_{P\iC}^*$ and $\eta$ a finite
sequence in $\a_{P\iC}^*.$
Given $\gd\in\cD,$  we define the representation $\pi_\gd$ of
$G$ in $V_{\pi_\gd}\coleq C^{\infty}(K\col \xi)^{(\eta)}$ by
$$
\pi_\gd \coleq \pi_{P, \xi, \gl}^{(\eta)}
$$
For $P \in \cP(A)$ we define the space
$
\holspace_P\coleq \cO(\faPdc) \otimes \cS(P),
$
with $\cS(P)$ defined as in (\ref{e: defi cSP}).
Furthermore, we put
\begin{equation}
\label{e: holspace}
\holspace:= \oplus_{P\in\cP(A)} \;\; \holspace_P.
\end{equation}

Given $\gf\in \holspace$ and $\gd = (P, \xi, \gl, \eta) \in \cD,$ we define
$\gf_\gd \in \Endo(C^\infty(K\col \xi)^{(\eta)})$ in a similar fashion as $\pi_\gd,$  by
$$
\gf_\gd\coleq \gf_{P, \xi}^{(\eta)}(\gl).
$$

Finally, given a sequence $\gd = (\gd_1, \ldots, \gd_N)$
of data from $\cD,$ we define
$\pi_{\gd}\coleq  \pi_{\gd_1} \oplus \cdots \oplus \pi_{\gd_N},$ $V_{\pi_{\gd}}
\coleq  V_{\pi_{\gd_1}} \oplus \cdots \oplus V_{\pi_{\gd_N}}$ and
$\gf_\gd\coleq  \gf_{\gd_1} \oplus \cdots \oplus \gf_{\gd_N}.$
\par
\begin{definition}\label{d: Delorme's intertwining conditions}
\rm
We say that a function $\gf \in \holspace$
satisfies Delorme's intertwining conditions (see \cite[Definition 3 (4.4)]{delormepw}) if
\vspace{3pt}
\begin{enumerate}
\itema
for every $N \in \Z_+$ and each  $\gd \in \cD^N$  the function $\gf_\gd$ preserves all invariant subspaces of $\pi_{\gd};$
\itemb
for all $N_1, N_2 \in \Z_+,$ all $\gd_1 \in \cD^{N_1}$ and $\gd_2 \in\cD^{N_2},$ and any two sequences of closed invariant subspaces $U_j \subset V_j$ for $\pi_{\gd_j},$ the induced maps $\bar \gf_{\gd_j} \in \Endo(V_j/U_j)$ are intertwined by all intertwining operators $T: V_1/U_1 \to V_2/U_2.$
\end{enumerate}
The space of functions $\gf \in \holspace$ satisfying (a) and (b)  is denoted by $\holspace(\cD).$
\end{definition}

\subsection{A simplification of the intertwining conditions}
In this section we will show  that condition (b) of Definition
\ref{d: Delorme's intertwining conditions} is in fact a consequence of condition (a) of the same definition.
\begin{lemma}\label{l: one Delorme condition sufficient}
Let $\gf \in \holspace$.
Then $\gf \in \holspace(\cD)$ if and only if $\gf$ satisfies  condition (a)
of Definition \ref{d: Delorme's intertwining conditions}.
\end{lemma}
\begin{proof}
Assume that $\gf$ satisfies condition (a) of the mentioned definition. Then we must show that $\gf$ satisfies condition (b) as well.
              \par
Let $\delta_j,\, \pi_{\gd_j}, \,U_j,\, V_j$ be as in (b), for $j =1,2.$
Let $T: V_1/U_1 \to V_2/U_2$ be an intertwining operator.
Let $p_j: V_j \to V_j/U_j$ denote the canonical
projection, for $j =1,2,$ and let $p \coleq  p_1 \oplus p_2.$
Since $T$ is equivariant, its graph $W$ is an invariant subspace of $V_1/U_1 \oplus V_2/U_2.$
Hence, $p^{-1}(W)$ is an invariant subspace of $V_1 \oplus V_2.$
Since $\varphi$ satisfies (a), it follows that $p^{-1}(W)$ is a $\varphi_{\delta_1} \oplus \varphi_{\delta_2}$-invariant subspace of $V_1 \oplus V_2.$ This in turn is easily seen to imply
that $T\circ \bar{\varphi}_{\delta_1} = \bar{\varphi}_{\delta_2}\circ T.$
\end{proof}

\subsection{Reduction to a single minimal parabolic subgroup}
In this section, we will show that the space $\holspace(\cD)$ of functions
satisfying Delorme's intertwining conditions is naturally isomorphic to a
space of functions defined in terms of just the minimal parabolic subgroup $P_0.$ We denote by $\cD_{P_0}$ the set of $4$-tuples $(P,\xi,\lambda,\eta)$
in $\cD$ for which $P=P_0$.

\begin{definition}
\label{d: restricted intertwining conditions}
\rm
We say that a function $\gf \in \holspace$
satisfies Delorme's intertwining conditions associated with $P_0$
if conditions (a) and (b) of
Definition \ref{d: Delorme's intertwining conditions} are valid with everywhere $\cD$ replaced by $\cD_{P_0}.$

The space of functions $\gf \in \holspace$ satisfying all
such intertwining conditions is denoted by $\cF(\cD_{P_0}).$
\end{definition}

The following analogue
of Lemma \ref{l: one Delorme condition sufficient} is now valid,
with essentially the same proof.

\begin{lemma}
\label{l: one restricted intertwining condition suffices}
Let $\gf\in \holspace.$ Then $\gf \in \cF(\cD_{P_0})$ if and only if $\gf$ satisfies condition (a) of
Definition \ref{d: Delorme's intertwining conditions} for every $N\in\Z_+$ and all $\gd\in (\cD_{P_0})^{N}.$
\end{lemma}

We consider the component
$$
 \holspace_{P_0} \coleq \cO(\fadc) \otimes \cS(P_0)
$$
of the direct sum $\holspace$ defined in (\ref{e: holspace}),
and denote the natural projection $\holspace \to \holspace_{P_0}$ by $\pr.$
Moreover, we define
$$
\holspace_{P_0}(\cD_{P_0}) \coleq  \holspace_{P_0} \cap \holspace(\cD_{P_0}).
$$
\medbreak
\begin{prop}
\label{p: reduction to one minimal}
The natural projection $\pr: \holspace \to \holspace_{P_0}$ restricts to
a linear isomorphism from $\holspace(\cD)$ onto $\holspace_{P_0}(\cD_{P_0}).$
\end{prop}

\begin{proof}
If $\gf \in \holspace$ and if  $\gd$ is a finite sequence in $\cD_{P_0},$ then
$\gf_\gd = (\gf_{P_0})_\gd.$ Hence $\gf \in \holspace(\cD_{P_0})$ if and only if
$\pr (\gf) \in \holspace(\cD_{P_0}).$

From $\cD_{P_0} \subset \cD$ it follows that $\holspace(\cD) \subset \holspace(\cD_0).$
Therefore, $\pr$ maps $\holspace(\cD)$ into $\holspace_{P_0} \cap \holspace(\cD_{P_0}) = \holspace_{P_0}(\cD_{P_0}).$ The proof will be completed by showing that the restricted projection $\pr_{\cD} \coleq \pr|_{\holspace(\cD)}: \holspace(\cD) \to \holspace_{P_0}(\cD_{P_0})$ is a linear isomorphism.
We will do this by defining a map
$$
\tau: \;\;\holspace_{P_0} \too \holspace
$$
which will turn out to induce a two-sided inverse for $\pr_\cD.$

Each parabolic subgroup $P \in \cP(A)$ has a Langlands decomposition $P = M_P A_P N_P.$ Here $M_P$ is a real reductive group of the Harish-Chandra class, with Cartan decomposition $M_P = (M_P\cap K) \exp(\fm_P \cap \fp).$ Moreover, ${}^*\fa\coleq \fm_P \cap \fa$ is maximal abelian in $\fm_P \cap \fp,$ and the centralizer of ${}^*\fa$ in $K_P\coleq M_P \cap K$ equals $M = Z_K(\fa).$
We select a minimal parabolic subgroup
$$
{}^*Q = M\,  {}^*\!A \, {}^*\!N
$$
of $M_P.$ In addition, we fix a minimal
parabolic subgroup  $Q$ of $G$ containing $A$ such that ${}^*Q = Q \cap M_P,$ and we fix an element $w \in N_K(\fa)$ such that $P_0 = w^{-1} Q w.$ In case $P = P_0$, we agree to make the special
choice $Q = P_0$ and $w =e.$

Then by the subrepresentation theorem applied to the group $M_P,$
for each $\xi \in \dMPds$ we may fix a representation $\gs \in\dM$
and an element $\mu \in {}^*\fadc$ such that
$\xi$ is equivalent to a subrepresentation of the parabolically induced (smooth) representation $\pi^{M_P}_{{}^*Q, \gs, \mu}$ of $M_P.$ In addition, we fix an $M_P$-equivariant embedding
$$
j_{\xi}:\; (\xi, V_\xi) \;\;\embeds \;\;
(\pi^{M_P}_{{}^*Q, \gs, \mu}, C^\infty(K \cap M_P: \gs)).
$$
Through induction by stages, this embedding induces an embedding
\begin{equation}
\label{e: j xi P}
j_{\xi P}^G:\; C^{\infty}(K\col \xi) \too C^\infty(K\col \gs)
\end{equation}
which intertwines $\pi_{P, \xi, \gl}$ with $\pi_{Q, \gs, \mu + \gl},$
for all $\gl \in \faPdc.$ Here ${}^*\fa_{\iC}^*$ and $\fa_{P\iC}^*$ are viewed as subspaces of $\fadc$ via the direct sum decomposition
$\fa = {}^*\fa \oplus \fa_P.$
Thus, (\ref{e: j xi P}) is a morphism
in the category $\HFGO$, see Definition \ref{d: sec59}.
In the special case  $P = P_0,$ we agreed that $Q =P_0$ and $w = e.$ In this case, we have $\mu = 0,$ and as $\gs$ must be equivalent to $\xi,$ it follows that $\gs = w\cdot \gs = \xi.$  It is now readily verified that $j_{\xi P}^G$ is the identity map of $C^\infty(K\col \xi).$

Left translation by $w$ induces a topological linear isomorphism
\begin{equation}
\label{e: L w}
L(w): \; C^\infty(K\col\gs) \too C^\infty(K\col w\gs),
\end{equation}
which intertwines the induced representation $\pi_{Q, \gs, \mu+ \gl}$
with  $\pi_{P_0, w \gs, w(\mu + \gl)},$ for all $\gl \in \faPdc.$
Here $w\gs$ denotes the representation of $M$ in $V_\gs,$ given by
$w\gs(m) =\gs(w^{-1}mw).$ We denote by $(w\cdot \gs, V_{w \cdot \gs})$ the unique element of $\dM$ which is equivalent to $(w\gs, V_\gs).$
Fix an equivalence $t_w: V_\gs \to V_{w\cdot \gs}$ and denote the induced map $C^\infty(K\col w\gs) \to C^\infty(K \col w\cdot\gs)$
by $T_w.$ Then writing $\cL_w = T_w\after L(w),$ we obtain a
continuous linear injection
\begin{equation}
\label{e: embedding K xi to K w gs}
\cL_w\after j_{\xi P}^G:\;\; C^\infty(K\col \xi) \too C^\infty(K\col w\cdot \gs)
\end{equation}
which intertwines the representation $\pi_{P,\xi, \gl}$
with $\pi_{P_0, w\cdot \gs, w(\mu + \gl)},$ for all $\gl \in \faPdc.$
In the special case $P = P_0,$ where $Q = P_0, w =e, \mu = 0$
and $\gs = \xi,$ we may take $t_w = \id_{V_\xi}$ and then
$\cL_w$ and hence (\ref{e: embedding K xi to K w gs})
become the identity map of $C^\infty(K\col \xi).$

We are now ready to define the map $\tau.$ Let $P, \xi, \gl$ be as above, and let
$\gf\in \holspace_{P_0}(\cD_0).$ Then $\gf_{P_0, w\cdot\gs, w(\mu + \gl)}$ leaves the invariant
subspace ${\rm im}(\cL_w\after j_{\xi P}^G)$ for the representation
$\pi_{P_0, w\cdot \gs, w(\mu +\gl)}$ invariant by condition (a) of Definition \ref{d: Delorme's intertwining conditions}, so that we may define
$$
\tau(\gf)_{P, \xi,\gl}: \;\; C^\infty(K\col \xi) \to C^\infty(K\col \xi)
$$
to be the unique linear map such that the following diagram commutes
\begin{equation}
\label{e: comm diag t gf}
\xymatrix{
\ar[d]_{\tau(\gf)_{P,\xi,\gl}} C^\infty(K\col \xi)
      \ar[rr]^{\cL_w\after  j_{\xi P}^G}
      &&\ar[d]^{\gf_{P_0, w\cdot\gs,w(\mu + \gl)}}  C^\infty(K\col w\cdot\gs)  \\
C^\infty(K\col \xi ) \ar[rr]_{\cL_w\after j_{\xi P}^G}
      && C^\infty(K\col w\cdot \gs).
}
\end{equation}
Since $\gf_{P_0, w\cdot\gs} \in \cO(\fadc) \otimes \End{C^\infty(K\col w\cdot\gs)}_{K\times K},$ it is readily seen that $\tau(\gf)_{P,\xi}$ defines a holomorphic function on $\faPdc,$ with values in $\cS(P, \xi).$
Accordingly, $\tau$ defines a linear map $\holspace_{P_0} \to \holspace.$

We will now finish the proof by showing that
\begin{enumerate}
\item[{\rm (i)}]
$\tau$ maps
$\holspace_{P_0}(\cD_{P_0})$ into $\holspace(\cD),$
\item[{\rm (ii)}]
$\tau$ restricts to a linear isomorphism $\tau_\cD: \holspace_{P_0}(\cD_{P_0})\to \holspace(\cD)$ which is a two-sided inverse
for $\pr_\cD.$
\end{enumerate}

We will first establish (i). Let $\gf \in \holspace_{P_0}(\cD_0),$ and let
$N \in \Z_+$ and $\gd \in \cD^N.$ We claim that it suffices to show that
there exists a $\gd' \in \cD_{P_0}^N$ and a linear embedding
$j: V_{\gd} \to V_{\gd'}$ intertwining $\pi_\gd$ with $\pi_{\gd'}$
such that
\begin{equation}
\label{e: intertwining embedding j}
\gf_{\gd'} \after j = j \after \tau(\gf)_\gd.
\end{equation}
Indeed assume the claim to hold, and let $W \subset V_\gd$ be an invariant subspace. Then $j(W)$ is an invariant subspace of $V_{\gd'}.$
Moreover, since $\gf \in \holspace_{P_0}(\cD_0),$ it follows that
$\gf_{\gd'}$ leaves $j(W)$ invariant. As $j$ is injective, it now follows
from (\ref{e: intertwining embedding j}) that $\tau(\gf)_\gd$ leaves
$W$ invariant. Thus, the validity of the claim would imply that the above condition (i) holds.

We turn to the proof of the claim. It clearly suffices to prove the claim in case $N =1,$ so that $\gd \in \cD.$ Thus, $\gd$ is a $4$-tuple of the form $(P, \xi, \gl_0, \eta)$ with notation as in the beginning
of Section \ref{s: the intertwining conditions}. In particular,
$\eta$ is a finite sequence in $\faPdc.$ Let $Q, \gs, \mu, w$ be associated
with the data $P, \xi$ as in the first part of this proof, where the
definition of $\tau$ was given. Then the injective linear map
(\ref{e: embedding K xi to K w gs})
intertwines
 $\pi_{P,\xi, \gl}$  with $\pi_{P_0, w\cdot\gs, w(\mu + \gl)}$ for all $\gl \in \faPdc.$

It follows that the map
$$
j:\;\; C^\infty(K\col \xi)^{(\eta)} \too C^\infty(K\col w\cdot\gs)^{(\eta)}
$$
induced by (\ref{e: embedding K xi to K w gs})
interwines $\pi_\gd$ with $\pi_{\gd'},$ where
$\gd' = (P_0, w\cdot\gs, w(\gl_0 + \mu), w\eta).$ Moreover, the commutativity of the diagram (\ref{e: comm diag t gf}) implies that
$$
j \after \tau(\gf)_{P,\xi, \gl}^{(\eta)} =
        \gf_{P_0, w\cdot\gs, w(\mu + \gl)}^{(w\eta)} \after j,
$$
or, abbreviated, $j \after \tau(\gf)_\gd = \gf_{\gd'} \after j.$
This establishes the claim.

We now consider the induced map $\tau_\cD: \holspace_{P_0}(\cD_0) \to \holspace(\cD),$ obtained by restriction of $\tau,$ and will finish
the proof by establishing (ii).

Let $\gf \in \holspace_{P_0}(\cD_{P_0}).$ Moreover, let $Q, \gs, \mu, w$ be associated to $P = P_0$ as above. By the special choices we made for this particular parabolic subgroup, the diagram (\ref{e: comm diag t gf}) becomes
$$
\xymatrix{
\ar[d]_{\tau(\gf)_{P_0,\xi,\gl}} C^\infty(K\col \xi)
      \ar[rr]^{ \id }
      &&\ar[d]^{\gf_{P_0, \xi ,\gl}
      \;\;\;\;\;\;\;\;\;\;\;\;} C^\infty(K\col \xi)  \\
C^\infty(K\col \xi ) \ar[rr]_{\id }
      && C^\infty(K\col \xi),
}
$$
It follows that $\pr\after \tau (\gf) = \gf,$ and we see that $\tau_\cD$ is
a right inverse to $\pr_\cD.$

We will finish the proof by showing that $\tau_\cD$ is also a left inverse.
Let $\psi \in \holspace(\cD),$ and let $P \in \cP(A)$ and $\xi \in \dMPds.$ Let $Q, \gs, \mu, w$ be as above. As $\psi$ satisfies condition
(b) of Definition \ref{d: Delorme's intertwining conditions},
it follows that the following diagram commutes, for all $\gl \in \faPdc,$
$$
\xymatrix{
\ar[d]_{\psi_{P,\xi,\gl}} C^\infty(K\col \xi)
      \ar[rr]^{\cL_w \after j_{\xi P}^G}
      &&\ar[d]^{\psi_{P_0, w\cdot\gs,w(\mu + \gl)}
      \!\!\!\!\!\!\!\!\!\!\!\!\!\!\!\!\!\!\!\!\!\!\!\!\!\!\!}  C^\infty(K\col w\cdot\gs)  \\
C^\infty(K\col \xi ) \ar[rr]_{\cL_w\after j_{\xi P}^G}
      && C^\infty(K\col w\cdot\gs).
}
$$
By comparison with (\ref{e: comm diag t gf}) we see that $\psi_{P,\xi} =
\tau(\psi_{P_0})_{P, \xi}.$ Hence, $\tau\after\pr(\psi) = \psi,$ for $\psi \in \cF(\cD).$ Therefore, $\tau_\cD $ is a left inverse to $\pr_\cD.$
\end{proof}

\subsection{Reformulation in our setting}
We shall now compare Delorme's derivation process with the process defined
in the present paper. In the following we shall identify  $V \oplus V$ with $\C^2 \otimes V$
via the map $(v_1, v_2) \mapsto (1,0) \otimes v_1 + (0,1) \otimes v_2.$ This identification
induces an identification $\Endo(V \oplus V) \simeq \Endo(\C^2) \otimes \Endo(V).$
                  \par
Let now $\eta \in \fadc$ and let
 $\Omega$ be an open subset of $\fa_\iC^*.$
For $\gf\in \cO(\Omega)$ we define $D^{(\eta)}\gf \in \cO(\Omega) \otimes \Endo(\C^2)$ by
$$
D^{(\eta)}\gf(\gl) =
\left(
\begin{array}{cc}
\gf(\gl) & \gf(\gl; \eta) \\
0 & \gf(\gl)
\end{array}
\right).
$$
Let $V$ be a quasi-complete locally convex space.
Then for $\Phi \in \cO(\Omega) \otimes \Endo(V),$ Delorme's derivative $\Phi^{(\eta)}$ is given
by
$$
\Phi^{(\eta)} = (D^{(\eta)} \otimes \id_{\Endo(V)}) \Phi.
$$
Here we note that
\begin{equation}
\label{e: V eta as tensor product}
V^{(\eta)} =  \C^2\otimes V,
\end{equation}
so that
$\Endo(V^{(\eta)}) = \Endo(\C^2) \otimes \Endo(V).$
It follows that Delorme's differentiation map
$(\dotvar)^{(\eta)}$ from $\cO(\Omega, \Endo(V)) $ to $\cO(\Omega, \Endo(V^{(\eta)}))$
is the unique continuous linear extension of the map $D^{(\eta)} \otimes \id_{\Endo(V)}.$
                   \par
Now assume that $\eta \neq 0$ and let $\cI_\eta$ be the cofinite ideal of $\cO_0 = \cO_0(\fa_\iC^*)$
defined in (\ref{e: idealcodim2}).
Then $\cI_\eta$ has codimension $2$ by Lemma \ref{l: codimension two}. More precisely, let $X \in \fac$
be such that $\eta(X) =  1.$ Then the map
$
(z_1, z_2) \mapsto z_1 + z_2 X  + \cI_\eta
$
defines a linear isomorphism $\kappa$ from $\C^2$ onto the $\cO_0$-module
$$
E_\eta\coleq  \cO_0/\cI_\eta.
$$
Now assume that $V$ is barrelled. Then in view of (\ref{e: V eta as tensor product})
the isomorphism $\kappa$ induces
an isomorphism from $\cO(\Omega, \Endo(V^{(\eta)}))$ onto $\cO(\Omega, \Endo(E_\eta \otimes V)),$
denoted $\kappa_*.$

\begin{lemma}
With notation as above, the following diagram commutes
$$
\xymatrix{
\O(\Omega,\End{V})\ar@{->}[d]_{(\dotvar)^{(\eta)}}\ar@{->}[drr]^{\diffD^{(E_{\eta})}}&&\\
\O(\Omega,\End{V^{(\eta)}})\ar@{->}[rr]_{\kappa_*} && \O(\Omega,\End{E_{\eta}\otimes V})
}
$$
\end{lemma}

\begin{proof}
By the same calculation as in the proof of Lemma \ref{sec40}, it follows that the multiplication
action $m(\gf)$ of an element $\gf \in \cO_0$ on $E_\eta$ with respect to the basis $\bar 1, \bar X$ is given by
the matrix
$$
M(\gf)\coleq  \left(
\begin{array}{cc}
\gf(0) & \gf(0; \eta) \\
0 & \gf(0)
\end{array}
\right).
$$
Let $\widetilde \kappa$ be the isomorphism $\Endo(\C^2) \to \Endo(E_\eta)$ induced by $\kappa.$ Then
it follows that for all $\gf \in \cO(\Omega)$ and $\gl \in \Omega,$
$$
J^{(E_\eta)} \gf(\gl) = m(\gg_0(T_\gl\gf)) = \widetilde \kappa \, (\, M(\gg_0(T_\gl\gf))\, ) =
\widetilde\kappa\,(\, D^{(\eta)}\gf(\gl)\,).
$$
This immediately implies the commutativity of the diagram in case $V = \C.$
From this we see that the diagram commutes with the spaces $\cO(\Omega, \Endo(W))$ replaced
by the subspaces $\cO(\Omega) \otimes \Endo(W),$ for $W$ equal to $V,\, V^{(\eta)}$ or $E_\eta\otimes V.$
By density of the mentioned spaces and continuity of the maps involved, the result follows.
\end{proof}

Given a finite sequence $\eta = (\eta_1,\ldots, \eta_N)$ of elements in $\fadc,$
we define the $\cO_0$-module $E_\eta\coleq E_{\eta_1} \otimes \cdots \otimes E_{\eta_N}.$

\begin{cor}
For every finite sequence $\eta$ as above, and all $\xi \in \dM$ and $\gl \in \fadc,$
$$
\pi_{P_0, \xi, \gl}^{(\eta)} \simeq \pi^{(E_\eta)}_{P_0, \xi, \gl}.
$$
\end{cor}

\begin{proof}
For $\eta$ of length one this follows from the above lemma. For arbitrary sequences
it follows from the  recurrent nature of the definition of $\pi^{(\eta)}$ and Proposition
\ref{p: iterateder}.
\end{proof}

\begin{definition}
\label{d: End sharp}
\rm
Let $V$ be a Harish-Chandra module. We denote by $\Endo(V)^\#$ the space
of endomorphisms $\gf \in \Endo(V)_{K\times K}$ such that for every $n \in \N$ the product map
$\gf^{\times n} \in \Endo(V^{\times n})$ leaves all $(\fg, K)$-invariant subspaces of $V^{\times n}$ invariant.
                  \par
If $(\pi, V_\pi)$ is an admissible representation of $G$ of finite length, we define $\Endo(\pi)^\# \coleq  \Endo((V_{\pi})_K)^\#.$
\end{definition}
\par
For
$E$ a finite dimensional $\cO_0$-module (where $\cO_0 = \cO_0(\fadc)$), and for $\Xi \subset \dM$ and $\gL \subset \fadc$ finite subsets, we recall the definition of the representation $\pi_{E, \Xi, \gL}$ by (\ref{e: defi pi E Xi gL}). Moreover, given $\gf \in \holspace_{P_0} = \cO(\fadc) \otimes \cS(P_0),$ we
recall the definition of $\gf_{E, \Xi, \gL}$ in (\ref{e: defi gf E Xi gL}).

\begin{lemma}\label{l:  stability of ic under subq}
Let $E,E'$ be finite dimensional $\cO_0$-modules such that $E$ is a subquotient of $E'.$
Let $\Xi \subset \Xi'\subset \dM$  be finite subsets and let $\gL \subset \gL'\subset \fadc$
be finite subsets. Then for all $\gf \in \holspace_{P_0},$
$$
\gf_{E', \Xi', \gL'} \in \Endo(\pi_{E', \Xi', \gL'})^\# \Longrightarrow
\gf_{E, \Xi, \gL} \in \Endo(\pi_{E, \Xi, \gL})^\#
$$
\end{lemma}

\begin{proof}
The result follows from the crucial observation that the representation $\pi_{E, \Xi, \gL}$ is a subquotient of
$\pi_{E', \Xi', \gL'}$ and the map $\gf_{E', \Xi', \gL'}$ induces the map $\gf_{E, \Xi, \gL}.$
Obviously, it suffices to prove this observation for $\Xi = \Xi'$ and  $\gL = \gL'.$
In this case, we put, for $\gl\in\a_\iC^*,$
$$
\pi_{\gl} \coleq \oplus_{\xi \in \Xi} \;\;\pi_{P_0, \xi, \gl}.
$$
Then $\pi$ is a holomorphic family of admissible smooth representations of $G$ over $\fadc,$ so that the functor $X_\pi: E \mapsto \pi^{(E)}$ has the exactness properties of
Lemma \ref{l: exactness properties X pi}. Hence, $\pi^{(E)}_\gl$ is a subquotient of $\pi^{(E')}_\gl,$ and $\gf_{E, \Xi, \gl}$ is induced by $\gf_{E', \Xi, \gl},$ for all $\gl \in \gL.$ The result now follows by taking the direct sum
over $\gl \in \gL.$
\end{proof}

\begin{prop}\label{p: dinter}
Let $\gf \in \holspace_{P_0}.$ Then the following conditions are equivalent:
\begin{enumerate}
\itema
$\gf \in \holspace_{P_0}(\cD_{P_0});$
\itemb
for every finite dimensional $\O_0$-module $E$ and every pair of finite sets $\Xi \subset \dM$
and $\Lambda\subset \a_\iC^*,$ the endomorphism $\gf_{E,\Xi,\Lambda}$ belongs to
$\Endo(\pi_{E,\Xi,\Lambda})^\#.$
\end{enumerate}
\end{prop}
\begin{proof}
First assume (b). Let $\gd =( \gd_1, \ldots, \gd_N) $ be a sequence of data in $\cD_{P_0}.$
Then for (a)  it suffices to show that
\begin{equation}
\label{e: gf in sharp}
\gf_\gd \in \Endo(V_{\pi_\gd})^\#,
\end{equation}
in view of Lemma  \ref{l: one restricted intertwining condition suffices}.
\par
We note that,  with $\gd_j = (P_0, \xi_j, \gl_j, \eta_j),$
$$
\pi_\gd = \oplus_{j=1}^n \; \pi^{(E_{\eta_j})}_{P_0, \xi_j, \gl_j}.
$$
Let $E$ be the direct sum of the modules $E_{\eta_j}$,
and put $\Xi \coleq \{\xi_1, \ldots, \xi_N\}$  and $\gL \coleq \{\gl_1, \ldots, \gl_N\}.$ For each $1 \leq j \leq N$ and $\gl \in \fadc,$
let $(\pi_{j})_{\gl} \coleq \pi_{P_0, \xi_j, \gl}.$ Then $\pi_j$ is a holomorphic family of admissible smooth representations of $G$
over $\fadc,$ so that
 Lemma \ref{l: exactness properties X pi} applies to the functors
 $X_{\pi_j}: E \mapsto \pi_j^{(E)}.$ It follows that
$$
\pi_{\gd_j} \preceq \pi_{j \gl_j}^{(E)}, \quad \text{and}\quad
\gf_{\gd_j} = \gf_{P_0, \xi_j, \gl_j}^{(E)}|_{V_{\gd_j}}
$$
in a natural fashion (here $\rho \preceq \pi$ indicates that
$\rho$ is a subrepresentation of $\pi$).
This in turn implies that
$$
\pi_\gd \preceq \pi_{E,\Xi,\gL}, \quad \text{and}\quad
\gf_\gd = \gf_{E, \Xi, \gL}|_{V_\gd}.
$$
Hence (\ref{e: gf in sharp}) follows.
                 \par
Conversely, assume (a). Let $E, \Xi, \gL$ be as stated in (b).
First we consider the case that  $E$ is of the form $E_\eta,$ with $\eta$ a finite sequence in $\fadc.$
Then
$$
\pi_{E_\eta, \Xi, \gL} = \oplus_{\xi \in \Xi, \gl \in \gL} \;\; \pi^{(\eta)}_{P_0, \xi, \gl}
$$
and it follows from (a) that $\gf_{E_\eta, \Xi, \gL}$ belongs to $\Endo(\pi_{E_\eta, \Xi, \gL})^\#.$
In view of   Lemma \ref{l: exactness properties X pi}, this implies
that $\gf_{E, \Xi, \gL}$ belongs to $\Endo(\pi_{E, \Xi, \gL})^\#$ for $E$ a direct sum of copies of $E_\eta.$ Finally, let $E$ be arbitrary. Then by Proposition
\ref{sec43} the module $E$ is a subquotient of $E_\eta^N$ for a suitable finite sequence $\eta$ and a
suitable $N\in \N.$ This implies, again by Lemma \ref{l: exactness properties X pi}, that $\pi_{E, \Xi, \gL}$ is a subquotient of $\pi_{E_\eta^N, \Xi, \gL}$
and that $\gf_{E, \Xi, \gL}$ is induced by $\gf_{E_\eta^N, \Xi, \gL}.$
From this  and Lemma \ref{l: stability of ic under subq} it follows that $\gf_{E, \Xi, \gL}$ belongs to $\Endo(\pi_{E, \Xi, \gL})^\#.$
Hence (b).
\end{proof}

\section{Conditions in terms of the Hecke algebra}\label{s: Hecke algebra}

Let $k$ be a field and $A$ a $k$-algebra with an approximate identity $(\alpha_j)_{j\in J}$. This
means that
\begin{enumerate}
\itema
$J$ is a partially ordered set;
\itemb
for all $j_1,j_2$ in $J$ with $j_1\leq j_2$ we have
$\alpha_{j_1}\alpha_{j_2} =\alpha_{j_2}\alpha_{j_1}=\alpha_{j_1};$
\itemc for every $a\in A$ there exists a $j\in J$ such that $\alpha_j a= a\alpha_j=a.$
\end{enumerate}
The Hecke algebra $\H(G,K)$ is an example of such an algebra  (see \cite[Chap.~I, \S 6]{knappvogan}).
Indeed, let $(\vartheta_j)_{j \in \N}$ be an increasing sequence
of finite subsets of $\dK,$ whose union is $\dK.$ Then
$$
\ga_j \coleq \sum_{\gd \in \vartheta_j} \dim(\gd)\, \chi_{\gd^\vee}, \qquad j \in \N,
$$
defines an approximate identity in $\H(G,K).$

We denote the opposite algebra of $A$ by $A^{opp}.$ It is readily seen
that the elements $\ga_j \otimes \ga_j$  form an approximate identity for the algebra $A \otimes A^{opp},$
so that this algebra is approximately unital.
                   \par
For $j \in J,$ let $A_j$ be the set of $a \in A$ with  $\ga_j a  = a \ga_j = a .$ Then
it is readily seen that $A_j$ is a subalgebra of $A.$ Moreover, (b) implies that
$A_{j_1} \subset A_{j_2}$ whenever
$j_1 \leq j_2$, and (c) implies that $A$ is the union of the subalgebras $A_j.$

\subsection{Some general facts on approximately unital $A$-modules}
Let $V$ be a left  (or right) $A$-module (in particular, $V$
is a $k$-linear space).
We say that $V$ is approximately unital if, for every $v\in V$,
there exists a $j\in J$ such that $\alpha_j\cdot v= v$.
               \par
For $j\in J$, let
$$
V_j := \{ \ga_j \cdot v \mid v \in V\}.
$$
Then $V_{j_1}\subset V_{j_2}$ whenever $j_1\leq j_2$.
We note that $V$ is approximately unital if and only if  $\cup_{j \in J}\, V_j = V.$
               \par
Let $\End{V}$ denote the algebra of $k$-linear endomorphisms of $V$ and let
$$
\pi: \; A\longrightarrow \mathrm{End}(V)
$$
denote the canonical algebra homomorphism.
Then $V_j$ is the image of $\pi(\alpha_j).$ We note that $V_j$ is invariant under the action
of $A_j,$ turning this space into a left $A_j$-module. Let $V^j$ denote the kernel of $\pi(\alpha_j)$.
Then since $\ga_j$ is an idempotent, it follows that
\begin{equation}
\label{e: deco V in V j}
V = V^j \oplus V_j.
\end{equation}\indent
For any $A$-module $V$, the submodule $V_{au}\coleq  A\cdot V$ is approximately unital, and it is in
fact the maximal submodule with this property. Note that
$$
V_{au} = \cup_{j \in J} V_j.
$$
                 \par
We now assume that $V$ is an approximately unital $A$-module. The $k$-linear
 space $\Hom_k(V,k)$ has a natural $A^{opp}$-module structure. Accordingly, we
define
$$
V^\vee\coleq (\Hom_k(V,k))_{au}.
$$
We denote by $\End{V}_0$ the image of the natural linear map
$V \otimes V^\vee\to \Endo(V)$
induced by $v\otimes v'\mapsto (u \mapsto v'(u)v).$
                 \par
We will say that an $A$-module $V$ is admissible if $V_j$ is finite dimensional for every $j \in J.$

\begin{lemma}
Let $V$ be an admissible approximately unital $A$-module. Then, as $A\otimes A^{opp}$-modules,
$$
\End{V}_0=\mathrm{End}(V)_{au}.
$$
\end{lemma}

\begin{proof}
It is readily seen that $\End{V}_0$ is an
$A\otimes A^{opp}$-submodule of $\mathrm{End}(V)$ which is approximately
unital. This implies that $\End{V}_0\subset \mathrm{End}(V)_{au}.$\par
Conversely,
 let $j \in J.$ We consider the inclusion map $\iota_j: V_j \to V$ and the epimorphism $p_j: V \to V_j$ determined by
$\iota_j \after p_j = \pi(\ga_j).$ These maps induce linear maps
$S_j : \End{V} \to \End{V_j}$ and $T_j: \End{V_j} \to \End{V}$
given by
$$
S_j: \;f \mapsto  p_j \after f \after \iota_j
\quad{\rm and}\quad
T_j:\; g \mapsto  \iota_j \after g \after p_j.
$$
It is readily checked that
$$
S_j \after T_j = \id_{\End{V_j}}, \quad
T_j \after S_j(f) = \pi(\ga_j) \after f \after \pi(\ga_j) =
\pi_{\End{V}}(\ga_j) f.
$$
Hence, $T_j$ is an injective linear map $\End{V_j} \to \End{V}$
with image $\End{V}_j.$ One now readily checks that the following
diagram commutes:
\def\rmi{{\rm i}}
$$
\xymatrix{
{\End{V}_j} \ar[rr]^{\rmi_j}  && {\End{V}}   \\
 {\End{V_j}} \ar[u]^{t_j}  &&
                \ar[ll]  V_j \otimes (V_j)^* \ar[u]_{i_j \otimes p_j^*}
}
$$
Here $\rmi_j$ is the inclusion map, and $t_j$ is the map uniquely
determined by $\rmi_j \after t_j = T_j.$ The map at the bottom
is the canonical inclusion.

We now observe that the map at the bottom is a linear isomorphism by admissibility of $V.$ Moreover, $p_j^*$ maps $(V_j)^*$
into $(V^*)_j \subset V^\vee,$ and we infer that $\End{V}_j \subset \End{V}_0.$
Finally, by taking the union over all $j,$ we conclude that
$\End{V}_{au}\subset \End{V}_0$.
\end{proof}

\subsection{A double commutant theorem}
\begin{definition}\label{d: commut}
\rm
Let  $(\pi, V)$ be an approximately unital $A$-module.
We define $\End{\pi}^\#$
 to be the space of
$\varphi \in \End{V}_0$ such that for every $n\in\N\setminus\{0\}$ and every $A$-submodule
$W \subset V^{\times n}$ we have $\varphi^{\times n}(W) \subset W.$
\end{definition}
\begin{lemma}\label{l: dcomm}
Let $(\pi,V)$ be an admissible approximately unital $A$-module. Then
$$
\pi(A)=\End{\pi}^\#.
$$
\end{lemma}
\begin{proof}
Since $\pi(A)$ is an approximately unital $A\otimes A^{opp}$-submodule of
$\End{V}$, it is contained in $\End{V}_{au}=\End{V}_0$.
Moreover, if $a\in A$, then for every $n\in\N\setminus\{0\}$ and every invariant
subspace $W$ of $V^{\times n}$ we have $\pi(a)^{\times n}(W)\subset W.$
Hence, $\pi(A)\subset \Endo(\pi)^\#.$
          \par
We now turn to the converse inclusion.
By definition $\cE\coleq  \End{\pi}^\#$ is a subspace of $\End{V}_0=\End{V}_{au}$.
Moreover $\cE$ is $A\otimes A^{opp}$-invariant,
hence an approximately unital $A\otimes A^{opp}$-module.
Fix $\gf \in \cE.$ Then it follows that there exists a $j\in J$ such that
$(\alpha_j\otimes\alpha_j)\cdot\varphi = \varphi.$
          \par
By admissibility, the space $V_j$ has a finite basis $u_1, \ldots, u_n$
over $k.$ Let $W$ be the $A$-submodule of $V^{\times n}$ generated by $(u_1, \ldots, u_n).$
Then by definition of  $\Endo(\pi)^\#,$ the space $W$ is $\gf^{\times n}$-invariant.
This implies the existence of an element $a \in A$ such that $\gf(u_k) = \pi(a) u_k$
for all $1 \leq k \leq n.$ Hence $\gf = \pi(a)$ on $V_j.$ It follows that
$\gf = \pi(a \ga_j)$ on $V_j.$ On the other hand, both $\gf = (\ga_j \otimes \ga_j) \cdot \gf = \pi(\ga_j) \after \gf \after \pi(\ga_j)$
and $\pi(a\ga_j)$ vanish on $V^j.$ In view of (\ref{e: deco V in V j}) this implies
that $\gf = \pi(a\ga_j)$ on $V.$ Hence $\gf = \pi(a\ga_j) \in \pi(A)$ and the proof is complete.
\end{proof}

\subsection{Application to Harish-Chandra modules}\label{s: hcmod}
Every admissible $(\fg, K)$-module $V$ is a module for the Hecke algebra $\H(G, K)$ in
a natural way, and as such it is admissible and approximately unital.
This assignment of an   $\H(G,K)$-module to an admissible $(\fg, K)$-module
defines a functor
which establishes an isomorphism of categories, from the category of admissible $(\fg, K)$-modules
onto the category of admissible approximately unital $\H(G,K)$-modules (see \cite[I, \S 6, Theorem 1.117]{knappvogan}).
               \par
Accordingly, if $V$ is an admissible $(\fg, K)$-module,
then the associated algebra $\End{\pi}^\#$ consists of
all $K\times K$-finite endomorphisms $\gf$ of $V$ with the property that,
for every positive integer $n,$ the map $\gf^{\times n}$ preserves all $(\fg, K)$-invariant
subspaces of
$V^{\times n}.$ In particular, the present notation is compatible with the notation introduced earlier in Definition \ref{d: End sharp}.
              \par
\begin{cor}\label{c: dcomm}
Let $(\pi,V)$ be an admissible $(\fg, K)$-module.
Then
$$
\pi(\H(G,K))=\End{\pi}^\#.
$$
\end{cor}
\begin{proof}
This follows from Lemma \ref{l: dcomm}.
\end{proof}

\subsection{Proof of the main theorem}
\label{s: proof of main thm}
For $P \in \cP(A)$ we define  $\PW^\pre_P(G,K)$
to be the subspace of $\cF_P$ (see (\ref{e: holspace})),
consisting of all functions $\gf \in \cF_P$ such that there exists
an $R > 0$ and for all $n\in \N$ a constant $C_n >0$ such that
the estimate (\ref{e: PW estimate}) holds for all $\xi \in \dMPds$ and all $\gl \in \faPdc.$ Furthermore, we define
$$
\PW^\pre(G,K) = \oplus_{P \in \cP(A)} \; \PW_P^\pre(G,K).
$$
Then $\PW^\pre(G,K)$ is a subspace of $\cF.$
We recall that $P_0 \in \cP(A)$ is a fixed minimal parabolic subgroup of $G.$

\begin{definition}{\ }
\rm
\begin{enumerate}
\itema
The Delorme Paley--Wiener space $\PW^D(G,K)$ is defined to be the intersection $\PW^\pre(G,K) \cap \cF(\cD).$
\itemb
The restricted Delorme Paley--Wiener space $\PW_{P_0}^D(G,K)$
is defined to be the intersection $\PW_{P_0}^\pre(G,K) \cap \cF_{P_0}(\cD_{P_0}).$
\end{enumerate}
\end{definition}

\begin{thm}\label{t: first PW iso}
The natural projection $\pr: \cF \to \cF_{P_0}$ restricts
to an isomorphism of $\PW^D(G,K)$ onto $\PW_{P_0}^D(G,K).$
\end{thm}

\begin{proof}
It is clear that $\pr$ maps $\PW^\pre(G,K)$ into $\PW^\pre_{\scriptscriptstyle P_0}(G,K).$ Combining this with Proposition
\ref{p: reduction to one minimal}, we see that $\pr_\cD$ maps
$\PW^D(G,K)$ injectively into $\PW^D_{P_0}(G,K).$

Let $\tau: \cF_{P_0} \to \cF$ be defined as in the proof
of Proposition \ref{p: reduction to one minimal}. Then it is
readily checked that $\tau$ maps $\PW^\pre_{\scriptscriptstyle P_0}(G,K)$
into $\PW^\pre(G,K).$ This implies that $\tau_\cD$ maps
$\PW^D_{\scriptscriptstyle P_0}(G,K)$ into $\PW^D(G,K).$
As $\pr_\cD \after \tau_\cD = \id$ on
$\PW^D_{\scriptscriptstyle P_0}(G,K),$
it follows that
$\pr_\cD$ maps the space $\PW^D(G,K)$
surjectively onto $\PW^D_{\scriptscriptstyle P_0}(G,K).$
\end{proof}

\begin{definition}
\label{d: defi PWH}
\rm
We define $\PW^H(G,K)$ to be the space of functions
$\gf \in \PW^{\rm pre}_{P_0}(G,K),$ such that for
every finite dimensional $\cO_0$-module $E,$ and every pair of finite
sets $\Xi \subset \dM,$ $\gL\subset \fadc,$ we have
$$
\gf_{E,\Xi,\Lambda} \in  \pi_{E,\Xi,\Lambda}(\H(G,K)).
$$
\end{definition}
\vspace{4pt}

\medbreak

\begin{thm}
\label{t: PWD is PWH}
The space $\PW^D_{P_0}(G,K)$ equals $\PW^H(G,K).$
\end{thm}

\begin{proof}
Both spaces are subspaces of $\PW^\pre_{\scriptscriptstyle P_0}(G,K).$
Let $\gf$ be an element of the latter space. Then $\gf \in
\PW_{\scriptscriptstyle P_0}^D(G,K)$ is equivalent
to $\gf\in \cF_{\scriptscriptstyle P_0}(\cD_{P_0}),$
which in turn is equivalent to condition (b) of Proposition \ref{p: dinter}.
In view of Corollary \ref{c: dcomm}, the latter condition is
equivalent to $\gf \in \PW^H(G,K).$
\end{proof}

\begin{definition}
\rm
We define the Arthur Paley--Wiener space $\PW^A(G,K)$ to be the space
of functions $\gf \in \PW^\pre_{P_0}(G,K)$ such that $\gf$
satisfies the
Arthur--Campoli relations (see Definition \ref{d: AC relations}).
\end{definition}

\begin{thm}\label{t: PWD = PWA}
The  space $\PW^H(G,K)$ equals $\PW^A(G,K).$
\end{thm}

\begin{proof}
Both are subspaces of $\PW_{\scriptscriptstyle P_0}^\pre(G,K).$
Let $\gf$ be a function in the latter space.
Then by Proposition \ref{p: reformulation ac} the assertion
$\gf \in \PW^A(G,K)$ is equivalent to the assertion that
for every finite dimensional $\cO_0$-module, and every pair of finite
subsets $\Xi \subset \dM,$ $\gL \subset \fadc,$ we have
\begin{equation}
\label{e: gf in double perp}
\gf_{E, \Xi, \gL} \in (\pi_{E, \Xi, \gL}(C_c^\infty(G,K))^{\perp})^{\perp}.
\end{equation}
As every function from $\PW_{P_0}^{\rm pre}(G,K)$ has values in
$$
\oplus_{\xi \in \dM} \;\Endo(C^\infty(K\col \xi))_{\theta\theta}
$$
for some finite subset $\theta \subset \dK,$ it follows by admissibility
that (\ref{e: gf in double perp}) is equivalent to
$$
\gf_{E, \Xi, \gL} \in \pi_{E, \Xi, \gL}(C_c^\infty(G,K)).
$$
By Proposition \ref{end-hecke} this in turn is equivalent to
$$
\gf_{E, \Xi, \gL} \in \pi_{E, \Xi, \gL}(\H(G,K)).
$$
If follows that $\gf \in \PW^A(G,K) \iff \gf\in \PW^H(G,K).$
\end{proof}

We now come to the main result of our paper.

\begin{thm}
\label{t: main thm}
The map $\pr: \gf \mapsto \gf_{P_0}$ defines a linear
isomorphism from $\PW^D(G,K)$ onto $\PW^A(G,K).$
\end{thm}

\begin{proof}
This follows from combining Theorems \ref{t: first PW iso},
 \ref{t: PWD is PWH} and
\ref{t: PWD = PWA}.
\end{proof}

\subsection{Another useful characterization of the Paley-Wiener space}
\label{s: another useful char of PW}
In this subsection we will obtain another useful characterization of
the Paley-Wiener space $\PW^A(G,K).$ This will allow us to derive
the Paley-Wiener theorem due to Helgason
\cite{helgasonpw} and Gangolli \cite{gangolli} from Arthur's Paley-Wiener theorem.

We define the Fourier transform $\Fou: \Cci(G,K)_{K \times K} \to
\cO(\fadc) \otimes \cS_{P_0}$ by $\Fou(f)_\xi(\gl) = \widehat f(P_0, \xi, \gl),$ see (\ref{e: defi Fourier}). From this definition
it is immediate that $\Fou$ is a $K\times K$-equivariant linear map.

In terms of this Fourier transform, Arthur's Paley-Wiener theorem
may be stated as follows (see \cite{arthurpw} and \cite{BSpwspace}).

\begin{thm}[Arthur's Paley-Wiener theorem]
The map $\Fou$ is a $K \times K$-equivariant linear isomorphism
from $C_c^\infty(G,K)_{K \times K}$ onto $\PW^A(G,K).$
\end{thm}

By Theorem \ref{t: PWD = PWA} the Paley-Wiener space
$\PW^A(G,K)$ equals the space $\PW^H(G,K)$ introduced in Definition \ref{d: defi PWH}. We shall now give another characterization of that space.

We use the notation $\PW(\fa)$ for the (Euclidean) Paley-Wiener space associated
with $\fa,$ i.e., $\PW(\fa)$ is the image of the classical Fourier
transform $C_c^\infty(\fa) \to \cO(\fadc).$
Then the space $\PW_{P_0}^\pre(G,K),$ introduced in the
beginning of Subsection \ref{s: proof of main thm},
equals $\PW(\fa) \otimes \cS(P_0).$

\begin{prop}
The space $\PW^H(G,K)$ is equal to the space of all functions $\gf \in \PW(\fa) \otimes \cS(P_0)$ with the following property.

For each  finite number of triples $(u_j, \xi_j, \gl_j) \in S(\fa^*) \times \dM \times \fadc,$ $1 \leq j \leq n$,
there exists an element $h \in \H(G,K)$ such that
$$
\gf(\xi_j, \gl_j; u_j) = \pi_{P_0, \xi_j, \gl_j; u_j}(h)\qquad \text{for all }\;\;\;1 \leq j \leq n.
$$
\end{prop}

\begin{proof}
We first assume that $\gf \in \PW^H(G,K).$ Then $\gf \in \PW(A) \otimes \cS(P_0).$ Let a finite number of triples $(u_j, \xi_j , \gl_j)$ be given. For each $u_j$ there exists a finite dimensional
$\cO_0$-module $E_j,$ and a linear functional $\eta_j \in \Endo(E_j)^*$ such that $\eta_j \after f^{(E_j)} = f(\gl_j;u_j),$ for all $f \in \cO(\fadc);$ see Lemma \ref{l: E associated with u}. Put $E = \oplus E_j,$ and let $\pr_j: E \to E_j$ denote the associated projection maps. There exists an element $h \in \H(G,K)$ such that $ \gf_{E, \Xi, \gL} =
\pi_{E, \Xi, \gL}(h).$ This implies that $\gf_{E,\xi_j, \gl_j} =
\pi_{E, \xi_j,\gl_j}(h).$
for each $j.$
By application of $\eta_j \after \pr_j \otimes I$ to the latter expression it follows that $\gf(\xi_j, \gl_j;u_j) = \pi_{P_0, \xi_j, \gl_j; u_j}(h),$ for all $j.$ This proves that $\PW^H(G,K)$ is included in
the space described.

To obtain the other inclusion, let $\gf \in \PW(\fa) \otimes \cS(P_0)$
satisfy the conditions of the space described. Let $E$ be a finite dimensional $\cO_0$-module, and let
$\Xi \subset M^\wedge$ and $\gL \subset \fadc$ be finite sets.
Fix a basis $\{\eta^k\}$ of $\Endo(E)^*.$
We may number the elements of $\gL$ by $\gl_j.$ Then for each $j, k$
there exists a $u_j^k \in S(\fa^*) $ such that $\eta^k f^{(E)}(\gl_j)
= f(\gl_j; u_j^k)$ for all $f \in \cO(\fadc),$
see Lemma \ref{l: dual of End E}.
By the assumption
on $\gf,$ there exists an element $h \in \H(G,K)$ such that
$\gf(\xi,\gl_j; u^k_j) = \pi_{P_0, \xi, \gl_j; u^k_j}(h)$ for
all $\xi \in \Xi$ and all $j,k.$
It follows that
$$
(\eta^k \otimes 1)\after \gf_{E, \xi, \gl_j} =
(\eta^k \otimes 1)\after \pi_{E,\xi, \gl_j}(h),
$$
for all $k, j$ and $\xi \in \Xi.$ This implies that
$\gf_{E, \xi, \gl_j} = \pi_{E, \xi, \gl_j}(h)$ for all
$j$ and all $\xi \in \Xi.$ Hence, $\gf_{E, \Xi, \gL}
= \pi_{E,\Xi, \gL}(h),$ and we conclude that $\gf \in \PW^H(G,K).$
\end{proof}

\begin{cor}
\label{c: char of PWH11}
The space $\PW^H(G,K)_{11}$ of $K\times K$-fixed elements in\break
$\PW^H(G,K)$ consists of all $\gf \in \PW(\fa)\otimes \cS(P_0, 1)_{11}$ such that for each finite number of pairs
 $(u_i, \gl_i) \in S(\fa^*) \times \fadc,$ $1 \leq j \leq n$,
there exists an element $h \in \H(G,K)_{11}$ such that
$$
\gf(\gl_j; u_j) = \pi_{P_0, 1, \gl_j; u_j}(h)\qquad \text{for all }\;\;\;1 \leq j \leq n.
$$
\end{cor}
\begin{proof}
Since $\cS(P_0)_{11} = \cS(P_0, 1)_{11},$ this follows
from the previous result by projection onto the $K \times K$-type $(1,1).$
\end{proof}

Let $P_1 \in \Endo(C^\infty(K:1))$ denote the $K$-equivariant projection onto the one-dimensional subspace of $C^\infty(M\bs K)$ consisting
of the constant functions. Then
we observe that
\begin{equation}
\label{e: char cS11}
\cS(P_0,1)_{11} \simeq \Endo(C^\infty(K:1))_{11} = \C \,P_1.
\end{equation}
Accordingly, we may may view $\PW^H(G,K)_{11}$ as a subspace of
$\PW(\fa)\otimes \C\, P_1.$

The inclusion map $\iota: K \to G$ induces a continuous linear
map $\iota^*: C^\infty(G) \to C^\infty(K)$ by pull-back.
The transposed of this map is an injective continuous
linear map $\iota_*: \cE'(K) \to \cE'(G).$ Accordingly, we shall
use this map to view $\cE'(K)$ as a subspace of $\cE'(G).$ In particular, the normalized Haar measure $dk$ will be viewed as the
element of $\H(G,K)$ given by $ dk(\gf) = \int_K \gf(k)\; dk,$ for $\gf \in C^\infty(G).$ Clearly, $dk \in \H(G,K)_{11}.$

\begin{lemma}
\label{l: 11 component of H}
The map $u \mapsto R_u dk$ induces a linear isomorphism
from $U(\fg)^K /U(\fg)^K \cap U(\fg)\fk$ onto $\H(G,K)_{11}.$
\end{lemma}

\begin{proof}
We define the linear map  $\ga: U(\fg) \otimes \cE'(K) \to \cE'(G)$ by
$$
\ga(u \otimes T) = R_u T,
$$
where $R$ denotes the right regular representation on $\cE'(G).$
Then $\ga$ factors to a linear isomorphism
$$
\bar \ga: U(\fg)\otimes_{U(\fk)} \cE'(K) \to \H(G,K),
$$
see  \cite[I, \S 6]{knappvogan}.
This map intertwines the $K\times K$-actions
$(\Ad \otimes R) \times (1 \otimes L)$ and $R \times L,$ hence
restricts to an isomorphism
$$
(U(\fg)\otimes_{U(\fk)} \cE'(K))_{11} {\buildrel \simeq\over \longrightarrow} \H(G,K)_{11}
$$
The space of left $K$-invariants in $\cE'(K)$ equals $\C dk.$
As $dk$ is also right
$K$-invariant, we see that
$$
(U(\fg)\otimes_{U(\fk)} \cE'(K))_{11} \simeq U(\fg)^K \otimes_{U(\fk)} \C dk \simeq U(\fg)^K / U(\fg)^K \cap U(\fg)\fk.
$$
The result now follows.
\end{proof}

In the following lemma, $W$ denotes the Weyl group of $\fa$ in $\fg.$
\begin{lemma}
\label{l: image H11}
The image of $\H(G,K)_{11}$ under Fourier transform
$h \mapsto \widehat h$ equals
the subspace $P(\fadc)^W \otimes \C P_1$ of $\cO(\fadc)\otimes \cS(P_0).$ \end{lemma}

\begin{proof}
Denote the image by $S.$ Then $S$ is contained in $\cO(\fadc) \otimes
\cS(P_0)_{11} \simeq \cO(\fadc) \otimes \C P_1.$
Let $h \in \H(G,K)_{11}.$ Then it follows that the Fourier transform
of $h$ is of the form $\psi \otimes P_1,$ with $\psi \in \cO(\fadc).$
In view of the previous result, $h = R_u(dk),$ with $u \in U(\fg)^K.$
Hence, for all $\gl \in \fadc,$
\begin{eqnarray*}
\psi(\gl)1_{M \bs K} & = & (\psi(\gl)\otimes P_1)1_{M\bs K}\\
&=& \pi_{P_0,1 , \gl}(h) 1_{M \bs K} \\
& =& \pi_{P_0, 1, \gl}(u)1_{M \bs K} = \gg(u, \gl) 1_{M \bs K},
\end{eqnarray*}
where $\gg$ denotes the Harish-Chandra algebra homomorphism $U(\fg)^K \to S(\fa)$ which has image $S(\fa)^W = P(\fadc)^W.$ The result now follows by application of Lemma
\ref{l: 11 component of H}
\end{proof}

\begin{cor}
\label{c: PW 11}
The space $\PW^H(G,K)_{11}$ equals $\PW(\fa)^W \otimes \C P_1.$
\end{cor}

\begin{proof}
Let $\gf \in \PW^H(G,K)_{11}.$ Let $\gl \in \fadc$ and $w \in W.$
Then by Corollary \ref{c: char of PWH11} there exists an element $h \in \H(G,K)_{11}$ such that
$\gf(\gl) = \widehat h(P_0, 1, \gl)$ and
$\gf(w\gl) = \widehat h(P_0, 1, w\gl).$
As $\widehat h(P_0, 1, \gl) = \widehat h(P_0, 1, w\gl)$ by
Lemma \ref{l: image H11}, we see that $\gf$ is $W$-invariant. Hence,
$\gf \in \PW(\fa)^W \otimes \C P_1.$

Conversely, assume that $\gf \in \PW(\fa)^W \otimes \C P_1.$
Write $\gf = \psi \otimes P_1,$ then $\psi$ is a $W$-invariant
holomorphic function.
Let  $(u_j, \gl_j) \in S(\fa^*) \times \fadc,$ $1 \leq j \leq n.$
Then there exists an element $p \in P(\fa^*)^W$ such that
$$
\psi( \gl_j; u_j) = p(\gl_j; u_j) \qquad \text{for all }\;\;\;1 \leq j \leq n.
$$
In view of Lemma \ref{l: image H11} there exists
an element $h \in \H(G,K)_{11}$ such that $\widehat h(P_0, 1, \gl) = p \otimes P_1.$ Hence,
$$
\gf( \gl_j; u_j) = \widehat h (P_0, 1, \gl_j; u_j) = \pi_{P_0, 1, \gl_j;u_j}(h), \qquad (1 \leq j \leq n).
$$
In view of Corollary \ref{c: char of PWH11} it follows that
$\gf \in \PW^H(G,K)_{11}.$
\end{proof}

We now note that the spherical Fourier transform
$$
\cF_{11}: C_c^\infty(G,K)_{11} \to \cO(\fadc)
$$
is given by the formula
$
\cF_{11}f(\gl) 1_{K/M} = \pi_{P_0, 1, \gl}(f) 1_{K/M}.
$
This implies that for all $f \in C_c^\infty(G,K)_{11}$ we have
$$
\cF_{11}f(\gl) \otimes P_1 = \Fou f(\gl).
$$

We can now finally deduce the Paley-Wiener theorem of Helgason
\cite{helgasonpw} and Gangolli \cite{gangolli}.

\begin{cor}
$\cF_{11}(C_c^\infty(G,K)_{11}) = \PW(\fa)^W.$
\end{cor}

\begin{proof}
By $K\times K$-equivariance, it follows from Arthur's Paley-Wiener theorem that
$$
\cF_{11} (C_c^\infty(G,K)_{11})\otimes \C P_1  =
\Fou(C^\infty_c(G,K)_{11}) = \PW^A(G,K)_{11}.
$$
The latter space equals $\PW^H(G,K)_{11},$ by Theorem \ref{t: PWD = PWA}.
Now apply Corollary \ref{c: PW 11}.
\end{proof}

\appendix

\gdef\thesection{\textsc{Appendix}}

\renewcommand{\thethm}{A.\arabic{thm}}
\renewcommand{\theequation}{A.\arabic{equation}}

\section{Some topological results}
In this appendix, all locally convex spaces will be assumed to be complex and Hausdorff.
If $X$ is a locally compact Hausdorff space, and $V$ a quasi-complete locally convex space,
then $C(X,V),$ equipped with the topology of uniform convergence on compact subsets,
is quasi-complete as well.
Let $\Omega$ be an open subset of a finite dimensional complex linear space $\fv_\iC.$  Then $\cO(\Omega, V),$
the space of holomorphic functions $\Omega \to V,$ is a closed subspace of $C(\Omega, V).$
Moreover, the map $(f, v) \mapsto (z \mapsto f(z)v)$ induces an embedding of the algebraic tensor
product $\cO(\Omega) \otimes V$ onto a subspace of $\cO(\Omega, V).$ Accordingly,
we shall view this tensor product as a subspace. Let $P(\fv_\iC)$ denote the space of polynomial
functions $\fv_\iC \to \C.$

\begin{lemma}
\label{l: density of P in O}
The algebraic tensor product $\cP(\fv_\iC)\otimes V$ is dense in $\O(\Omega,V).$
\end{lemma}
In particular, $\cO(\Omega) \otimes V$ is a dense subspace of $\cO(\Omega, V).$
\begin{proof}
By using partitions of unity, one readily sees that $C(\Omega) \otimes V$ is dense in $C(\Omega, V).$
On the other hand, by application of the Stone-Weierstrass theorem, it follows that $P(\fv_\iC)$ is dense
in $C(\Omega).$ The lemma now readily follows.
\end{proof}
      \par
Let $V_1$, $V_2$ and $V_3$ be quasi-complete locally convex spaces
and let
$$
\beta:V_1\times V_2\rightarrow V_3
$$
be a  bilinear map. If $X$ is a set, then
$\gb$ induces a $\C^X$-bilinear map
$$
\begin{array}{rcl}
\beta_*: V_1^X\times V_2^X & \rightarrow & V_3^X\\
(f_1,f_2) & \mapsto & \big(x\mapsto\beta(f_1(x),f_2(x))\big).
\end{array}
$$
If $X$ is a locally compact Hausdorff space, we will say that
$\gb$ preserves continuity on $X$ if $\gb_*$ maps
 $C(X,V_1)\times C(X,V_2)$ into $C(X,V_3).$
We note that  $\gb_*: C(X,V_1)\times C(X,V_2)\to C(X,V_3)$ is $C(X)$-bilinear.
              \par
Similarly, if $\Omega$ is an open subset of $\fv_\iC,$ we will say that $\gb$ preserves holomorphy on $\Omega$
if $\gb_*$ maps $\cO(\Omega, V_1) \times \cO(\Omega, V_2)$ into $\cO(\Omega, V_3).$ Note that
the map $\gb_*$ is $\cO(\Omega)$-bilinear.
              \par
If $X,Y$ are topological spaces, and $A \subset X,$ then a map $f: X \to Y$ is said to be
continuous relative to $A$ if $f|_A: A \to Y$ is continuous (with respect to the restriction
topology on $A$).

\begin{lemma}
\label{l: criterion preservation C}
Assume that $\gb$ is separately continuous and in addition
continuous relative
to every compact subset
of $V_1 \times V_2.$ Let $X$ be a locally compact Hausdorff space.
Then $\gb$ preserves continuity on $X.$  Moreover,
\begin{enumerate}
\itema
if $V_1$ is barrelled,
then $\beta_*: C(X,V_1)\times C(X,V_2) \to C(X,V_3)$
is continuous in the first variable;
\itemb
if $V_2$ is barrelled, then $\beta_*: C(X,V_1)\times C(X,V_2) \to C(X,V_3)$
is continuous in the second variable.
\end{enumerate}
\end{lemma}

\begin{proof}
Let $K\subset X$ be compact, and $f_1\in C(X,V_1)$, $f_2\in C(X,V_2)$. Put $C_1\coleq f_1(K)$, $C_2\coleq f_2(K)$.
Since $\beta$ is continuous relative to $C_1\times C_2$ and
$x\mapsto (f_1(x),f_2(x))$ has continuous restriction to
$K,$ with values in $C_1\times C_2$,
it follows that  $\beta_*(f_1,f_2)$ is continuous
relative to $K$.
Since $X$ is locally compact,  we see that $\beta_*$ has values in $C(X,V_3)$. This establishes
the first assertion.
\par
We now turn to the remaining assertions. By symmetry, it suffices to establish (a).
Thus, assume that $V_1$ is barrelled and let $f_2 \in C(X, V_2)$ be fixed.
From the first part of the proof we know that  $f_1 \mapsto \gb_*(f_1, f_2)$
is a linear map from $C(X, V_1)$ to $C(X, V_3).$ To establish its continuity, fix a seminorm
$q_3$ of $V_3,$ and a compact subset $K$ of $X.$ Put $C_2 = f_2(K).$ Then the family
of continuous linear maps $\gb(\dotvar, v_2) \in \Hom(V_1, V_3)$ is pointwise bounded,
for $v_2 \in C_2.$ By barrelledness of $V_1$ it follows that the family is equicontinuous.
Hence, there exists a continuous seminorm $q_1$ on $V_1$ such that
$$
q_3(\gb(v_1, v_2)) \leq q_1(v_1), \qquad v_1 \in V_1, v_2 \in C_2.
$$
Let $f_1 \in C(X , V_1).$ Then substituting $f_1(x)$ for $v_1$ and $f_2(x)$ for $v_2\;$ $(x \in K)$
in the above estimate, we find that
$$
\sup_{x \in K} q_3(\gb_*(f_1, f_2)(x)) \leq \sup_{x\in K} q_1(f_1(x)).
$$
This establishes the continuity.
\end{proof}


\begin{lemma}\label{sec45}
Assume that
$V_1, V_2$ are barrelled, and let $\gb$ be as in Lemma \ref{l: criterion preservation C}
If $\Omega \subset \fv_\iC$ is open, then $\gb$ preserves holomorphy on $\Omega.$
Moreover, the map $\beta_*:\; \O(\Omega,V_1)\times\O(\Omega,V_2) \to \O(\Omega,V_3)$ is a  separately
continuous $\cO(\Omega)$-bilinear map.
\end{lemma}

\begin{proof}
It is readily seen that $\gb_*$ maps the subspace $(\O(\Omega)\otimes V_1)\times (\O(\Omega)\otimes V_2)$
into the closed  subspace $\cO(\Omega, V_3)$ of the quasi-complete space $C(\Omega, V_3).$
By continuity of $\gb_*$ in the first variable (see Lemma \ref{l: criterion preservation C}), and by density and closedness, it follows
that $\gb_*$ maps $\O(\Omega, V_1)\times \O(\Omega)\otimes V_2$ into $\cO(\Omega, V_3).$
By the same kind of argument applied to the second variable, the result follows.
\end{proof}

Examples of a different nature are provided by the composition of maps.
If $V,W$ are two locally convex spaces then by $\Hom(V,W)$ we denote
the space of continuous complex linear maps $V \to W.$ Unless otherwise specified this space is equipped with the strong operator topology. If $W$ is quasi-complete, and $V$ barrelled, then by application of the principle
of uniform boundedness, it follows that $\Hom(V,W)$ is quasi-complete
(see, e.g., \cite[Ch.~III, 27, \S~4, no.~2, Cor.~4]{boutvs3}).
                      \par
Assume now that $V_1, V_2$ and $V_3$ are arbitrary locally convex spaces,
not necessarily quasi-complete.
\begin{lemma}
\label{l: continuity of composition}
The map
\begin{equation}
\label{e: gb as composition}
\gb: \;\Hom(V_1, V_2) \times \Hom(V_2, V_3) \to \Hom(V_1, V_3),  \;\; (A,B) \mapsto B \after A
\end{equation}
is bilinear and separately continuous. If $V_2$ is barrelled, then $\gb$ is continuous relative to subsets of the form
$\Hom(V_1, V_2) \times C,$ with $C \subset \Hom(V_2,V_3)$ compact.
\end{lemma}

\begin{proof} 
As $V_2$ is barrelled, every compact subset of $\Hom(V_2, V_3)$
is equicontinuous. The result now follows from \cite[Ch.~III, 33, \S~5, no.~5, Prop.~9]{boutvs3}.
\end{proof}

\begin{cor}\label{c: continuity preserved by bilinear}
Let $V_1, V_2, V_3$ be quasi-complete, and assume $V_2$ is barrelled. Let $\gb$
be as in (\ref{e: gb as composition}). Then $\gb$ preserves continuity on $X,$
for any locally compact Hausdorff space $X.$
\end{cor}

\begin{proof}
This follows from Lemma \ref{l: continuity of composition} combined with Lemma
\ref{l: criterion preservation C}.
\end{proof}

\begin{lemma}
\label{l: composition of holomorphic homomorphisms}
Let $V_1, V_2, V_3$ be quasi-complete, and assume both $V_1$  and $ V_2$ to be barrelled.
Let $\gb$ be the composition map given by (\ref{e: gb as composition}) and let
$\Omega \subset \fv_\iC$ be open. Then
$\gb$ preserves  holomorphy on $\Omega.$
\end{lemma}

\begin{proof}
Let $A\in \cO(\Omega, \Hom(V_1, V_2))$ and $B \in \cO(\Omega, \Hom(V_2, V_3)).$
Then by the previous result it follows that $\gb_*(A,B): \mu \mapsto B(\mu)A(\mu)$ is continuous.
By barrelledness of $V_1,$ the function $f: (\mu, v)\mapsto A(\mu)v, \Omega \times V_1 \to V_2$
is continuous and holomorphic in $\mu.$ Similarly, the function
$g: (\nu, v_2) \mapsto B(\nu)v_2, \Omega \times V_2 \to V_3$ is continuous and holomorphic in $\nu.$
It follows that $h: (\mu,\nu,v_1) \mapsto g(\nu, f(\mu, v_1))$ is continuous, and separately
holomorphic in $\mu, \nu.$ In particular, for fixed $v_1,$ the map $(\mu, \nu) \mapsto h(\mu, \nu, v_1)$
is holomorphic in each of the two variables. If $\xi \in V_3'$ then $(\mu, \nu) \mapsto \xi h(\mu, \nu, v_1)$
is continuous and separately holomorphic, hence holomorphic. It follows that $(\mu, \nu)\mapsto
h(\mu, \nu, v_1)$ is weakly holomorphic, hence holomorphic. We conclude that
$\mu \mapsto h(\mu, \mu, v_1)$ is holomorphic. This implies that the function
$(\mu, v_1) \mapsto \gb_*(A,B)(\mu)v_1$ is continuous on $\Omega\times V_1,$ holomorphic in $\mu$ and linear in $v.$ This
in turn implies that $\gb_*(A,B)$ is holomorphic as a function with values in $\Hom(V_1, V_3).$
\end{proof}

Given an element $u$ of the symmetric algebra $S(\fv)$ of $\fv_\iC,$ we will
write  $\partial_u$ for the naturally associated constant coefficient (holomorphic) differential operator on $\fv_\iC$, see \S \ref{ss: derivation process}. This operator
acts on the space of holomorphic functions
$\cO(\Omega),$ for any open subset $\Omega \subset \fv_\iC.$ We agree to write $f(\dotvar; u) = \partial_u f,$ for $f \in \cO(\Omega),$ see (\ref{e: HC notation for action of Sv}).

If $V$ is a quasi-complete locally convex
space, then $\partial_u \otimes \id_V$ has a unique extension to a continuous linear endomorphism
of $\cO(\Omega, V).$ Indeed, uniqueness is obvious from density of $\cO(\Omega)\otimes V$ in $\cO(\Omega, V).$
Existence follows for instance by application of the Cauchy integral formula.
                      \par
Let $A$ be a finite dimensional associative algebra. Suppose
$$
M:\O(\Omega)\rightarrow \O(\Omega,A)
$$
is a continuous algebra homomorphism which can be represented by an element from $S(\fv) \otimes A,$ i.e.,
there exist $u_i \in S(\fv)$ and $a_i \in A$ such that
\begin{equation}
\label{e: defi M}
M f = \sum_i \partial_{u_i}f \cdot a_i.
\end{equation}
 Then it follows
from the above that the map $M \otimes \id_V$ has a unique extension
to a continuous linear map  $\cO(\Omega, V) \to \cO(\Omega, A \otimes V).$
The extension is denoted by $M \hatotimes\, \id_V,$ or more briefly by $M$ again.

\begin{rem}\label{r: J is M}
\rm
In the present paper, the following examples are of particular importance.
Let $V$ be a quasi-complete locally convex space and
let $\cI$ be a cofinite ideal in $\cO_0.$ Then $M = J_\cI$ is of the form \eqref{e: defi M}, see \eqref{e: JI}, and defines an algebra homomorphism $\cO(\Omega) \to \cO(\Omega, \cO_0/\cI).$ It follows that
$J_\cI \otimes \id_V: \cO(\Omega) \otimes V \to \cO(\Omega, \cO_0/\cI) \otimes V$
has a unique extension to a continuous linear map
$$
J_\cI \,\hatotimes\,\id_V: \;\;\cO(\Omega, V) \longrightarrow \cO(\Omega, \cO_0/\cI \otimes V),
$$
which will often briefly be denoted by $J_\cI$ again.
                      \par
Similarly, let $E$ be a finite dimensional $\cO_0$-module. Then $\diffD^{(E)}: \cO(\Omega) \to
\cO(\Omega, \End{E})$ is an algebra homomorphism of the form \eqref{e: defi M}, see Definition \ref{d: diffD} and Corollary \ref{c: continuity of diffD}. Hence,
$\diffD^{(E)} \otimes \id_V$ has a unique extension to a continuous linear map
$$
J^{(E)} \,\hatotimes \,\id_V:\;\; \cO(\Omega, V) \longrightarrow \cO(\Omega, \Endo(E) \otimes V),
$$
which will often briefly be denoted by $J^{(E)}: f \mapsto f^{(E)}.$
\end{rem}

\begin{lemma}
Let $V_1, V_2$ and $V_3$ be quasi-complete locally convex Hausdorff spaces, and let $\gb: V_1 \times V_2 \to V_3$
be a bilinear form which preserves holomorphy on the open subset $\Omega$ of $\fv_\iC.$
Let $\gb': E_1 \times E_2 \to E_3$ be a bilinear form of finite dimensional spaces.
Then the bilinear form $\gb'\otimes \gb:  E_1 \otimes V_1 \times E_2 \otimes V_2 \to E_3 \otimes V_3$
preserves holomorphy on $\Omega$ as well.
\end{lemma}

\begin{proof}
This is an easy consequence of the finite dimensionality  of the spaces $E_j.$
\end{proof}

Let $m^A$ denote the
(continuous) bilinear product map
$$
m^A: \;\; (a,b) \mapsto ab, \;\;A \times A \to A.
$$
\begin{prop}
\label{p: property of M}
Let $V_1$, $V_2$, $V_3$ be quasi-complete locally convex Hausdorff spaces, and let $\beta: V_1 \times V_2 \to V_3$
be a bilinear form which preserves holomorphy on $\Omega.$ Let $A$ be a finite dimensional
associative algebra, and let $M: \cO(\Omega) \to \cO(\Omega,A)$ be as in (\ref{e: defi M}).
Then, for all $f_j\in\O(\Omega,V_j)$, $j=1,2$,
$$
M\beta_*(f_1,f_2)=(m^A\otimes\beta)_*(Mf_1,Mf_2).
$$
\end{prop}

The following lemma prepares for the proof.
Given $\mu \in \Omega,$ let $\cM_\mu$ be the maximal ideal of the algebra $\cO(\Omega)$ consisting of the functions
vanishing at $\mu.$

\begin{lemma}
\label{l: M and fm mu}
Let $M$ be as in \eqref{e: defi M}.There exists a number $d \in \N$ with the following property.
Let $V$ be a quasi-complete locally convex Hausdorff space. Then for every $k \geq d,$
the operator
$M\hatotimes  \id_V$ maps $\cM_\mu^k \cO(\Omega, V)$ into $\cM_\mu^{k - d} \cO(\Omega, A \otimes V).$
\end{lemma}

\begin{proof}
It suffices to show that, for any element
$u \in S(\fv)$ of order $d,$ the differentiation $\partial_u$ maps $\cM_\mu^k \cO(\Omega, V)$
into $\cM_\mu^{k - d} \cO(\Omega, V).$ Clearly, it suffices to do this for $u = X \in \fv$ and $d = 1.$
In this case the result follows from the Leibniz rule
$$
(\partial_u \hatotimes \id_V) ( \gf F ) = (\partial_u \gf) F + \gf (\partial_u \hatotimes \id_V)(F),
$$
for all $\gf \in \cO(\Omega)$ and $F \in \cO(\Omega, V).$
\end{proof}

\begin{proof}[Proof of Proposition \ref{p: property of M}]
We will first prove the identity for $f_j\in\O(\Omega)\otimes V_j$, $j=1,2$.
Let $m: \cO(\Omega) \times \cO(\Omega) \to \cO(\Omega)$ denote the multiplication map.
Then $\beta_*(f_1,f_2)= (m \otimes \gb)(f_1, f_2)\in \O(\Omega)\otimes V_3 .$
Moreover, since $M:\O(\Omega)\rightarrow\O(\Omega,A)$
is an algebra homomorphism, it follows that
$$
\begin{array}{rcl}
M\beta_*(f_1,f_2) & = & (M\otimes\id_{V_3})\circ(m\otimes\beta)(f_1,f_2)\\
& = & (m^A\otimes\beta)_*([M\otimes\id_{V_1}]f_1,[M\otimes\id_{V_2}]f_2)\\
& = & (m^A\otimes\beta)_*(Mf_1,Mf_2).
\end{array}
$$
We will now establish the identity for general $f_j \in \cO(\Omega,  V_j).$
Clearly, it suffices to prove the identity of holomorphic functions at
a fixed point $\mu$ of $\Omega.$ Moreover, since the operations $(f_1, f_2) \mapsto \gb_*(f_1, f_2)$
and $M$ commute with restriction, we may assume that $\Omega$ is a polydisk centered at $\mu.$
Let $d$ be as in Lemma \ref{l: M and fm mu}. Fix $k > d.$
                    \par
From the power series expansions of $f_1, f_2$ at the point $\mu$ it follows
that there exist polynomials $p_j \in P(\fv) \otimes V_j$ and holomorphic functions
$r_j \in \cM_\mu^k \cO(\Omega, V_j)$ such that
$$
f_j = p_j + r_j, \quad j =1,2.
$$
From the $\cO(\Omega)$-bilinearity of $\gb_*$ it readily follows that
\begin{equation}
\label{e: beta modulo}
\beta_*(f_1,f_2) = \beta_*(p_1, p_2) + R,
\end{equation}
with $R \in \cM_\mu^k\cO(\Omega, V_3).$ From the $\cO(\Omega)$-bilinearity of $\gb_*$
combined with Lemma \ref{l: M and fm mu} it follows that
\begin{equation}
\label{e: m modulo}
(m^A\otimes\beta)_*(M f_1,M f_2) = (m^A\otimes \beta)_* (M p_1, M p_2) + \rho,
\end{equation}
with $\rho \in \cM^{k - d}_\mu \cO(\Omega, V_3 \otimes A).$
Evaluation at the point $\mu$ gives $\ev_\mu(R) = R(\mu) = 0$ and $\ev_\mu(\rho) = \rho(\mu) = 0.$
Combining this with (\ref{e: beta modulo}), (\ref{e: m modulo}) and the first part of the proof applied to $(p_1,p_2)$ we find
\begin{eqnarray*}
\ev_\mu M \beta_*(f_1,f_2) & = & \ev_\mu M \beta_*(p_1, p_2)\\
&=& \ev_\mu (m^A\otimes\beta)_*(M p_1,M p_2)\\
&=& \ev_\mu (m^A\otimes\beta)_*(M f_1,M f_2).
\end{eqnarray*}
This establishes the desired identity at the point $\mu.$
\end{proof}

\begin{cor}\label{sec47}
${}$
Let $V_1,V_2,V_3$ and $\gb, \Omega$ be as Proposition \ref{p: property of M}.
\begin{enumerate}
\itema For any cofinite ideal $\cI$ of $\O_0$,
$$
J_\cI\circ\beta_* = (m^{\O_0/\cI}\otimes\beta)_*\circ (J_\cI, J_\cI )
$$
on $\cO(\Omega, V_1) \times \cO(\Omega, V_2).$
\itemb For any finite dimensional $\O_0$-module $E$,
$$
\diffD^{(E)}\circ\beta_*=(m^{\End{E}}\otimes\beta)_*\circ (\diffD^{(E)},\diffD^{(E)} )
$$
on $\cO(\Omega, V_1) \times \cO(\Omega, V_2).$
\end{enumerate}
\end{cor}

\begin{proof}
In view of Remark \ref{r: J is M}, this follows from Proposition \ref{p: property of M}.
\end{proof}

\begin{cor}
\label{c: multiplicativity of diffD}
Let $V_1, V_2, V_3$ be quasi-complete locally convex Hausdorff spaces, and assume that $V_1, V_2$ are barrelled.
Let $S \in \cO(\Omega, \Hom(V_1, V_2))$ and $T \in \cO(\Omega, \Hom(V_2, V_3)).$
Then  the map $TS: \mu \mapsto T(\mu)S(\mu)$ belongs to $\cO(\Omega, \Hom(V_1, V_3)).$ Moreover,  the following holds.
\begin{enumerate}
\itema For any cofinite ideal $\cI$ of $\O_0$,
$$
J_\cI(TS)  = J_\cI(T) J_\cI(S).
$$
\itemb
For any finite dimensional $\O_0$-module $E$,
$$
(TS)^{(E)} =  T^{(E)} S^{(E)}.
$$
\end{enumerate}
\end{cor}

The expression on the right-hand side of (a) should be read as the natural pointwise product
of the functions $J_\cI S \in \cO(\Omega, \cO_0/\cI \otimes \Hom(V_1, V_2))$ and $J_\cI T \in
\cO(\Omega, \cO_0/\cI \otimes \Hom(V_2, V_3)).$
                 \par
For (b) we note that $S^{(E)}$ is holomorphic on $\Omega$ with values in
$\Endo(E) \otimes \Hom(V_1, V_2) \simeq \Hom(E\otimes V_1, E \otimes V_2).$
Likewise, $T^{(E)}$ is a $\Hom(E\otimes V_2, E \otimes V_3)$-valued holomorphic function
on $\Omega.$ Accordingly, the expression on the right-hand side of (b) should be read as
the pointwise composition.

\begin{proof}
This follows from the previous corollary combined with Lemma \ref{l: composition of holomorphic homomorphisms}.
\end{proof}

\bibliographystyle{amsplain}
\bibliography{references}
\end{document}